\documentclass[11pt]{article}
\usepackage{url}
\usepackage{enumerate}
\usepackage{amsmath}%
\usepackage{stmaryrd}
\usepackage{amsxtra}%
\usepackage{amsfonts}%
\usepackage{amssymb}%
\usepackage[ruled]{algorithm2e}
\usepackage{amsthm}
\usepackage{graphicx}
\usepackage{caption}
\usepackage{subfigure}
\usepackage{color,hyperref}
\definecolor{darkblue}{rgb}{0.7,0.0,0.0}
\hypersetup{colorlinks,breaklinks,
             linkcolor=darkblue,urlcolor=darkblue,
            anchorcolor=darkblue,citecolor=darkblue}

\usepackage[margin=3cm]{geometry}

\DeclareMathOperator*{\argmax}{argmax}

\newcommand{\iid}{\emph{i.i.d.}}
\newcommand{\dist}{\text{dist}}
\newcommand{\A}{{\mathcal A}}

\newcommand{\G}{{\mathcal G}}
\renewcommand{\H}{{\mathcal H}}
\renewcommand{\S}{{\mathbb S}}
\newcommand{\vb}[1]{\mathbf{#1}}
\renewcommand{\v}{\vb{v}}
\renewcommand{\O}{{\mathcal O}}
\renewcommand{\bar}[1]{{\overline{#1}}}
\newcommand{\R}{\mathbb{R}}
\newcommand{\N}{\mathbb{N}}
\newcommand{\Z}{\mathbb{Z}}

\newcommand{\eps}{\varepsilon}
\renewcommand{\P}{\mathbb{P}}
\newcommand{\E}{\mathbb{E}}
\renewcommand{\hat}[1]{\widehat{#1}}
\renewcommand{\tilde}[1]{\widetilde{#1}}
\renewcommand{\phi}{\varphi}
\newcommand{\beq}{\begin{equation}}
\newcommand{\eeq}{\end{equation}}

\theoremstyle{plain}
\newtheorem{theorem}{Theorem}
\newtheorem{proposition}{Proposition}
\newtheorem{corollary}{Corollary}
\newtheorem{lemma}{Lemma}

\theoremstyle{definition}
\newtheorem{definition}{Definition}

\newtheorem*{acknowledgements}{Acknowledgements}
\theoremstyle{remark}
\newtheorem{remark}{Remark}

\numberwithin{equation}{section}
\begin{document}
\title{Directed last passage percolation with discontinuous weights\thanks{The research in this paper  was partially supported by NSF grant DMS-0914567.}}
\author{Jeff Calder\thanks{Department of Mathematics, University of California, Berkeley. Email: {\tt jcalder@berkeley.edu}}}


\date{December 12, 2014}

\maketitle

\begin{abstract}
We prove that a directed last passage percolation model with discontinuous macroscopic (non-random) inhomogeneities has a continuum limit that corresponds to solving a Hamilton-Jacobi equation in the viscosity sense. This Hamilton-Jacobi equation is closely related to the conservation law for the hydrodynamic limit of the totally asymmetric simple exclusion process.  We also prove convergence of a numerical scheme for the Hamilton-Jacobi equation and present an algorithm based on dynamic programming for finding the asymptotic shapes of maximal directed paths.  
\end{abstract}

\section{Introduction}
\label{sec:intro_ch4}

The directed last passage percolation (DLPP) problem can be formulated as follows: Let $X(i,j)$ be nonnegative independent random variables defined on the lattice $\N^2$, and define the last passage time from $(1,1)$ to $(M,N)$  by
\begin{equation}\label{eq:dlpp-intro}
L(M,N) = \max_{p \in \Pi_{M,N}} \sum_{(i,j) \in p} X(i,j),
\end{equation}
where $\Pi_{M,N}$ denotes the set of up/right paths from $(1,1)$ to $(M,N)$ in $\N^2$.  Of interest are the asymptotics of $L$ as $M,N \to \infty$, and their first order fluctuations.  

DLPP is an example of a stochastic growth model, and has  many applications in mathematical and scientific contexts.  For example, DLPP is equivalent to zero-temperature directed polymer growth in a random environment---an important model in statistical mechanics~\cite{comets2004,huse1985pinning,imbrie1988diffusion,bolthausen1989note}.  The model describes a hydrophilic polymer chain wafting in a water solution containing randomly placed hydrophobic molecules (impurities) that repel the individual monomers in the polymer chain.  Due to thermal fluctuations and the random positions of impurities, the shape of the polymer chain is best understood as a random object.  The statistical mechanical model for a directed polymer assumes that the shape of the polymer can be described by a directed path $p \in \Pi_{M,N}$, thus suppressing entanglement and U-turns.  The presence, or strength, of an impurity at site $(i,j)$ is described by a random variable $X(i,j)$, and the energy of a path $p\in \Pi_{M,N}$ is given by
\begin{equation}
  -\beta\sum_{(i,j) \in p} X(i,j),
  \label{eq:dir-polymer}
\end{equation}
where $\beta=1/T>0$ is the inverse temperature.  The typical shape of a polymer is one that minimizes \eqref{eq:dir-polymer}.  Of interest is the quenched polymer distribution on paths defined by
\begin{equation}
  Q(p;M,N) = \frac{1}{Z(M,N)}\exp\left( \beta \sum_{(i,j) \in p} X(i,j) \right),
  \label{eq:quenched-polymer}
\end{equation}
where $p \in \Pi_{M,N}$ and the normalization factor $Z(M,N)$ is called the \emph{partition function}, and is given by
\begin{equation}
  Z(M,N) = \sum_{p \in \Pi_{M,N}} \exp\left( \beta \sum_{(i,j) \in p} X(i,j) \right).
  \label{eq:partition-function}
\end{equation}
In the zero-temperature limit, i.e., $\beta\to \infty$, the quenched polymer distribution concentrates around paths maximizing \eqref{eq:dir-polymer}, and we formally have
\[
  \lim_{\beta \to \infty} \frac{1}{\beta} \log\left( Z(M,N)\right) = \max_{p \in \Pi_{M,N}} \sum_{(i,j) \in p} X(i,j) = L(M,N),
\]
  Directed polymers are related to several other stochastic models for growing surfaces, such as directed invasion percolation, ballistic deposition, polynuclear growth, and low temperature Ising models~\cite{krug1991}.  

  DLPP with independent and identically distributed (\iid) exponential weights $X(i,j)$ is equivalent to the totally asymmetric simple exclusion process (TASEP), which is an important stochastic interacting particle system~\cite{ferrari2006,rost1981non}, and to randomly growing Young diagrams~\cite{johansson2000shape,vershik1995asymptotic,seppalainen1996hydrodynamic}.  Briefly, the dynamics of TASEP involve a particle configuration on the lattice $\Z$, evolving in time, with the dynamical rule that a particle jumps to the right after an exponential waiting time if the right neighboring site is empty.    The correspondence between DLPP and TASEP proceeds via the following stochastic corner growth model: Partition $\R^2$ into squares defined by the edges of the lattice $\Z^2$.  Imagine that at time $t=0$, all the squares in $[0,\infty)^2$ are colored white, while the remaining squares are colored black. For each $(i,j) \in \N^2$, assign a passage time random variable $X(i,j)$ to the square with $(i,j)$ on the northeast corner. The dynamic rule governing the growth process is the following: A white square at location $(i,j)$ is colored black exactly $X(i,j)$ time units after both its south and west neighbors become black.  The time until square $(M,N)$ is colored black is exactly $L(M,N)$---the last passage time from $(1,1)$ to $(M,N)$---and the set of all black squares is a randomly growing Young diagram.  
  
    There is a one-to-one correspondence between TASEP configurations, and configurations of black and white squares in the corner growth model. The idea is that when a white square is colored black, it corresponds to a particle jumping from a site $j$ to its necessarily vacant neighbor $j+1$.   The explicit correspondence is as follows: For every edge separating a white and black square, assign a value of 1 to vertical edges, and a value of 0 to horizontal edges.  The TASEP configuration corresponds exactly to reading these binary values sequentially from $(1,\infty)$ to $(\infty,1)$. We give this correspondence more rigorously in Section \ref{sec:formal-eq} (see Figure \ref{fig:tasep}). There are further applications of DLPP in queueing theory~\cite{baccelli2000asymptotic,glynn1991}, and the model is also related to greedy lattice animals~\cite{martin2002linear}. 

One quantity of interest in DLPP is the time constant, $U$, given by
\begin{equation}
  U(x) :=\lim_{N\to \infty} \frac{1}{N}L\left(\lfloor Nx\rfloor\right),
  \label{eq:time-const}
\end{equation}
where $x = (x_1,x_2) \in [0,\infty)^2$.
The exact form of $U$ is known for \iid~geometric weights~\cite{johansson2000shape}, and \iid~exponential weights~\cite{rost1981non}, and is given by
\begin{equation}
  U(x) = \mu(x_1 + x_2) + 2\sigma \sqrt{x_1x_2},
  \label{eq:time-const-known}
\end{equation}
where $\mu$ and $\sigma^2$ are the mean and variance, respectively, of the either geometric or exponential weights.
For more general distributions, Martin~\cite{martin2004limiting} showed that $U$ is continuous on $[0,\infty)^2$ and gave the following asymptotics at the boundary:
\[U(1,\alpha) = \mu + 2\sigma\sqrt{\alpha} + o(\sqrt{\alpha}).\]
In similar fashion to the longest increasing subsequence problem~\cite{baik1999distribution}, the fluctuations of $L$ for geometric and exponential weights  are non-Gaussian, and instead follow the Tracy-Widom distribution asymptotically~\cite{johansson2000shape}.  It is an open problem to determine $U(x)$ and the fluctuations of $L$ for weights other than geometric and exponential.

We study the DLPP problem with independent weights $X(i,j)$ that are either geometric or exponential, but not identically distributed.    For exponential DLPP, we assume that $X(i,j)$ is exponentially distributed with mean $\lambda(iN^{-1},jN^{-1})$ where $\lambda:[0,\infty)^2 \to [0,\infty)$, and we consider the aymptotics as $N\to \infty$.  The setup is identical for geometric DLPP, except that the macroscopic inhomogeneity is in the parameter $q$ of the geometric distribution.  For directed polymers, this models a macroscopic (non-random) inhomogeneity in the strength of impurities; while for TASEP, it corresponds to an inhomogeneity in the rate at which particles move to the right.    Our main result, presented in Section \ref{sec:results}, is a Hamilton-Jacobi equation for the continuum limit of this DLPP problem.
  
In the exponential case with continuous $\lambda$, Rolla and Teixeira~\cite{rolla2008} showed that $U$ has a variational interpretation.  Their result is in many ways analogous to the variational problem for the longest chain problem~\cite{deuschel1995} that we exploited in our previous work~\cite{calder2014,calder2013b}.  Macroscopic inhomogeneities have also been considered for TASEP~\cite{georgiou2010}, and for other similar growth models~\cite{rezakhanlou2002continuum}.  In particular, Georgiou et al.~\cite{georgiou2010} proved a hydrodynamic limit for TASEP with a spatially (but not temporally) inhomogeneous jump rate $c$, which may admit discontinuities.  Their result gives the limiting density profile in terms of a variational problem, and they connected this to a conservation law in the special case that the rate $c(s)$ is piecewise constant with one jump, i.e.,
\[c(s) = \begin{cases}
c_1,& s \leq 0\\
c_2,& s > 0.
\end{cases}\]
In the context of exponential DLPP, this would be equivalent to assuming that the macroscopic mean $\lambda:[0,\infty)^2\to [0,\infty)$ is given by $\lambda(x)=c_1^{-1}$ for $x_1 \geq x_2$ and $\lambda(x)=c_2^{-1}$ otherwise.  Our main result, Theorem \ref{thm:main}, gives a Hamilton-Jacobi equation for the limiting time constant in DLPP when the macroscopic inhomogeneity $\lambda$ is piecewise Lipschitz.  In the context of TASEP, this allows for a discontinuous inhomogeneous jump rate which has a spatial \emph{and} temporal dependence.

\subsection{Main result}
\label{sec:results}

Let us mention the conventions used in this paper.  We say $X$ is geometrically distributed with parameter $q$ if 
\[\P(X = k) = (1-q)^kq,\]
for $k\in \{0,1,2,3,\dots\}$ and $0 < q \leq 1$, so that we have
\begin{equation}\label{eq:geom-mom}
\E(X) = \frac{1-q}{q} \ \ \text{and} \ \ \text{Var}(X) = \frac{1-q}{q^2}.
\end{equation}
We say that $X$ is exponentially distributed with mean $\lambda\geq 0$ if for $\lambda>0$ we have
\[\P(X \in dx) = \frac{1}{\lambda} e^{-\frac{x}{\lambda}} dx \ \ \ \text{for } x \in [0,\infty),\]
and when $\lambda=0$ we have $X=0$ with probability one.  Here we have
\begin{equation}\label{eq:exp-mom}
\E(X) = \lambda \ \ \text{and} \ \ \text{Var}(X) = \lambda^{2}.
\end{equation}
In order to ensure that our results are applicable to both exponential and geometric DLPP, we parameterize these distributions instead by their mean $\mu$.  For the exponential distribution there is no change; we have $\lambda = \mu$.  For the geometric distribution, we have by \eqref{eq:geom-mom} that a geometric random variable with mean $\mu \geq 0$ has parameter
\begin{equation}
  q = \frac{1}{1 + \mu}.
  \label{eq:qmu}
\end{equation}
For both cases, the variance is of course a function of the mean; in the exponential case we have $\sigma =\mu$, and in the geometric case we have $\sigma = \sqrt{\mu(1+\mu)}$.  

Let us now present our main result.  We consider the following two-sided DLPP model, similar to~\cite{baik2005phase,arous2011current,corwin2010limit,baik2000limiting,borodin2008airy}. Let $X(i,j)$ be independent nonnegative random variables defined on the lattice $\N_0^2$, where $\N_0=\{0,1,2,\dots\}$.  Let $L(M,N;Q,P)$ denote the last passage time from $(M,N)\in \N^2_0$ to $(Q,P)\in \N^2_0$, where $M \leq Q$ and $N \leq P$.  This is defined as follows:
\begin{equation}\label{eq:dlpp}
L(M,N;Q,P) = \max_{p \in \Pi_{(M,N),(Q,P)}} \sum_{(i,j) \in p} X(i,j),
\end{equation}
where $\Pi_{(M,N),(Q,P)}$ denotes the set of up/right paths from $(M,N)$ to $(Q,P)$ in $\N_0^2$. The macroscopic inhomogeneity is described by functions $\mu:[0,\infty)^2 \to [0,\infty)$ and $\mu_s:\partial \R^2_+ \to [0,\infty)$, where $\R_+=(0,\infty)$.  Specifically, given a parameter $N$ we make the following assumption:
  \begin{align}
    &\text{The weights } X(i,j) \text{ are independent with mean } \notag \\
    &\hspace{0.1\textwidth}\E(X(i,j)) = \begin{cases}
      \mu(iN^{-1},jN^{-1}),& \text{if } (i,j) \in \N^2,\\
      \mu(iN^{-1},jN^{-1}) + \mu_s(iN^{-1},jN^{-1}),& \text{if } i=0 \text{ or } j=0.
    \end{cases}\label{eq:weights-assumption}
 \end{align}
  The term $\mu$ corresponds to the macroscopic mean within the bulk $\R^2_+$, and the term $\mu_s$ corresponds to an additional source active only on the boundary $\partial \R^2_+$. 
  
We also assume the weights $X(i,j)$ are either all geometrically distributed, or all expontially distributed. We can construct the random variables $X(i,j)$ on a common probability space as follows. Let $Y(i,j)$ be \iid~exponential random variables with mean $\lambda=1$, where $i,j \in \N_0$. In the exponential case, we can simply set 
  \[X(i,j) = \begin{cases}
  \mu(iN^{-1},jN^{-1}) Y(i,j),& \text{if }  (i,j) \in \N^2,\\
 \left(\mu(iN^{-1},jN^{-1}) + \mu_s(iN^{-1},jN^{-1})\right) Y(i,j),&\text{if } i=0 \text{ or } j=0.
 \end{cases}\]
 This setup is similar to~\cite{rolla2008}.
 In the geometric case, we note that if $Y$ is an exponential random variable with mean $\lambda=1$, then for any $\nu>0$, $X=\lfloor \nu Y\rfloor$ is geometrically distributed with parameter $q = 1 - e^{-\frac{1}{\nu}}$. In order to obtain $\E(X) = \mu>0$, we need that
 \[\frac{1}{1+\mu} = q = 1 - e^{-\frac{1}{\nu}},\]
 which gives that $\nu = 1/(\log(1+\mu) - \log(\mu))$. If $\mu=0$, then we set $\nu =0$. Hence, let us set $\nu(x)= 1/(\log(1+\mu(x)) - \log(\mu(x)))$ for $\mu(x)>0$ and $\nu(x) = 0$ when $\mu(x) = 0$. We make a similar definition for $\nu_s$. Setting
  \[X(i,j) = \begin{cases}
 \left\lfloor \nu(iN^{-1},jN^{-1}) Y(i,j)\right\rfloor,& \text{if }  (i,j) \in \N^2,\\
\left\lfloor\left(\nu(iN^{-1},jN^{-1}) + \nu_s(iN^{-1},jN^{-1})\right) Y(i,j)\right\rfloor,&\text{if } i=0 \text{ or } j=0,
 \end{cases}\]
 we see that $X(i,j)$ are independent geometric random variables satisfying \eqref{eq:weights-assumption}. 

 Before stating the, somewhat technical, hypotheses on $\mu$ and $\mu_s$, we need to introduce some notation. We say a curve $\Gamma$ in $\R^2$ is continuous and strictly increasing if it can be parameterized in the form
\[\Gamma: t \mapsto (t,f(t)) \ \ \text{for } t \in I,\]
where $f:I \to \R$ is continuous and strictly increasing, and $I$ is an interval in $\R$.    We make a similar definition for strictly decreasing.  Notice that a continuous strictly increasing (resp.~decreasing) curve can also be parameterized in the form $t \mapsto (f(t),t)$ where $f:I\to\R$ is continuous and strictly increasing (resp.~decreasing). For simplicity, we will also use $\Gamma$ to denote the locus of points that lie on the curve $\Gamma$.

Let $\Gamma$ be a continuous strictly decreasing curve in $[0,1]^2$ with endpoints $(1,0)$ and $(0,1)$, and let $\Omega\subset [0,\infty)^2$ denote the bounded component of the complement of $\Gamma$ in $[0,\infty)^2$. Let $\{\Gamma_i\}_{i \in \Z}$ be a locally finite non-intersecting collection of continuous strictly increasing curves. For each $i$ we assume one endpoint of $\Gamma_i$ is on $\partial ([0,\infty)^2 \setminus \Omega)$ and the other endpoint is at $\infty$, i.e., the curve is unbounded.    The complement of $\cup_i \Gamma_i$ in $[0,\infty)^2\setminus \Omega$ therefore consists of a family $\{\Omega_i\}_{i \in \Z}$ of connected components. Each curve $\Gamma_i$ is on the boundary of exactly two components, which we may assume are labeled $\Omega_i$ and $\Omega_{i-1}$.   See Figure \ref{fig:domain} for an illustration of these quantities.  
\begin{figure}
\centering
\includegraphics[width=0.6\textwidth]{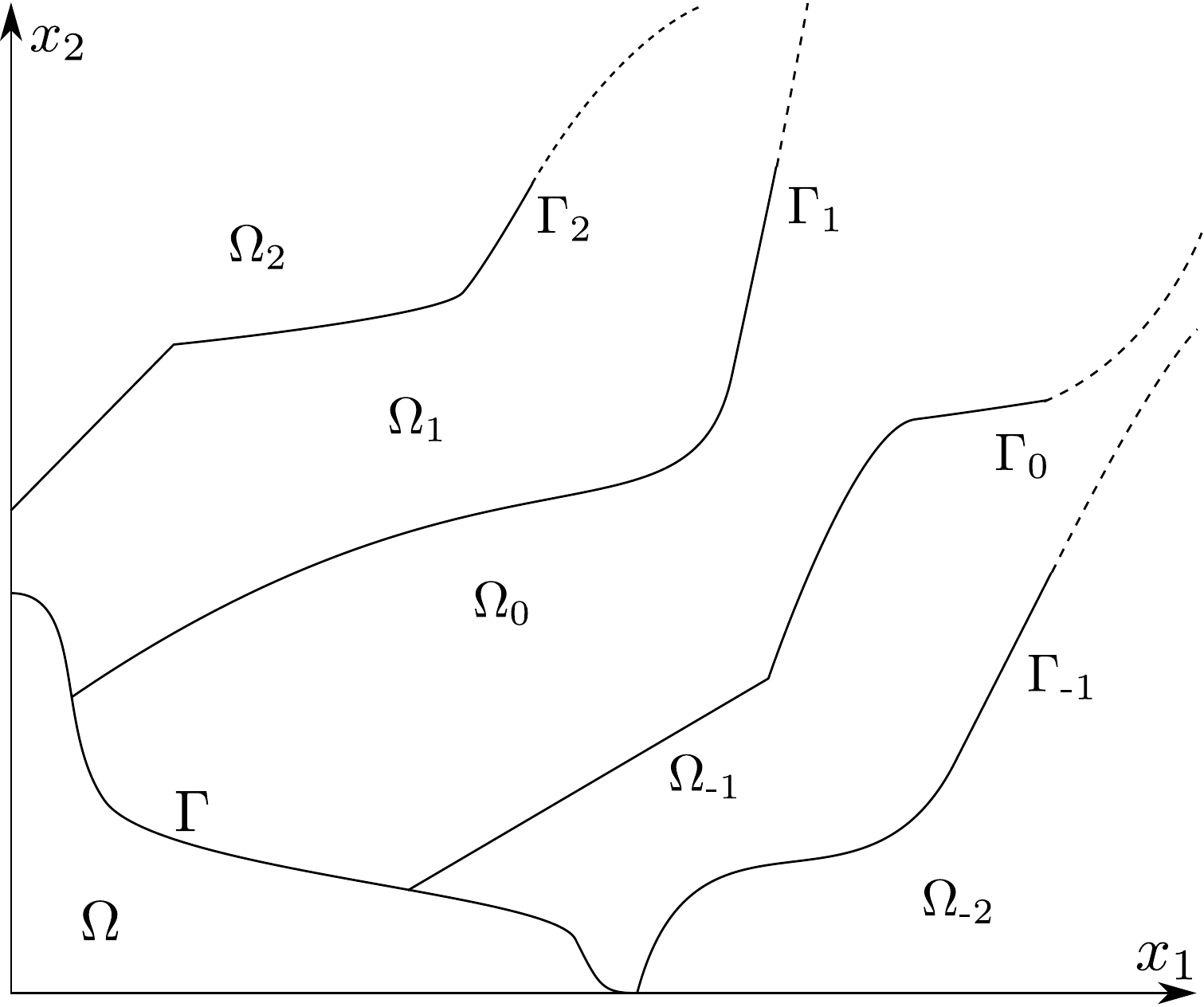}
\caption{Depiction of quantities $\Omega_i$ and $\Gamma_i$.  The function $\mu$ is assumed to be Lipschitz with constant $C_{lip}$ when restricted to any $\Omega_i$, and $\mu=0$ on $\Omega$.}
\label{fig:domain}
\end{figure}

We place the following assumptions on $\mu$ and $\mu_s$:
\begin{itemize}
\item[(F1)] The function $\mu:[0,\infty)^2\to [0,\infty)$ is bounded and upper semicontinuous, $\mu\vert_\Omega = 0$, and there exists a constant $C_{lip}$ such that for every $i \in \Z$, $\mu\vert_{\Omega_i}$ is Lipschitz continuous with constant $C_{lip}$.
\item[(F2)] The source term $\mu_s:\partial \R^2_+ \to [0,\infty)$ is bounded and upper semicontinuous with a locally finite set of discontinuities.
\end{itemize}
Throughout the paper we will regard $\mu_s$ as a function on $[0,\infty)^2$ by setting $\mu_s=0$ on $\R^2_+$.    We also make the following technical assumption: 
\begin{itemize}
\item[(F3)] For every  $i \in \Z$ and $x \in \Gamma_i$, there exists $\eps > 0$ and $\zeta\in \{-1,+1\}$ such that for all $y \in B_\eps(x)\cap \Gamma_i$ we have 
\begin{equation}
\zeta \left(\lim_{\Omega_{i-1} \ni z \to y} \mu(z) - \lim_{\Omega_i \ni z \to y} \mu(z)\right) \geq 0.
\end{equation}
\end{itemize}

Our main result is the following continuum limit:
\begin{theorem}\label{thm:main}
  Let $\mu:[0,\infty)^2 \to [0,\infty)$ satisfy (F1) and (F3), and let $\mu_s:\partial \R^2_+ \to [0,\infty)$ satisfy  (F2).   Suppose that the weights $X(i,j)$ satisfy \eqref{eq:weights-assumption} and are either all exponential, or all geometric random variables, constructed on a common probability space as above. In the exponential case, set $\sigma = \mu$, and in the geometric case, set $\sigma = \sqrt{\mu(1+\mu)}$.  Then with probability one we have
\begin{equation}\label{eq:main-limit}
\frac{1}{N} L(0;\lfloor N\cdot\rfloor) \longrightarrow U \ \ \text{locally uniformly on } [0,\infty)^2,
\end{equation}
where $U$ is the unique monotone viscosity solution of 
\[\text{\emph{(P)}}\left\{\begin{aligned}
(U_{x_1} - \mu)_+ (U_{x_2} - \mu)_+ &= {\sigma}^2& &\text{on } \R^2_+,\\
U &= \phi& &\text{on } \partial \R^2_+,
\end{aligned}\right.\]
and $\phi(x) = (x_1 + x_2)\int_0^1 \mu(tx) + \mu_s(tx) \, dt$.
\end{theorem}
Here, $U_{x_1}$ and $U_{x_2}$ denote the partial derivatives of $U$, $t_+$ denotes the positive part of $t$ given by $\max(t,0)$, and by monotone we mean that $U$ is monotone non-decreasing with respect to all variables. 

Theorem \ref{thm:main} is an extension of our previous work~\cite{calder2014,calder2013b}, in which we proved a similar result for the longest chain problem. This result can be viewed as a type of stochastic homogenization~\cite{souganidis1999}, where the effective Hamiltonian is given in  (P).  A similar stochastic homogenization result has been obtained recently for first passage percolation~\cite{krishnan2013}, though in that case the exact form of the effective Hamiltonian is unknown.  The Hamilton-Jacobi equation (P) is also closely related to the conservation law for the hydrodynamic limit of TASEP~\cite{georgiou2010}, and in Section \ref{sec:formal-eq} we show a formal equivalence between the two continuum limits. 

We believe this new Hamilton-Jacobi equation will prove to be a useful tool for studying the DLPP problem, both theoretically and numerically.  As an example, in Section \ref{sec:max-curve} we show how to combine the numerical solution of this Hamilton-Jacobi equation with dynamic programming to find the asymptotic shapes of optimal paths.  We also believe that this work will provide a new perspective on the hydrodynamic limit of TASEP, and may be useful for studying the corresponding conservation law.  

Some remarks on the hypotheses (F1), (F2), and (F3) are in order. First, the assumption that $\mu$ and $\mu_s$ are bounded in (F1) and (F2) is made for simplicity. It can be replaced by the assumption that $\mu$ and $\mu_s$ are bounded on compact sets, with minor changes to the proofs. Recall that in the exponential case, we have $\sigma^2 = \mu^2$, and in the geometric case, we have $\sigma^2 = \mu(1+\mu)$.  Thus, if $\mu$ satisfies (F1), (F3), then so will $\sigma^2$, though possibly with a larger Lipschitz constant $C_{lip}$.  Since it is convenient for the analysis, we will often regard $\mu$ and $\sigma^2$ as independent functions both satisfying (F1) and (F3). We will only need to recall the relationship between $\mu$ and $\sigma$ at a few key points.  In particular, the  uniqueness proof for (P) (see Section \ref{sec:comparison}) requires that $\mu$ and $\sigma^2$ satisfy (F3) simultaneously with the same choice of $\zeta$.  This is of course always true, since $\sigma$ is a monotone increasing function of $\mu$ in both the exponential and geometric cases. 

Let us briefly comment on the significance of $\Gamma$ and $\Omega$.  The correspondence between exponential DLPP and TASEP (described in detail in  Section \ref{sec:formal-eq}) implies that the initial macroscopic density $\rho_0$ for TASEP is encoded into the curve $\Gamma$.    If $\Gamma$ and $\Omega$ are not present, then we have TASEP with the common step initial condition $\rho_0(s)=1$ for $s \leq 0$ and $\rho_0(s)=0$ for $s>0$.  Suppose now that $\Gamma$ and $\Omega$ are present, and parameterize $\Gamma$ by $t \mapsto (t,f(t))$ where $f$ is continuous and strictly decreasing with $f(0)=1$ and $f(1)=0$.  Let us assume additionally that $f$ is continuously differentiable.  Based on the correspondence between TASEP and exponential DLPP, the initial density will be given by
\[\rho_0(s) = \begin{cases}
    1,& \text{if } s \leq -1 \\
    -f'(t_s)/(1-f'(t_s)),& \text{if } s \in (-1,1) \\
    0,& \text{if } s \geq 1.
\end{cases}\]
where for $s \in (-1,1)$, $t_s$ is the unique $t \in (0,1)$ satisfying $s=t-f(t)$.  Thus by choosing $f$ appropriately, one can obtain a large class of initial densities $\rho_0$ for TASEP with this setup.

The rest of the paper is organized as follows: In Section \ref{sec:formal-eq} we show formally that (P) is equivalent to the conservation law for the hydrodynamic limit of TASEP~\cite{georgiou2010}.  The proof of Theorem \ref{thm:main} is given in Section \ref{sec:conv} after some preliminary results. In particular, in Section \ref{sec:var} we present and analyze a variational problem for (P), and in Section \ref{sec:comparison}, we prove a comparison principle for (P), which generalizes our previous work~\cite{calder2014}.  In Section \ref{sec:num}, we present a fast numerical scheme for computing the viscosity solution of (P), and we present the results of various numerical simulations in Section \ref{sec:sim}. Finally, in Section \ref{sec:max-curve}, we give an algorithm based on dynamic programming for finding the asymptotic shape of optimal DLPP paths, and in Section \ref{sec:discussion} we discuss possible directions for future work.

\subsection{Formal equivalence to hydrodynamic limit of TASEP}
\label{sec:formal-eq}

We show here a formal equivalence between (P) and the hydrodynamic limit of TASEP, given in~\cite{georgiou2010}. TASEP is an interacting stochastic particle system on $\Z$ with state space $\{0,1\}^\Z$, whose elements, $\eta$, represent particle configurations.  If a particle is present at site $j \in \Z$, then  $\eta_j=1$, and if no particle is present, then $\eta_j=0$.  The process is exclusionary in the sense that at most one particle can occupy each site at a given time.  The stochastic dynamics proceed as follows: a particle at site $j$ jumps to site $j+1$ after an exponential waiting time, provided the site $j+1$ is empty.  The exponential waiting times are independent and begin at  the exact moment the right neighboring site is vacated.  These dynamics, along with an initial condition $\eta(0):\Z \to \{0,1\}$, generate the stochastic process $\eta = \{\eta_i(t) \, : \, i \in \Z,\, t \in [0,\infty)\}$.  

In the standard TASEP model, the exponential waiting times are independent with rate $c=1$.  As in~\cite{georgiou2010},  we allow the rates to have a macroscopic spatial (and temporal) dependence, i.e., the rate at position $j\in \Z$ and time $t \in [0,\infty)$ is $c(jN^{-1},tN^{-1})$, where $c: \R\times [0,\infty) \to (0,\infty)$, and $N$ is a parameter that we will send to $\infty$.  A central object of study is the macroscopic density $\rho(s,t)$, which is the almost sure limit (assuming it exists) of the discrete densities as follows:
\begin{equation}\label{eq:density-limit}
\lim_{N \to \infty} \frac{1}{N}\sum_{i=\lfloor Na\rfloor + 1}^{\lfloor Nb\rfloor} \eta_i(Nt) = \int_a^b \rho(s,t) \, ds.
\end{equation}
Georgiou et al.~\cite{georgiou2010} showed that for
\[c(s,t) = c(s) = \begin{cases}
c_1,& s \leq 0\\
c_2,& s > 0,
\end{cases}\]
$\rho$ can be identified as the unique entropy solution of the scalar conservation law
\begin{equation}\label{eq:conservation-law}
\rho_t + (c(s) \rho (1-\rho))_s = 0, \ \ \rho(s,0) = \rho_0(s),
\end{equation}
where $\rho_0$ denotes the initial macroscopic density.  We are using $s$ for the spatial variable in \eqref{eq:conservation-law} to avoid confusion with the spatial variables in (P). In what follows, we show formally that the conservation law \eqref{eq:conservation-law} is equivalent to (P). For simplicity, we will  ignore the initial condition $\rho_0$ and the boundary condition in (P), and restrict ourselves to showing that the (P) and \eqref{eq:conservation-law} are equivalent in the bulk. We shall also assume that $\rho \in C^1$.

Consider now the exponential DLPP model with macroscopic mean $\lambda : [0,\infty)^2 \to (0,\infty)$, i.e.,  $\mu=\sigma=\lambda$. Let $L$ denote the last passage time given by \eqref{eq:dlpp}, and let us write $L(m,n) = L(1,1;m,n)$ for convenience. Let $U$ be the unique monotone viscosity solution of (P), and let us assume that $U \in C^1$ and $\lambda > 0$ so that $U_{x_1},U_{x_2}>\lambda > 0$. Of course, the viscosity solution of a Hamilton-Jacobi equation is in general not $C^1$; the argument we give here is purely formal. By Theorem \ref{thm:main} we have
\begin{equation}\label{eq:dlpp-main}
  \frac{1}{N} L(\lfloor Nx\rfloor) \longrightarrow U(x) \ \ \text{with probability one.}
\end{equation}
We also note that (P) can be rearranged as follows:
\begin{equation}\label{eq:P}
\frac{U_{x_1}(x)U_{x_2}(x)}{U_{x_1}(x)+U_{x_2}(x)} =\lambda(x).
\end{equation}

Let us now describe in detail the correspondence between TASEP and DLPP, which can also be found here~\cite{prahofer2002current,baik2012convergence}.  We assign to a TASEP configuration $\eta$ the site counter
\begin{equation}\label{eq:particle-count}
I_j(t) = \text{number of particles that have jumped from site } j \text{ to site } j+1 \text{ up to time } t.
\end{equation}
and  the height function
\begin{equation}\label{eq:height}
h_j(t) = \begin{cases}
2I_0(t) + \sum_{i=1}^j \big(1-2\eta_{i}(t)\big),& j\geq 1,\\
2I_0(t),& j=0,\\
2I_0(t) + \sum_{i=j+1}^0 \big(1-2\eta_i(t)\big),& j\leq -1.
\end{cases}
\end{equation}
Then we have $h_0(0)=0$, and $h_j(t)-h_j(0) = 2I_j(t)$.  
The height function $h_j(t)$ is a stochastically growing interface, and is related to the corner growth model described in Section \ref{sec:intro_ch4}.  Roughly speaking, the dynamical rule for the growth of $h_j(t)$ is that when a particle jumps to the right (from $j$ to $j+1$), a valley $\diagdown \diagup$ turns into a mountain $\diagup \diagdown$, and the height at site $j$ increases by $2$. See Figure \ref{fig:tasep} for reference.
\begin{figure}[t]
  \centering
  \includegraphics[width=\textwidth]{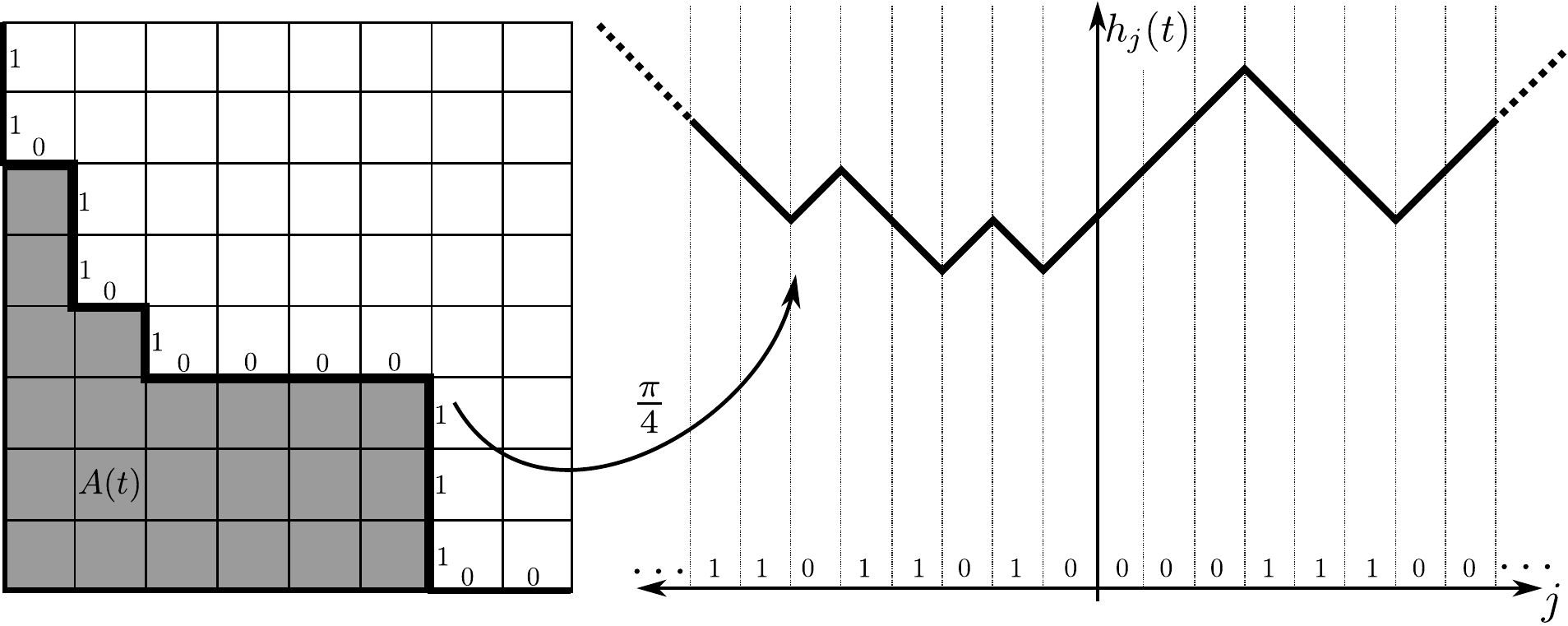}
  \caption{A visual depiction of the correspondence between TASEP and DLPP.  On the left, the gray region is the set $A(t)$---the $t$ sub-level set of $L$---and on the right we show the corresponding TASEP height function $h_j(t)$ obtained by rotating the boundary of $A(t)$ by $\pi/4$.}
  \label{fig:tasep}
\end{figure}

Let us now define the random set
\[A(t) = \Big\{ (m,n) \in \Z^2_+ \, : \, L(m,n) \leq t\Big\}.\]
Since $L$ is non-decreasing in both arguments, it implicitly defines its own height function, $\tilde{h}_j(t)$, which describes the boundary of $A(t)$ as follows:
\[A(t) = \Big\{ (m,n) \in \Z^2_+ \, : \, \tilde{h}_{m-n}(t) \geq m+n\Big\}.\]
The correspondence between TASEP and DLPP is the identification $\tilde{h}_j(t)=h_j(t)$ in the sense of joint distributions.  This connection is made rigorous by choosing appropriate boundary rates for DLPP here~\cite{prahofer2002current}.  Visually, the correspondence is obtained by rotating the boundary of $A(t)$ by $\pi/4$ to obtain the  height function $h_j(t)$ (see Figure \ref{fig:tasep}).

The correspondence between TASEP and DLPP says, at least formally, that 
\begin{equation}\label{eq:sets}
\Big\{ (m,n) \in \Z^2_+ \, : \, L(m,n) \leq t\Big\} = \Big\{ (m,n)\in \Z^2_+ \, : \, h_{m-n}(t) \geq m+n\Big\}.
\end{equation}
By \eqref{eq:density-limit} and \eqref{eq:height},  $h_j(t)$ has a macroscopic continuum limit,  $h^\infty$, such that
\begin{equation}\label{eq:tasep-main}
\frac{1}{N} h_{\lfloor sN\rfloor}(\lfloor tN\rfloor) \longrightarrow h^\infty(s,t) =g(t)  + s - 2\int_0^s \rho(s',t)\, ds',
\end{equation}
with probability one, where $g(t):=\lim_{N\to \infty} 2N^{-1} I_0(tN)$.
It follows from \eqref{eq:tasep-main} that
\begin{equation}\label{eq:tasep-h}
h^\infty_s(s,t) = 1 - 2\rho(s,t).
\end{equation}
Combining \eqref{eq:dlpp-main}, \eqref{eq:sets}, and \eqref{eq:tasep-main} we have that
\begin{equation}\label{eq:sets2}
\Big\{ x \in \R^2_+ \, : \, U(x) = t\Big\} = \Big\{ x \in \R^2_+ \, : \,  h^\infty(x_1-x_2,t) = x_1+x_2\Big\}.
\end{equation}
It follows from \eqref{eq:sets2} that
\begin{equation}\label{eq:main-eq}
h^\infty\big(x_1-x_2,U(x)\big) = x_1+x_2.
\end{equation}
This is in some sense the ``master equation'' relating the continuum limits of TASEP and DLPP.  Let us illustrate how to use \eqref{eq:main-eq} to derive the conservation law \eqref{eq:conservation-law} from (P); deriving (P) from \eqref{eq:conservation-law} follows in a similar fashion.

Differentiating \eqref{eq:main-eq} in both $x_1$ and $x_2$ we have
\begin{align}
h^\infty_s(s,t) + h^\infty_t(s,t) U_{x_1}(x) &= 1\label{eq1}\\
-h^\infty_s(s,t) + h^\infty_t(s,t) U_{x_2}(x) &= 1.\label{eq2}
\end{align}
where $t=U(x)$ and $s=x_1-x_2$.
Adding \eqref{eq1} and \eqref{eq2} we have
\begin{equation}\label{eq:dhdt}
h^\infty_t(s,t) = \frac{2}{U_{x_1}(x) + U_{x_2}(x)}.
\end{equation}
Similarly, by rearranging and dividing \eqref{eq1} by \eqref{eq2} we have
\begin{equation}\label{eq:dhds}
\frac{U_{x_1}(x)}{U_{x_2}(x)} = \frac{1-h^\infty_s(s,t)}{1+h^\infty_s(s,t)} \stackrel{\eqref{eq:tasep-h}}{=}\frac{\rho(s,t)}{1-\rho(s,t)}.
\end{equation}
This equality can also be obtained by noting that the slope of the level set $\{U(x)=t\}$ is given locally by the ratio of ones to zeros in the TASEP configuration.

Solving for $\rho$ in \eqref{eq:dhds} we have $\rho = U_{x_1}/(U_{x_1} + U_{x_2})$, which yields
\begin{equation}\label{eq:flux}
\rho(s,t)\big( 1 - \rho(s,t) \big) = \frac{U_{x_1}(x)U_{x_2}(x)}{(U_{x_1}(x) + U_{x_2}(x))^2} \stackrel{\eqref{eq:P}}{=} \frac{\lambda(x)}{U_{x_1}(x) + U_{x_2}(x)},
\end{equation}
where we invoked the Hamilton-Jacobi equation (P) in the second equality above. Since $U$ is strictly monotone increasing in both $x_1$ and $x_2$, there is a one-to-one correspondence between the coordinates $x=(x_1,x_2)$ and $(s,t)=(x_1-x_2,U(x))$. Let us write $c(s,t):=\lambda(x)^{-1}$.  Since $\lambda$ is the exponential mean, $c$ is the exponential rate for TASEP.  Then 
combining \eqref{eq:flux} with \eqref{eq:dhdt} we have
\begin{equation}\label{eq:dhdt2}
h^\infty_t(s,t) = 2c(s,t)\rho(s,t) \big(1-\rho(s,t) \big).
\end{equation}
Differentiating with respect to $s$ on both sides of \eqref{eq:dhdt2} and applying \eqref{eq:tasep-h} we have
\begin{equation}\label{eq:conservation-law2}
-2\rho_t(s,t) =  2\big( c(s,t) \rho(s,t)(1-\rho(s,t))\big)_s,
\end{equation}
which is precisely the conservation law \eqref{eq:conservation-law}.  Furthermore, by combining \eqref{eq:dhdt2} and \eqref{eq:tasep-h}, we have the following Hamilton-Jacobi equation for $h^\infty$:
\begin{equation}\label{eq:h-HJ}
  h^\infty_t(s,t) =\frac{c(s,t)}{2} \left(1-h_s^\infty(s,t)^2\right).
\end{equation}

It seems to us that this formal computation could be made precise when $\rho$ and $U$ are indeed $C^1$ functions. This is the case, for example, when $\lambda$ is constant. In the general case where $\rho$ and $U$ are not $C^1$, it may be possible to make this formal computation precise using the machinery of viscosity solutions, and we plan to investigate this in a future work.

\section{Variational problem}
\label{sec:var}

In this section we give a variational interpretation for $U$ and analyze its relevant properties. This variational problem first appeared in~\cite{rolla2008}, in a different form, for exponential DLPP with a continuous macroscopic rate $\lambda$, and is similar to the well-known variational problem for the longest chain problem~\cite{deuschel1995,calder2014,calder2013b}.  

Let us first introduce some notation.  We denote by $\leqq$ the coordinatewise partial order on $\R^d$, i.e., $x\leqq y$ if and only if $x_i\leq y_i$ for all $i$, where $x=(x_1,\dots,x_d),y \in \R^d$.  We write $x \leq y$ if $x \leqq y$ and $x\neq y$, and we write $x< y$ if $x_i < y_i$  for all $i$.  For $x,y \in \R^d$ with  $x \leqq y$, we will often use the following interval notation
\[ [x,y] = \big\{ z \in \R^d \, : \, x \leqq z \leqq y\big\},\]
and
\[ (x,y] = \big\{ z\in \R^d \, : \, x < z \leq y \big\},\]
with similar definitions for $[x,y)$ and $(x,y)$.

Let $\A$ denote the set of $C^1$ monotone curves, given by
\begin{equation}\label{eq:A}
\A = \Big\{\gamma \in C^1([0,1]; [0,\infty)^2) \, : \, \gamma'(t) \geqq 0 \ \text{for all } t \in [0,1]\Big\}.
\end{equation}
We write $\gamma(t)=(\gamma_1(t),\gamma_2(t))$ to denote the components of $\gamma$.
For $\mu,\sigma:[0,\infty)^2 \to \R$, let us define $\ell_{\mu,\sigma}:[0,\infty)^2 \times [0,\infty)^2 \to [0,\infty)$ by
\begin{equation}\label{eq:ell}
\ell_{\mu,\sigma}(x,p) = \mu(x)(p_1 + p_2) + 2\sigma(x)\sqrt{p_1p_2},
\end{equation}
and for $\gamma \in \A$ we set
\begin{equation}\label{eq:energy}
J_{\mu,\sigma}(\gamma) = \int_0^1 \ell_{\mu,\sigma}(\gamma(t),\gamma'(t)) \, dt.
\end{equation}
Notice that $\ell_{\mu,\sigma}(x,kp) = k\ell_{\mu,\sigma}(x,p)$ for any $k \geq 0$, hence $J_{\mu,\sigma}(\gamma)$ is independent of the parametrization of $\gamma$.
We finally define
\begin{equation}\label{eq:vardefU}
U_{\mu,\sigma}(x) =\sup \Big\{ J_{\mu,\sigma}(\gamma) \, : \, \gamma \in \A, \ \gamma(0) = 0, \ \text{and} \ \gamma(1) = x\Big\},
\end{equation}
for $x \in [0,\infty)^2$.  Borrowing language from optimal control theory~\cite{bardi1997}, we will call $U_{\mu,\sigma}$ the value function for this variational problem.      We will often write $J, \ell$ and $U$ in place of $J_{\mu,\sigma}, \ell_{\mu,\sigma}$ and $U_{\mu,\sigma}$, respectively, when it is clear from the context what $\mu$ and $\sigma$ are.  Notice that when $x \in \partial \R^2_+$ with $x_2=0$ we have
\begin{equation}\label{eq:Ubc1}
U(x) = \int_0^{x_1} \mu(t,0)\, dt.
\end{equation}
A similar formula holds when $x \in \partial \R^2_+$ with $x_1=0$, and in general we can write
\begin{equation}\label{eq:Ubc2}
U(x) = (x_1+x_2) \int_0^1 \mu(tx) \, dt,
\end{equation}
for $x \in \partial \R^2_+$.

We also define
\begin{equation}\label{eq:vardef}
W_{\mu,\sigma}(x,y) =\sup \Big\{ J_{\mu,\sigma}(\gamma) \, : \, \gamma \in \A, \ \gamma(0) = x, \ \text{and} \ \gamma(1) = y\Big\},
\end{equation}
for $x,y \in [0,\infty)^2$ with $x \leqq y$. As before, we will often drop the subscripts on $W_{\mu,\sigma}$ when convenient.     Similar to \eqref{eq:Ubc1}--\eqref{eq:Ubc2}, when $x,y \in[0,\infty)^2$ with $x \leqq y$ and $x_2=y_2$, we can write
\begin{equation}\label{eq:Wbc1} 
W(x,y) = \int_{x_1}^{y_1} \mu(t,x_2) \, dt,
\end{equation}
with a similar formula holding when $x_1=y_1$.  In general, whenever $x \leqq y$ but $x_i=y_i$ for some $i$ we can write
\begin{equation}\label{eq:Wbc2}
W(x,y) = (y_1-x_1 + y_2-x_2) \int_0^1 \mu(x + (y-x)t) \, dt.
\end{equation}

The remainder of this section is organized as follows.  In Section \ref{sec:reg} we prove that $U$ and $W$ are uniformly continuous, under assumptions on $\mu$ and $\sigma$ that are similar to (F1) and (F3), but slightly weaker.  Then in Section \ref{sec:hjb}, we show that $U_{\mu + \mu_s,\sigma}$ is a viscosity solution of (P), and prove a similar result for $W_{\mu,\sigma}$.   This result, Theorem \ref{thm:hjb} in Section \ref{sec:hjb}, follows from classical optimal control theory~\cite{bardi1997}, and (P) is exactly the Hamilton-Jacobi-Bellman equation for the variational (optimal-control) problem \eqref{eq:vardefU}.  For more information on Hamilton-Jacobi equations and optimal control, we refer the reader to~\cite{bardi1997}. 

\subsection{Regularity}
\label{sec:reg}

H\"older or Lipschitz regularity of the value function in optimal control theory is a standard classical result~\cite{bardi1997}.  However, it is typically assumed that $x \mapsto \ell_{\mu,\sigma}(x,p)$ is uniformly continuous, which is not compatible with (F1).  We show here that the specific form of $\ell_{\mu,\sigma}$ allows us to show that $U_{\mu+\mu_s,\sigma}$ and $W_{\mu,\sigma}$ are uniformly continuous, provided the discontinuities in $\mu$ occur along monotone increasing curves. 

Since it is useful later, we will slightly weaken the hypothesis (F1), and allow $\mu$ to be ``badly behaved'' within a narrow tube of the monotone curves $\Gamma_i$.  This weakened hypothesis is specifically designed so that the regularity result applies to inf- and sup-convolutions of functions satisfying (F1).  Inf- and sup-convolutions are commonly used for regularization in the theory of viscosity solutions~\cite{bardi1997,crandall1992}.

The weakened hypothesis requires the following notation; for $\theta \geq0$ define
\begin{align}\label{eq:eta-reg}
\Gamma_{i,\theta} &{}={}\Big\{x \in [0,\infty)^2 \, : \, \dist(x,\Gamma_i) \leq \theta\Big\},\\
 \Omega_{i,\theta} &{}={}\Big\{ x \in \Omega_i \, : \, \dist(x,\Gamma_i) > \theta \ \text{ and } \ \dist(x,\Gamma_{i+1}) > \theta\Big\},\\
\Gamma_{\theta} &{}={}\Big\{x \in [0,\infty)^2 \, : \, \dist(x,\Gamma) \leq \theta\Big\},\\
 \Omega_{\theta} &{}={}\Big\{ x \in \Omega \, : \, \dist(x,\Gamma) > \theta \Big\}.
\end{align}
The weakened version of (F1) is the following:
\begin{itemize}
\item[(F1*)] The function $\mu:[0,\infty)^2\to [0,\infty)$ is bounded and upper semicontinuous, $\mu\vert_{\Omega_\theta}  = 0$, and there exists a constant $C_{lip}$ such that for every $i \in \Z$, $\mu\vert_{\Omega_{i,\theta}}$ is Lipschitz continuous with constant $C_{lip}$.
\end{itemize}
    
We now give the regularity result for $W$.
\begin{theorem}\label{thm:reg}
  Suppose that $\mu$ satisfies (F1*) for $\theta \geq 0$, and suppose that $\sigma:[0,\infty)^2 \to [0,\infty)$ is bounded and Borel-measurable. Then for every $R>0$ there exist a modulus of continuity $\omega$, and a constant $C=C(C_{lip},\|\mu\|_{\infty}, \|\sigma\|_{\infty},R)>0$ such that
\begin{equation}\label{eq:approx-holder}
  |W_{\mu,\sigma}(z,x) - W_{\mu,\sigma}(z,y)| \leq C\left(\sqrt{|x-y|} + \omega(|x-y|) + \omega(\theta)\right),
\end{equation}
for all $x,y,z \in [0,R]^2$ with $x,y \geqq z$.  Furthermore, $\omega$ depends only on $\Gamma,\{\Gamma_i\}_{i \in \Z}$ and $R>0$.
\end{theorem}
\begin{proof}
  Let $R>0$.  We will prove the result for $z = 0$; the case of $z \neq 0$ is very similar.  For simplicity of notation, let us set $V(x) = W_{\mu,\sigma}(0,x)$. Notice that we can reduce the proof to the case where $x,y \in [0,R]^2$ with $x \leqq y$.  Indeed, let  $x,y \in [0,R]^2$ and set 
 $x' = (\min(x_1,y_1),\min(x_2,y_2))$.  Then we have 
 \[|U(x)-U(y)| \leq |U(x)-U(x')| + |U(y) - U(x')|,\]
 and $x' \leqq x$ and $x' \leqq y$.

Thus let us assume that $x \leqq y$. Let $\eps > 0$ and let $\gamma \in \A$ such that $\gamma(0)=0$, $\gamma(1)=y$, and  $V(y) \leq J(\gamma) + \eps$.   Define
\[s_1 = \sup \big\{t>0 \, : \, \gamma(t) \leqq x\big\} \ \ \text{and} \ \ s_2 = \inf \big\{t>0 \, : \, \gamma(t) \geqq x\big\}.\]
Without loss of generality, we may assume that $\gamma_2(s_2) = x_2$.  Define 
\[\bar{\gamma}(t) = \Big(\min\big(x_1,\gamma_1(t)\big),\gamma_2(t)\Big) \ \ \text{ for } t \in [0,s_2].\]

The proof is split into two steps now.

\vspace{10pt}

1.  We claim that
\begin{equation}\label{eq:final}
|V(x) - V(y)| \leq  \int_{s_1}^{s_2} \left|\mu(\gamma(t)) - \mu(\bar{\gamma}(t))\right| \gamma_2'(t) \, dt + C\sqrt{|x-y|} + \eps.
\end{equation}
where $C=C(\|\mu\|_{\infty},\|\sigma\|_{\infty},R)$.

To see this: First note that  $\gamma(s_2) \geqq x$ and $\gamma(1) = y$.  It follows that
\begin{align}\label{eq:s2-est}
\int_{s_2}^1 \ell(\gamma(t),\gamma'(t)) \, dt  &\leq \|\mu\|_{\infty} \int_{s_2}^1 \gamma_1'(t) + \gamma_2'(t) \, dt + 2\|\sigma\|_{\infty} \int_{s_2}^1 \sqrt{\gamma'_1(t)\gamma'_2(t)} \, dt \notag \\
&\leq 2\|\mu\|_{\infty}|x-y| + 2\|\sigma\|_{\infty}\left( \int_{s_2}^1 \gamma_1'(t) \, dt \int_{s_2}^1 \gamma_2'(t) \, dt \right)^\frac{1}{2} \notag \\
&\leq 2\|\mu\|_{\infty}|x-y| + 2\|\sigma\|_{\infty} |x-y| \notag \\
&= 2(\|\mu\|_{\infty} + \|\sigma\|_{\infty}) |x-y|,
\end{align}
where the second line follows from H\"older's inequality.  
We claim now that $\gamma_1(s_1) = x_1$.  To see this: suppose to the contrary that $\gamma_1(s_1) < x_1$, which implies that $s_1<s_2$.   By the definition of $s_1$ we must have $\gamma_2(s_1)=x_2$ and $\gamma_2(s) > x_2$ for $s > s_1$.  This contradicts our assumption that $\gamma_2(s_2)=x_2$.   Hence $\gamma_1(s_1)=x_1$.

Now we have
\begin{equation}\label{eq:s1s2}
\int_{s_1}^{s_2} \gamma_1'(t) \,dt = \gamma_1(s_2) - \gamma_1(s_1) \leq y_1 - x_1 \leq |x-y|.
\end{equation}
Since $\gamma=\bar{\gamma}$ on $[0,s_1]$ and  $\bar{\gamma}(s_2) = x$ we have
\begin{align}\label{eq:temp1}
V(y) - V(x) &\leq J(\gamma) + \eps - \int_0^{s_2}\ell(\bar{\gamma}(t),\bar{\gamma}'(t)) \, dt \notag \\
&= \int_{s_1}^{s_2} \ell(\gamma(t),\gamma'(t))  - \ell(\bar{\gamma}(t), \bar{\gamma}'(t)) \,dt + \int_{s_2}^1 \ell(\gamma(t),\gamma'(t)) \,dt + \eps \notag\\
&\hspace{-1mm}\stackrel{\eqref{eq:s2-est}}{\leq} \underbrace{\int_{s_1}^{s_2} \ell(\gamma(t),\gamma'(t))  - \ell(\bar{\gamma}(t), \bar{\gamma}'(t)) \,dt}_A  + C |x-y| + \eps,
\end{align}
where $C=C(\|\mu\|_{\infty},\|\sigma\|_{\infty})$. 
If $s_1=s_2$ then the claim \eqref{eq:final} follows from \eqref{eq:temp1}.  So suppose that $s_1<s_2$.  
Since $\bar{\gamma}_1'(t) = 0$ and $\bar{\gamma}_2'(t) = \gamma_2'(t)$ for $t \in (s_1,s_2)$, we have
\begin{align}\label{eq:Atemp}
A &= \int_{s_1}^{s_2} \left(\mu(\gamma(t)) - \mu(\bar{\gamma}(t))\right) \gamma_2'(t) +  \mu(\gamma(t))\gamma_1'(t) + 2\sigma(\gamma(t))\sqrt{\gamma_1'(t)\gamma_2'(t)} \, dt\notag \\
&\leq \int_{s_1}^{s_2} \left|\mu(\gamma(t)) - \mu(\bar{\gamma}(t))\right| \gamma_2'(t) \, dt + \|\mu\|_{\infty}\int_{s_1}^{s_2} \gamma_1'(t) \, dt \notag \\
&\hspace{5.5cm}+ 2\|\sigma\|_{\infty} \left( \int_{s_1}^{s_2} \gamma_1'(t) \, dt \int_{s_1}^{s_2} \gamma_2'(t) \, dt\right)^\frac{1}{2}\notag \\
&\hspace{-1mm}\stackrel{\eqref{eq:s1s2}}{\leq} \int_{s_1}^{s_2} \left|\mu(\gamma(t)) - \mu(\bar{\gamma}(t))\right| \gamma_2'(t) \, dt +  C(\|\mu\|_\infty,\|\sigma\|_\infty,R)\sqrt{|x-y|},
\end{align}
which establishes \eqref{eq:final}.

\vspace{10pt}

2.  We claim that
\begin{equation}\label{eq:Bclaim}
\int_{s_1}^{s_2} \left|\mu(\gamma(t)) - \mu(\bar{\gamma}(t))\right| \gamma_2'(t) \, dt \leq C\left(\sqrt{|x-y|} + \omega(|x-y|) + \omega(\theta) + \theta\right),
\end{equation}
where $C=C(C_{lip},R,\|\mu\|_\infty,\|\sigma\|_\infty)$.  Notice that once \eqref{eq:Bclaim} is established, the proof is completed by combining \eqref{eq:Bclaim} with \eqref{eq:final} and sending $\eps \to 0$.

Since the collection of curves $\{\Gamma_i\}_{i=-\infty}^\infty$ is locally finite, we may assume that $\Gamma_{1,\theta},\dots,\Gamma_{M,\theta}$ are the only tubular neighborhoods that have a non-empty intersection with $[0,R]^2$. Since $\Gamma_i$ is continuous and strictly increasing, we can parameterize the portion of $\Gamma_i$ that intersects $[0,R]^2$ as follows:
\[\Gamma_i: t \mapsto (t,f_i(t)), \ \ t \in I_i,\]
where $f_i:I_i \to [0,\infty)$ is continuous and strictly increasing, and $I_i$ is a closed interval in $[0,R]$.  Similarly we can parameterize $\Gamma$ as
\[\Gamma: t \mapsto (t,f(t)), \ \ t \in [0,1],\]
where $f:[0,1]\to[0,1]$ is continuous and strictly decreasing.  Note that the  functions $f_1,\dots,f_M,f$ share a common modulus of continuity $\omega$, by virtue of their compact domains.  We also note that $\omega$ and $M$ depend only on $\Gamma, \{\Gamma_i\}_{i \in \Z}$, and $R>0$.

\vspace{10pt} 

To prove \eqref{eq:Bclaim}, first set $c=\omega(\theta) + \theta$.  A simple computation shows that
\begin{equation}\label{eq:dist}
\dist((t,f_i(t) + c),\Gamma_i) > \theta \ \ \text{ and } \ \ \dist((t,f_i(t) - c),\Gamma_i)>\theta,
\end{equation}
for any $t\in I_i$.  A similar statement holds for $\Gamma$ and $f$.  For each $i\in\{1,\dots,M\}$, we define
\[m_i^+ = \sup_{I_i\cap[x_1,y_1]} f_i, \ \ \text{and} \ \ m_i^- = \inf_{I_i\cap[x_1,y_1]} f_i,\]
and
\begin{equation} \label{eq:Ki}
K_i = \Big\{t \in (s_1,s_2) \, : \, (x_1,m_i^--c) \leqq \gamma(t) \leqq (y_1,m_i^++c)\Big\}.
\end{equation}
Similarly we set 
\[m^+ = \sup_{[0,1]\cap[x_1,y_1]} f, \ \ \text{and} \ \ m^- = \inf_{[0,1]\cap[x_1,y_1]} f,\]
\begin{equation} \label{eq:K}
K = \Big\{t \in (s_1,s_2) \, : \, (x_1,m^--c) \leqq \gamma(t) \leqq (y_1,m^++c) \Big\}.
\end{equation}
and
\begin{equation}\label{eq:H}
\H =(s_1,s_2) \setminus ( K \cup K_1 \cup \cdots \cup K_M)
\end{equation}

By the definition of $m_i^\pm$ and $m^\pm$ we have
\begin{equation}\label{eq:lambda-cont}
m_i^+ - m_i^- \leq \omega(y_1-x_1) \ \ \text{and} \ \ m^+ - m^- \leq \omega(y_1-x_1).
\end{equation}
It follows from \eqref{eq:dist}--\eqref{eq:H}, (F1*), and the fact that $\gamma$ is monotone, that whenever $t \in \H$ we have either $\mu(\gamma(t))=\mu(\bar{\gamma}(t)) = 0$ or 
 \[\mu(\gamma(t)) - \mu(\bar{\gamma}(t)) = \mu_{i,\theta}(\gamma(t)) - \mu_{i,\theta}(\bar{\gamma}(t)),\]
 for some $i \in \{0,\dots,M\}$.  
Thus, invoking (F1*), we have
\begin{equation}\label{eq:good}
|\mu(\gamma(t)) - \mu(\bar{\gamma}(t))| \leq C_{lip} |\gamma(t) - \bar{\gamma}(t)| \leq C_{lip}|x-y|,
\end{equation} 
for all $t \in \H$.

For any $i\in \{1,\dots,M\}$, we have
\begin{align}\label{eq:bad}
\int_{K_i}|\mu(\gamma(t)) - \mu(\bar{\gamma}(t))| \gamma_2'(t) \, dt &\leq 2\|\mu\|_{\infty} \int_{K_i} \gamma_2'(t) \, dt \notag \\
&\leq 2\|\mu\|_{\infty} |m_i^+ - m_i^- + 2c| \notag \\
&\hspace{-1mm}\stackrel{\eqref{eq:lambda-cont}}{\leq} 2\|\mu\|_{\infty} \omega(|x-y|) + 4\|\mu\|_\infty(\omega(\theta) + \theta).
\end{align}
We have an identical estimate when $K_i$ is replaced by $K$. 
Combining \eqref{eq:good} with \eqref{eq:bad} we have
\begin{align}\label{eq:done}
\int_{s_1}^{s_2} &\left|\mu(\gamma(t)) - \mu(\bar{\gamma}(t))\right| \gamma_2'(t) \, dt \notag \\
&= \int_{\H} \left|\mu(\gamma(t)) - \mu(\bar{\gamma}(t))\right| \gamma_2'(t) \, dt + \int_{K\cup K_1 \cup \dots \cup K_M}\left|\mu(\gamma(t)) - \mu(\bar{\gamma}(t))\right| \gamma_2'(t) \, dt\notag\\
&\leq C_{lip} |x-y| \int_{s_1}^{s_2} \gamma_2'(t) \, dt + \sum_{i=1}^M \int_{K_i}\left|\mu(\gamma(t)) - \mu(\bar{\gamma}(t))\right| \gamma_2'(t) \, dt \notag \\
&{} \hspace{0.45\linewidth}+ \int_{K}\left|\mu(\gamma(t)) - \mu(\bar{\gamma}(t))\right| \gamma_2'(t) \, dt  \notag \\
&\leq C_{lip}R|x-y| + 2(M+1)\|\mu\|_{\infty} \omega(|x-y|) + 4(M+1)\|\mu\|_\infty (\omega(\theta) + \theta),
\end{align}
which establishes \eqref{eq:Bclaim} and completes the proof.
\end{proof}
\begin{corollary}\label{cor:reg}
Suppose that $\mu$ satisfies (F1*) for $\theta \geq 0$, and suppose that $\sigma$ is bounded and Borel-measurable. Then for every $R>0$ there exist a modulus  of continuity $\omega$, and a constant $C=C(C_{lip},\|\mu\|_{\infty}, \|\sigma\|_{\infty},R)>0$ such that
\begin{equation}\label{eq:approx-holder2}
  |W_{\mu,\sigma}(x,z) - W_{\mu,\sigma}(y,z)| \leq C\left(\sqrt{|x-y|} + \omega(|x-y|) + \omega(\theta)\right),
\end{equation}
for all $x,y,z \in [0,R]^2$ with $x,y \leqq z$.  As in Theorem \ref{thm:reg}, $\omega$ depends only on $\Gamma,\{\Gamma_i\}_{i \in \Z}$ and $R>0$.
\end{corollary}
\begin{proof}
  The proof follows from Theorem \ref{thm:reg} by symmetry.
\end{proof}
\begin{remark}\label{rem:reg}
Notice in Theorem \ref{thm:reg} that if $\theta = 0$ then we have the estimate
\begin{equation}\label{eq:estimate}
  |W_{\mu,\sigma}(z,x)-W_{\mu,\sigma}(z,y)| \leq C\left(\sqrt{|x-y|} + \omega(|x-y|)\right),
\end{equation}
for all $x,y,z \in [0,R]^2$ with $x,y\geq z$.  Inspecting the proof of Theorem \ref{thm:reg}, we see that $\omega$ is the modulus of continuity of the curves $\Gamma, \{\Gamma_i\}_{i \in \Z}$ as functions over both coordinate axes. Thus, the regularity of $W$ is inherited from the regularity of the curves $\Gamma, \{\Gamma_i\}_{i \in \Z}$.  For example, if the curves $\Gamma, \{\Gamma_i\}_{i \in \Z}$ are H\"older-continuous with exponent $\alpha\leq 1/2$ as functions over both coordinate axes, then we have that $W(z,\cdot) \in C^{0,\alpha}([z_1,R]\times[z_2,R])$ for every $R>0$ and its H\"older seminorm depends only on $\|\mu\|_\infty$, $\|\sigma\|_\infty$, $R$, and $C_{lip}$.  The same remark holds for Corollary \ref{cor:reg} and \eqref{eq:approx-holder2}.
\end{remark}

We now plan to use Theorem \ref{thm:reg} to prove a similar regularity result for $U_{\mu+\mu_s,\sigma}$.  To do this, we relate $W$ and $U$ via the following dynamic programming principle:
\begin{proposition}\label{prop:dpp}
Suppose that $\mu$ satisfies (F1*) for $\theta\geq0$, $\mu_s$ satisfies (F2),  and suppose that $\sigma$ is bounded and Borel-measurable.  Then for any $y \in [0,\infty)^2$ we have
\begin{equation}\label{eq:dpp-bc}
  U_{\mu+\mu_s,\sigma}(y) = \max_{x \in \partial \R^2_+ \, : \, x \leqq y }  \Big\{ U_{\mu+\mu_s,\sigma}(x) + W_{\mu,\sigma}(x,y) \Big\}.
\end{equation}
\end{proposition}
Notice that the boundary source $\mu_s$ is absent in the term $W_{\mu,\sigma}$ in \eqref{eq:dpp-bc}.  This allows us to concentrate much of our analysis on $W_{\mu,\sigma}$, which involves only the macroscopic inhomogeneities in the bulk $\R^2_+$, and then extend our results to hold for $U_{\mu+\mu_s,\sigma}$ via the dynamic programming principle \eqref{eq:dpp-bc}.
\begin{proof}
  We first note that the maximum in \eqref{eq:dpp-bc} is indeed attained, due to the continuity of $U_{\mu+\mu_s,\sigma}$ restricted to $\partial \R^2_+$ and Corollary \ref{cor:reg}.

If $y \in \partial \R^2_+$, then in light of \eqref{eq:Ubc2}, \eqref{eq:Wbc2} and the fact that $\mu_s\geq0$, the maximum in \eqref{eq:dpp-bc} is attained at $x=y$ and the validity of \eqref{eq:dpp-bc} is trivial.

Suppose now that $y \in \R^2_+$ and let $v(y)$ denote the right hand side in \eqref{eq:dpp-bc}, and set $U=U_{\mu+\mu_s,\sigma}$.  We first show that $U \leq v$.  Let $\eps > 0$ and $\gamma \in \A$ such that $\gamma(0)=0$, $\gamma(1)=y$ and $J_{\mu+\mu_s,\sigma}(\gamma) \geq U(y) - \eps$.  Let
\[s = \sup \big\{ t \in [0,1] \, : \gamma(t) \in \partial \R^2_+\big\}.\]
Then we have $0 \leq s < 1$.  Set $x=\gamma(s)$ and
\[\gamma^1(t) = \gamma(st) \ \ \text{and} \ \ \gamma^2(t) = \gamma\big(s + t(1-s)\big),\]
for $t \in [0,1]$.  Then we have
\[U(y) \leq J_{\mu+\mu_s,\sigma}(\gamma) +\eps  = J_{\mu + \mu_s,\sigma}(\gamma^1) + J_{\mu,\sigma}(\gamma^2) + \eps \leq U(x) + W(x,y) + \eps.\]
Sending $\eps \to 0$ we have $U \leq v$. 

We now show that $v \leq U$.  Let $x \in \partial \R^2_+$ be a point at which the maximum is attained in \eqref{eq:dpp-bc} and let $\eps > 0$.  Let $\gamma^1 \in \A$ with $\gamma^1(0)=0$, $\gamma^1(1)=x$ such that $U(x) \leq J_{\mu+\mu_s,\sigma}(\gamma^1) + \frac{\eps}{3}$.  Let $z \in [x,y]$ such that $z \in \R^2_+$ and $W(x,y) \leq W(z,y) + \frac{\eps}{3}$.  Let $\gamma^2 \in \A$ with $\gamma^2(0)=z$, $\gamma^2(1)=y$ such that $W(z,y) \leq J_{\mu,\sigma}(\gamma^2) + \frac{\eps}{3}$.  We can stitch together $\gamma^1$ and $\gamma^2$ as follows
\[\gamma(t) = \begin{cases}
\gamma^1(3t),& \text{if } 0 \leq t < \frac{1}{3}, \\
x + (3t-1)(z-x),& \text{if } \frac{1}{3} \leq t < \frac{2}{3}, \\
\gamma^2(3t-2),& \text{if } \frac{2}{3} \leq t \leq 1.
\end{cases}\]
Then we have
\begin{align*}
v(y) = U(x) + W(x,y) &\leq U(x) + W(z,y) + \frac{\eps}{3}\\
&\leq J_{\mu+\mu_s,\sigma}(\gamma^1)  + J_{\mu,\sigma}(\gamma^2) + \eps\leq J_{\mu+\mu_s,\sigma}(\gamma) + \eps \leq U(y) + \eps,
\end{align*}
where we used the fact that $\gamma^2(t) \in \R^2_+$ for all $t$, hence $J_{\mu,\sigma}(\gamma^2)=J_{\mu+\mu_s,\sigma}(\gamma^2)$.
Sending $\eps \to 0$ we have $v \leq U$.  
\end{proof}

Before continuing with the regularity result for $U_{\mu+\mu_s,\sigma}$, let us introduce a bit of notation.  For $\xi \in \R^d_+$, let $\pi_\xi:\R^d \to [0,\xi]$ denote the projection mapping $\R^d$ onto the convex set $[0,\xi]$.  For $x \in [0,\infty)^d$, $\pi_\xi$ is given explicitly by
\begin{equation}\label{eq:proj}
\pi_\xi(x) = \Big(\min(x_1,\xi_1),\dots,\min(x_d,\xi_d)\Big).
\end{equation}

\begin{corollary}\label{cor:Ureg}
Suppose that $\mu$ satisfies (F1*) for $\theta\geq 0$, $\mu_s$ satisfies (F2), and suppose that $\sigma$ is bounded and Borel-measurable.  Then for every $R>0$ there exists a modulus of continuity $\omega$, and a constant $C=C(C_{lip},\|\mu\|_{\infty}, \|\sigma\|_{\infty},\|\mu_s\|_\infty,R)>0$ such that
\begin{equation}\label{eq:approx-holder3}
  |U_{\mu+\mu_s,\sigma}(x) - U_{\mu+\mu_s,\sigma}(y)| \leq C\left(\sqrt{|x-y|} + \omega(|x-y|) + \omega(\theta)\right),
\end{equation}
for all $x,y \in [0,R]^2$.  As in Theorem \ref{thm:reg}, $\omega$ depends only on $\Gamma,\{\Gamma_i\}_{i \in \Z}$ and $R>0$.
\end{corollary}
\begin{proof}
  Let $x,y \in [0,R]^2$ and set $U=U_{\mu+\mu_s,\sigma}$ and $W=W_{\mu,\sigma}$.  As in Theorem \ref{thm:reg} we may assume that $x \leqq y$.  By Proposition \ref{prop:dpp}, there exists $y' \in \partial \R^2_+$ with $y' \leqq y$ such that
\begin{equation}\label{eq:dpp1-U}
U(y) = U(y') + W(y',y).
\end{equation}
Set $x' = \pi_x(y')$.  Then since $x' \in \partial \R^2_+$ and $x' \leqq x$, we have  by Proposition \ref{prop:dpp} that
\begin{equation}\label{eq:dpp2-U}
U(x) \geq U(x') + W(x',x).
\end{equation}
By subtracting \eqref{eq:dpp2-U} from \eqref{eq:dpp1-U} and recalling \eqref{eq:Ubc2} we have
\begin{align*}
|U(x) - U(y)| &= U(y)-U(x) \\
&\leq U(y') - U(x') + W(y',y) - W(x',x)\\
&\leq \|\mu + \mu_s\|_\infty|x'-y'| + |W(y',y) - W(x',y)| + |W(x',y) - W(x',x)|. 
\end{align*}
The proof is completed by applying Theorem \ref{thm:reg} and Corollary \ref{cor:reg} and noting that \\ $|x'-y'| \leq |x-y|$.
\end{proof}
Of course, Remark \ref{rem:reg} holds with obvious modifications for $U$ and \eqref{eq:approx-holder3}.
\begin{remark}\label{rem:discont}
The hypothesis that the curves $\Gamma_i$ are continuous and strictly increasing cannot in general be weakened to continuous and non-decreasing.  For example, consider the case where $\mu=\sigma=1$ on $[0.5,1]\times[0,1]$ and $\mu=\sigma=0$ on $[0,0.5)\times[0,1]$.  Then we have
\[U_{\mu,\sigma}(x) = \begin{cases}
0,&\text{if } x \in [0,0.5)\times[0,1],\\
x_1 + x_2 - 0.5 + 2\sqrt{(x_1-0.5)x_2},&\text{if } x \in [0.5,1]\times[0,1],
\end{cases}\]
which has a discontinuity along the vertical line $\{x_1=0.5\}$, which would correspond to one of the curves $\Gamma_i$ on which $\mu$ is discontinuous.  
\end{remark}

\subsection{Hamilton-Jacobi-Bellman equation}
\label{sec:hjb}

In this section we show in Theorem \ref{thm:hjb} that $U_{\mu + \mu_s,\sigma}$ is a viscosity solution of (P).  In fact, (P) is the Hamilton-Jacobi-Bellman equation for the simple optimal control problem~\cite{bardi1997} defined by $U_{\mu + \mu_s,\sigma}$.  For more information on  the  connection between Hamilton-Jacobi equations and optimal control problems, we refer the reader to~\cite{bardi1997}.

Let us pause momentarily to recall the definition of viscosity solution of
\begin{equation}\label{eq:genHJ}
H(x,Du) = 0  \ \ \text{ on } \ \   \O,
\end{equation}
where $\O \subset \R^d$ is open, $H:\O\times \R^d \to \R$ is locally bounded with $p \mapsto H(x,p)$ continuous for every $x \in \O$, and $u: \O \to \R$ is the unknown function.  For more information on viscosity solutions of Hamilton-Jacobi equations,  we refer the reader to~\cite{bardi1997,crandall1992}.  

We denote by $\text{USC}(\O)$ (resp.~$\text{LSC}(\O)$) the set of upper semicontinuous (resp.~lower semicontinuous) functions on $\O$. For $u:\O \to \R$, the \emph{superdifferential} of $u$ at $x \in \O$, denoted $D^+u(x)$, is the set of all $p \in \R^d$ satisfying
\begin{equation}\label{eq:superjet}
u(y) \leq u(x) + \langle p,y-x\rangle + o(|x-y|) \ \text{ as } \ \O \ni y \to x.
\end{equation}
Similarly, the \emph{subdifferential} of $u$ at $x\in \O$, denoted $D^-u(x)$, is the set of all $p \in \R^d$ satisfying
\begin{equation}\label{eq:subjet}
u(y) \geq u(x) + \langle p,y-x\rangle + o(|x-y|) \ \text{ as } \ \O \ni y \to x.
\end{equation}
Equivalently, we may set
\[D^+ u(x) = \{ D\phi(x) \, : \, \phi \in C^1(\O) \ \text{ and } \ u-\phi \ {\rm has \ a \ local \ max \ at \ } x\},\]
and
\[D^- u(x) = \{ D\phi(x) \, : \, \phi \in C^1(\O) \ \text{ and } \ u-\phi \ {\rm has \ a \ local \ min \ at \ } x\}.\]
\begin{definition}
A \emph{viscosity subsolution} of \eqref{eq:genHJ} is a function $u \in \text{USC}(\O)$ satisfying
\begin{equation}
\liminf_{y \to x} H(y,p) =:H_*(x,p) \leq 0 \ \text{ for all } \ x \in \O \ \text{ and } \ p \in D^+u(x).
\end{equation}
Similarly, a \emph{viscosity supersolution} of \eqref{eq:genHJ} is a function $u\in \text{LSC}(\O)$ satisfying
\begin{equation}
\limsup_{y \to x} H(y,p)=:H^*(x,p) \geq 0 \ \text{ for all } \ x \in \O \ \text{ and } \ p \in D^-u(x).
\end{equation}
\end{definition}

The functions $H_*$ and $H^*$ are the lower and upper semicontinuous envelopes of $H$ with respect to the spatial variable, respectively. We will often say $u$ is a viscosity solution of
\[ H(x,Du) \leq 0 \ \ \ (\text{resp. } H(x,Du) \geq 0) \ \ \text{on } \ \ \O,\]
to indicate that $u$ is a viscosity subsolution (resp.~supersolution) of \eqref{eq:genHJ}. 
 If $u$ is a viscosity subsolution and supersolution of \eqref{eq:genHJ}, then we say that $u$ is a \emph{viscosity solution} of \eqref{eq:genHJ}.  Notice that viscosity solutions defined in this way are necessarily continuous.  

\begin{theorem}\label{thm:hjb}
Suppose that $\mu,\sigma:[0,\infty)^2 \to [0,\infty)$ are Borel-measurable and bounded.  Let $z \in [0,\infty)^2$ and set
  $V(x) = W_{\mu,\sigma}(z,x)$ for $x \in [z,\infty)$.   If $V$ is continuous then $V$ satisfies
\begin{equation}\label{eq:hjb}\left.\begin{aligned}
(V_{x_1} - \mu)_+(V_{x_2} - \mu)_+ &=\sigma^2& &\text{on } (z,\infty),\\
\min(V_{x_1},V_{x_2}) &\geq \mu& &\text{on } (z,\infty),\\
\end{aligned}\right\}\end{equation}
in the viscosity sense.
\end{theorem}
Recall that $[z,\infty)=\{y\in \R^2 \, : \, y \geqq z\}$, and $(z,\infty) = \{y \in \R^2 \, : \, y > z\}$. 
\begin{proof}
The proof is based on a standard technique from optimal control theory for relating variational problems to Hamilton-Jacobi equations~\cite{bardi1997}.  The proof is very similar to~\cite[Theorem 2]{calder2014}.   We will only sketch parts of the proof here.    

The proof is based on the following dynamic programming principle 
\begin{equation}\label{eq:dpp}
V(y) = \sup_{x \in \partial B_r(y) \, : \, x \leqq y } \Big\{ V(x) + W(x,y)\Big\},
\end{equation}
which holds for $y \in (z,\infty)$ and $r>0$ small enough so that $B_r(y) \subset (z,\infty)$. The proof of \eqref{eq:dpp} is very similar to the proof of Proposition \ref{prop:dpp}.

We now show that $V$ is a viscosity solution of \eqref{eq:hjb}.  Let $y \in (z,\infty)$ and let $p \in D^- V(y)$.   As in~\cite{calder2014}, we can use the dynamic programing principle to obtain
\begin{equation}\label{eq:super-a}
\sup_{a \in \R^2_+} \Big\{ -\big\langle p - \mu_*(y)(1,1),a\big\rangle + 2\sigma_*(y)\sqrt{a_1a_2}\Big\} \leq 0.
\end{equation}
Suppose now that $\sigma_*(y)=0$. Then we automatically have 
\[(p_1-\mu_*(y))_+(p_2 - \mu_*(y))_+ \geq 0 = \sigma^2_*(y).\]  
Furthermore, it follows from \eqref{eq:super-a} that $\min(p_1,p_2) \geq \mu_*(y)$, so we are done.  Consider now $\sigma_*(y)>0$.  Setting $a_1=1$ in \eqref{eq:super-a} we have
\[\sup_{a_2>0} \Big\{ -(p_1 - \mu_*(y)) - (p_2 - \mu_*(y))a_2 + 2\sigma_*(y)\sqrt{a_2}\Big\} \leq 0.\]
It follows that $p_2 > \mu_*(y)$.  By a similar argument we have $p_1 > \mu_*(y)$, and hence we have $\min(p_1,p_2) > \mu_*(y)$. This establishes that $V$ is a viscosity solution of 
\[ \min(V_{x_1},V_{x_2}) \geq \mu \ \ \text{ on } \  (z,\infty).\]
  Now set 
\begin{equation}\label{eq:maximizer}
a_1 = \sqrt{\frac{p_2 - \mu_*(y)}{p_1 - \mu_*(y)}}  \ \  \text{ and } \ \ a_2 = \sqrt{\frac{p_1 - \mu_*(y)}{p_2 - \mu_*(y)}} 
\end{equation}
in \eqref{eq:super-a} and simplify to find that
\[(p_1 - \mu_*(y)) (p_2 - \mu_*(y)) \geq {\sigma_*}^2(y).\]
Therefore $V$ is a viscosity solution of 
\[(V_{x_1} - \mu)_+(V_{x_2} - \mu)_+ \geq \sigma^2 \ \ \text{ on } \ (z,\infty).\]

  Let $y \in (z,\infty)$ and let $p \in D^+ V(y)$.  Utilizing the dynamic programing principle \eqref{eq:dpp} again we have
\begin{equation}\label{eq:sub-a}
\sup_{a \in \R^2_+ \, : \, a_1a_2 = 1} \Big\{ -\big\langle p - \mu^*(y)(1,1),a\big\rangle + 2\sigma^*(y)\Big\} \geq 0.
\end{equation}
If $\min(p_1,p_2) \leq \mu^*(y)$ then we immediately have
\[(p_1 - \mu^*(y))_+ (p_2 - \mu^*(y))_+ = 0 \leq {\sigma^*(y)}^2.\]
If  $\min(p_1,p_2)> \mu^*(y)$ then we have that
\[\limsup_{|a| \to \infty \, : \, a \in \R^2_+}-\big\langle p - \mu^*(y)(1,1),a\big\rangle + 2\sigma^*(y) = -\infty.\] 
It follows that the supremum in \eqref{eq:sub-a} is attained at some $a^* \in \R^2_+$. Introducing a Lagrange multiplier $\lambda>0$, the necessary conditions for $a^*$ to be a maximizer of the constrained maximization problem \eqref{eq:sub-a} are
\[a^*_1 = \lambda(p_2-\mu^*(y)), \ a^*_2 = \lambda(p_1-\mu^*(y)), \ \ \text{and } a_1^*a_2^* = 1.\] 
It follows that $\lambda = (p_1-\mu^*(y))^{-\frac{1}{2}}(p_2-\mu^*(y))^\frac{1}{2}$ and $a^*$ is given by \eqref{eq:maximizer}.   Substituting this into \eqref{eq:sub-a} we find that
\[(p_1 - \mu^*(y)) (p_2 - \mu^*(y)) \leq {\sigma^*}^2(y),\]
and hence $V$ is a viscosity solution of
\[(V_{x_1} - \mu)_+(V_{x_2} - \mu)_+ \leq \sigma^2 \ \ \text{ on } \ (z,\infty),\]
which completes the proof.
\end{proof}
\begin{remark}\label{rem:hjb}
  It follows from Theorem \ref{thm:hjb} that $U=U_{\mu + \mu_s,\sigma}$ is a viscosity solution of (P) and satisfies
  \begin{equation}
    \min(U_{x_1},U_{x_2}) \geq \mu \ \ \text{on } \R^2_+
    \label{eq:mingrad}
  \end{equation}
  in the viscosity sense.  Indeed, we can simply apply Theorem \ref{thm:hjb} with  $\mu+\mu_s$ in place of $\mu$ and $z=0$, in which case we have $U(x) = W_{\mu+\mu_s,\sigma}(0,x)$.  
\end{remark}

\section{Comparison Principle}
\label{sec:comparison}

We study here the general Hamilton-Jacobi equation
\begin{equation}
\left. \begin{aligned}
H(x,Du) &= 0& &\text{on } (z,\infty), \\
u &= \phi& &\text{on } \partial (z,\infty).
\end{aligned}\right\}
\end{equation}
Here, $z \in [0,\infty)^d$,  $\phi:\partial (z,\infty)\to \R$ is continuous and monotone, $H:\R^d_+\times \R^d \to \R$ is the Hamiltonian, and $u:[z,\infty)\to \R$ is the unknown function.  For simplicity of notation, we will set $z=0$ throughout much of this section.  The case where $z\neq 0$ follows by a simple translation argument.    

We place the following assumptions on $H$: 
\begin{itemize}
\item[(H1)] For every $x \in \R^d_+$, the mapping $H(x,\cdot):\R^d \to \R$ is monotone non-decreasing.
\item[(H2)] There exists a modulus of continuity $m$ such that
\begin{equation}\label{eq:modulus-H}
H(x,p) - H(y,p) \leq m(|p||x-y| + |x-y|)
\end{equation}
for all $p \in [0,\infty)^d$ and $x,y \in \R^d_+$.
\end{itemize}
The assumption (H1) is clearly satisfied by (P), and generalizes the comparison results in our previous work~\cite{calder2014}, which was focused on the special case of $H(x,p) = p_1\cdots p_d - f(x)$.  The assumption (H2) is standard in the theory of viscosity solutions~\cite{crandall1992}.  

We now give a comparison principle for Hamiltonians $H$ satisfying (H1) and (H2).  
\begin{theorem}\label{thm:comp}
Suppose that $H$ satisfies (H1) and (H2). Let $u \in \text{USC}([0,\infty)^d)$ be a viscosity solution of
\begin{equation}\label{eq:u-sub}
H(x,Du) \leq 0 \ \ \text{in } \R^d_+,
\end{equation}
let $v \in \text{LSC}([0,\infty)^d)$ be a monotone viscosity solution of 
\begin{equation}\label{eq:v-super}
H(x,Dv) \geq a \ \ \text{in } \R^d_+,
\end{equation}
where $a >0$, and suppose that $u \leq v$ on $\partial \R^d_+$. Then $u \leq v$ on $\R^d_+$.
\end{theorem}
The proof of Theorem \ref{thm:comp} is based on the auxiliary function technique, which is standard in the theory of viscosity solutions~\cite{crandall1992,bardi1997}, with modifications to incorporate the lack of compactness resulting from the unbounded domain $\R^d_+$. A standard technique for dealing with unbounded domains is to assume the Hamiltonian $H$ is uniformly continuous in the gradient $p$ and modify the auxiliary function (see, for example~\cite[Theorem 3.5]{bardi1997}).   Since (P) is not uniformly continuous in the gradient, we cannot use this technique.  In our previous work~\cite{calder2014}, we included an additional boundary condition at infinity to induce compactness.  It turns out that this is not necessary, and in the proof of Theorem \ref{thm:comp}, we instead heavily exploit the structure of the Hamiltonian, namely (H1), to produce the required compactness.
\begin{proof}
Since $v$ is monotone (i.e., non-decreasing), it is bounded below by $v(0)$.  Without loss of generality we may assume that $v(0)=0$. Let $h > 0$ and set $v_h(x) = v(x) + h(x_1 + x_2)$.  It follows from (H1) that $v_h$ is a viscosity solution of \eqref{eq:v-super}. Assume by way of contradiction that $\sup_{\R^d_+}(u - v_h) > 0$.   Let $\Psi:\R\to\R$ be a $C^1$ function satisfying
\begin{equation}\label{eq2:Psi-prop}
\left.\begin{aligned}
\Psi(t) &= t& &\text{for all } t \leq 1,\\
\Psi(t) &\leq 2& &\text{for all } t \in \R,\\
0 < \Psi'(t)  &\leq 1&  &\text{for all } t \in \R.
\end{aligned}\right\}
\end{equation}
For $c>0$ set $\bar{u}(x) = c\Psi(c^{-1}u(x))$, and choose $c$ large enough so that 
\[\delta:=\sup_{\R^d_+}(\bar{u} - v_h) > 0.\]
Since $\Psi$ is $C^1$ and $\Psi' > 0$, it is a standard application of the chain rule~\cite{bardi1997} to show that $\bar{u}$ is a viscosity solution of 
\begin{equation}\label{eq2:newsub}
H\left(x,\Psi'(c^{-1}u(x))^{-1} D\bar{u}\right) \leq 0 \ \ \text{on } \R^d_+.
\end{equation}
Since $\Psi'(t) \in (0,1]$ for all $t\geq 0$, we can apply (H1) to \eqref{eq2:newsub} to find that $\bar{u}$ is a viscosity solution of \eqref{eq:u-sub}.

For $\alpha>0$ we define
\begin{equation}\label{eq2:Phialpha}
\Phi_\alpha(x,y) = \bar{u}(x) - v_h(y) - \frac{\alpha}{2} |x-y|^2,
\end{equation} 
and $M_\alpha = \sup_{\R^d_+\times \R^d_+} \Phi_{\alpha}$.
Since $\bar{u} \leq 2c$ and $v_h \geq 0$, we have by \eqref{eq2:Phialpha}  that
\begin{equation}\label{eq2:xy}
|x-y| \leq \frac{2\sqrt{c}}{\sqrt{\alpha}} \ \ \text{whenever } \Phi_\alpha(x,y) \geq 0.
\end{equation}
Since $v_h(y) \geq h(y_1 + y_2)$ we have
\begin{equation}\label{eq:tem}
\Phi_\alpha(x,y) \leq 2c - h(y_1 + y_2).
\end{equation}
Since $\Phi_\alpha$ is upper semicontinuous and $M_\alpha \geq \delta > 0$, it follows from \eqref{eq2:xy} and \eqref{eq:tem} that for every $\alpha > 0$ there exist $x_\alpha,y_\alpha \in [0,\infty)^d$ such that 
\begin{equation}\label{eq2:max}
\Phi_\alpha(x_\alpha,y_\alpha) = M_\alpha \geq \delta > 0,
\end{equation}
and 
\begin{equation}\label{eq2:compact}
y_{\alpha,1} + y_{\alpha,2} \leq \frac{2c}{h}.
\end{equation}
Furthermore, by \eqref{eq2:xy} and \eqref{eq2:compact} we see that, upon passing to a subsequence if necessary, we have $x_\alpha,y_\alpha \to x_0$ as $\alpha \to \infty$ for some $x_0 \in [0,\infty)^d$. Since  $(x,y) \mapsto \bar{u}(x)-v_h(y)$ is upper semicontinuous we have
\[\limsup_{\alpha \to \infty} M_\alpha \leq \limsup_{\alpha \to \infty} \bar{u}(x_\alpha) - v_h(y_\alpha) \leq  \bar{u}(x_0) - v_h(x_0).\]
Since $M_\alpha \geq \bar{u}(x_0) - v_h(x_0)$ for all $\alpha$ we have that $M_\alpha \to \bar{u}(x_0)-v(x_0) = \delta>0$ as $\alpha \to \infty$ and hence
\begin{equation}\label{eq2:key}
\alpha|x_\alpha - y_\alpha|^2 \longrightarrow 0.
\end{equation}
Since $\bar{u} \leq v_h$ on $\partial \R^d_+$ we must have $x_0 \in \R^d_+$, and therefore $x_\alpha,y_\alpha \in \R^d_+$ for $\alpha$ large enough.  

Set $p=\alpha(x_\alpha-y_\alpha)$.  By \eqref{eq2:max} we have that
\[p\in D^+ \bar{u}(x_\alpha)\cap D^-v_h(y_\alpha).\]
  Therefore we have
\[H(x_\alpha,p) \leq 0 \ \ \text{and} \ \ H(y_\alpha,p) \geq a.\]
Subtracting the above inequalities and invoking (H2) we have
\[0<a \leq H(y_\alpha,p) - H(x_\alpha,p)\leq m(|p||x_\alpha-y_\alpha| + |x_\alpha-y_\alpha|) \leq m(\alpha|x_\alpha-y_\alpha|^2 + |x_\alpha-y_\alpha|).\]
Sending $\alpha\to \infty$ we arrive at a contradiction.  Therefore $u \leq v_h$, and sending $h \to 0^+$ completes the proof. 
\end{proof}

We now aim to extend this comparison principle to Hamiltonians with discontinuous spatial dependence.  The techniques we use here are a generalization of our previous work on the longest chain problem~\cite{calder2014}.  
We make the following definitions.
\begin{definition}
  Given a function $u: [z,\infty) \to \R$ and $\xi \in [z,\infty)$, we define the $\xi$-\emph{truncation} of $u$ by $u^\xi:= u \circ \pi_\xi$, where $\pi_\xi$ is the projection mapping $\R^d$ onto $[0,\xi]$ defined in \eqref{eq:proj}.
\end{definition}
\begin{definition}
Let $u$ be a viscosity solution of
\begin{equation}\label{eq:trunc-def}
H(x,Du) \leq 0 \ \ \text{on } (z,\infty).
\end{equation}
We say that $u$ is \emph{truncatable} if for every $\xi \in (z,\infty)$, the  $\xi$-truncation $u^\xi$ is a viscosity solution of \eqref{eq:trunc-def}.
\end{definition}
This notion of truncatability is in spirit the same as~\cite[Definition 2.7]{calder2014}, though the exact definition is slightly different for notational convenience.  We first show that the value function $W$ is truncatable. 
\begin{proposition}\label{prop:Utrunc}
Suppose that $\mu,\sigma:[0,\infty)^2 \to [0,\infty)$ are Borel-measurable and bounded.  Let $z \in [0,\infty)^2$ and define $V(x) = W_{\mu,\sigma}(z,x)$ for $x \in [z,\infty)$.   If $V$ is continuous then $V$ is a truncatable viscosity solution of 
\begin{equation}\label{eq:hjb-sub}
(V_{x_1} - \mu)_+(V_{x_2} - \mu)_+ =\sigma^2 \ \ \text{on } (z,\infty).
\end{equation}
\end{proposition}
\begin{proof}
It follows from Theorem \ref{thm:hjb} that $V$ is a viscosity solution of \eqref{eq:hjb-sub}.  We need only show that $V$ is truncatable.   Let $\xi \in (z,\infty)$, let $\chi:[0,\infty)^2\to \{0,1\}$ denote the characteristic function of $[z,\xi]$, and set $\bar{V} = W_{\chi\cdot \mu,\chi\cdot \sigma}(z,\cdot)$.  
By the definition of $\bar{V}$ and $\chi$ we have $\bar{V}(x) = V(x)  = V^\xi(x)$ for any $x \in [z,\xi]$.  Let $x \in [z,\infty)\setminus [z,\xi]$, $\eps>0$, and let $\gamma \in \A$ with $\gamma(0)=z$, $\gamma(1)=x$ such that $\bar{V}(x) \leq J_{\chi \cdot \mu,\chi\cdot \sigma}(\gamma) + \eps$. 
Let $\gamma^1$ denote the portion of $\gamma$ inside $[z,\xi]$, let $\gamma^2$ denote the remaining portion of $\gamma$, and reparametrize $\gamma^1$ and $\gamma^2$ so that $\gamma^1,\gamma^2:[0,1]\to\R^2$.  Letting $y=\gamma^1(1) \in [z,\xi]$ we have
\[\bar{V}(x) \leq J_{\chi\cdot \mu,\chi\cdot\sigma}(\gamma^1) + J_{\chi\cdot \mu,\chi\cdot\sigma}(\gamma^2) + \eps = J_{\mu,\sigma}(\gamma^1) + \eps \leq V(y) + \eps.\]
Since $y \leqq x$ and $y \in [z,\xi]$, we also have $\bar{V}(x) \geq \bar{V}(y)=V(y)$.  It follows that
\[\bar{V}(x) = \sup_{y \in [z,\xi] \, : \, y \leqq x} V(y).\]
By continuity of $V$, the supremum above is attained, and the maximizing argument of $y$ is exactly $y=\pi_\xi(x)$---the projection of $x$ onto $[0,\xi]$.  Therefore we have $\bar{V}(x) = V(\pi_\xi(x))$. Since $x$ is arbitrary, we see that $\bar{V} = V \circ \pi_\xi = V^\xi$, the $\xi$-truncation of $V$.

 Since $V^\xi=V\circ \pi_\xi$ is continuous, it follows from Theorem \ref{thm:hjb} that $V^\xi$ is a viscosity solution of 
\[(V_{x_1} - \chi\mu)_+(V_{x_2} - \chi\mu)_+ \leq\chi\sigma^2 \ \ \text{on } (z,\infty).\]
Since $0 \leq \chi \leq 1$ and $t \mapsto (p_1-t)_+(p_2 - t)_+$ is monotone decreasing, it follows that $V^\xi$ is viscosity subsolution of \eqref{eq:hjb-sub}, which completes the proof.
\end{proof}
We now show that truncatability enjoys a useful $L^\infty$-stability property.
\begin{proposition}\label{prop:trunc}
Let $z \in \R^2_+$ and for each $k\geq 1$ suppose that $u_k \in C([z,\infty))$ is a truncatable viscosity solution of 
\begin{equation}\label{eq:hjb-k}
H_k(x,Du_k) \leq 0 \ \ \text{on } (z,\infty).
\end{equation}
If  $u_k \to u$ locally uniformly, for some $u \in C([z,\infty))$, then $u$ is a truncatable viscosity solution of
\begin{equation}\label{eq:limit-hjb}
\underbar{H}\,(x,Du) \leq 0 \ \ \text{on } (z,\infty),
\end{equation}
where
\[\underbar{H}\,(x,p) := \liminf_{\substack{k\to \infty \\ y \to x}} H_k(y,p).\]
\end{proposition}
We should note that the $\liminf$ operation defining $\underbar{H}$ is taken jointly as $k\to \infty$ and $y \to x$. This is a standard operation in the theory of viscosity solutions (see~\cite[Section 6]{crandall1992}), and it can be written more precisely for a function $f:\O \to \R$ as
\[\liminf_{\substack{k\to \infty \\ y \to x}} f_k(y) = \lim_{j\to \infty} \inf\left\{f_k(x) \, : \, k \geq j, x \in \O \ \text{ and } |y-x|\leq \frac{1}{j} \right\}. \]
\begin{proof}
It is a standard result (see \cite[Remark~6.3]{crandall1992}) that $u$ is a viscosity solution of \eqref{eq:limit-hjb}. To see that $u$ is truncatable: Fix $\xi \in (z,\infty)$,  let $u^\xi$ be the $\xi$-truncation of $u$, and let $u^\xi_k$ be the $\xi$-truncation of $u_k$. Since $u_k$ is truncatable, we have that $u^\xi_k$ is a viscosity solution of \eqref{eq:hjb-k} for every $k$.  Furthermore, we have $u^\xi_k \to u^\xi$ locally uniformly, and therefore $u^\xi$ is a viscosity solution of \eqref{eq:limit-hjb}.  Thus  $u$ is truncatable.
\end{proof}
We now relax (H2) and allow $H$ to have discontinuous spatial dependence.  Given a set $\O \subset \R^d_+$ we assume $H$ satisfies
\begin{itemize}
\item[(H3)${}_\O$] There exists a modulus of continuity $m$ such that for all $\xi \in \O$ there exists $\eps_\xi > 0$ and $\v_\xi \in \S^{d-1}$ such that
\begin{equation}\label{eq:one-sided-cont}
H(y,p) - H(y+\eps\v,p) \leq m(|p|\eps + \eps)
\end{equation}
for all $p \in \R^d$, $y \in B_{\eps_\xi}(\xi)$, $\eps\in (0,\eps_\xi)$, and $\v \in \S^{d-1}$ with $|\v-\v_\xi| < \eps_\xi$.
\end{itemize}
This hypothesis is similar to one used by Deckelnick and Elliott~\cite{deckelnick2004} to prove uniqueness of viscosity solutions to Eikonal-type Hamilton-Jacobi equations with discontinuous spatial dependence.  It is also a generalization of the cone condition used in our previous work~\cite{calder2014}.

If we assume the subsolution is truncatable, then we can prove the following comparison principle, which holds for Hamiltonians $H$ with discontinuous spatial dependence.
\begin{theorem}\label{thm:comp-trunc}
Suppose that $H$ satisfies (H3)${}_\O$ for some $\O \subset \R^d_+$. Let $u \in C([0,\infty)^d)$ be a truncatable viscosity solution of \eqref{eq:u-sub} and  let $v \in C([0,\infty)^d)$ be a monotone  viscosity solution of \eqref{eq:v-super}.  Suppose that $u\leq v$ on $[0,\infty)^d \setminus \O$.
Then  $u \leq v$ on $\R^d_+$.
\end{theorem}
The proof of Theorem \ref{thm:comp-trunc} is similar to \cite[Theorem 2.8]{calder2014}, so we postpone it to the appendix.

For the remainder of the section we set
\begin{equation}\label{eq:Hmine}
H(x,p) = (p_1 - \mu(x))_+(p_2 - \mu(x))_+ - \sigma^2(x).
\end{equation}
  Our aim now is to apply the comparison principles from Theorems \ref{thm:comp} and \ref{thm:comp-trunc} to obtain a comparison principle, and a perturbation result, for the Hamilton-Jacobi equation (P).  First we need to show that (H2) and (H3)${_\O}$ are satisfied by $H$ given in \eqref{eq:Hmine}.
\begin{proposition}\label{prop:Hcont}
Suppose that  $\mu,\sigma:[0,\infty)^2 \to [0,\infty)$, and let $H$ be given by \eqref{eq:Hmine}.   Then for any $x,y \in \R^2_+$ 
\begin{equation}\label{eq:Hcont}
H(y,p) - H(x,p) \leq 2|p|(\mu(x)-\mu(y))_+ + \sigma^2(x)-\sigma^2(y).
\end{equation}
\end{proposition}
\begin{proof}
Let $p \in [0,\infty)^2$, and set $h(t)=(p_1-t)_+ (p_2-t)_+$ so that
\[H(x,p) = h(\mu(x)) - \sigma^2(x).\]
Suppose first that $\mu(y) < \min(p_1,p_2)$.  Since $h$ is convex, we have
\[h(\mu(x)) - h(\mu(y)) \geq h'(\mu(y)) (\mu(x)-\mu(y)) =  -(p_1 + p_2 - 2\mu(y))(\mu(x)-\mu(y)).\]
Since $p_1 + p_2 - 2\mu(y) \geq 0$ we have
\begin{align*}
h(\mu(y)) - h(\mu(x)) &\leq (p_1 + p_2 - 2\mu(y))(\mu(x)-\mu(y)) \\
&\leq (p_1 + p_2 - 2\mu(y))(\mu(x)-\mu(y))_+ \\
&\leq (p_1+p_2)(\mu(x)-\mu(y))_+.
\end{align*}
Therefore we have
\begin{equation}\label{eq:almost}
h(\mu(y)) - h(\mu(x)) \leq 2|p|(\mu(x)-\mu(y))_+.
\end{equation}
If $\mu(y) \geq \min(p_1,p_2)$ then we have $h(\mu(y)) = 0 \leq  h(\mu(x))$, and hence \eqref{eq:almost} holds.
\end{proof}
\begin{remark}\label{rem:H2}
It follows from Proposition \ref{prop:Hcont} that $H$ satisfies (H2) if  $\mu$ and $\sigma^2$ are globally Lipschitz continuous on $\R^2_+$.
\end{remark}
\begin{corollary}\label{cor:comp}
Suppose that $\mu$ and $\sigma^2$ are non-negative and globally Lipschitz continuous on $\R^2_+$.  Let $u\in \text{USC}([0,\infty)^2)$ be a viscosity solution of
\begin{equation}\label{eq:specific-sub}
(u_{x_1} - \mu)_+ (u_{x_2} - \mu)_+ \leq {\sigma}^2 \ \ \text{on } \R^2_+,
\end{equation}
and let $v \in \text{LSC}([0,\infty)^2)$ be a monotone viscosity solution of
\begin{equation}\label{eq:specific-super}
(v_{x_1} - \mu)_+ (v_{x_2} - \mu)_+ \geq {\sigma}^2 \ \ \text{on } \R^2_+.
\end{equation}
Furthermore, suppose that 
\begin{equation}\label{eq:zeros}
\big\{ x\in \R^2_+ \, : \, \mu(x)=0 \big\} \supset \big\{ x \in \R^2_+ \, : \, \sigma(x)=0\big\}.
\end{equation}
Then  $u \leq v$ on $\partial \R^2_+$ implies $u \leq v$ on $\R^2_+$.  
\end{corollary}
\begin{proof}
We claim that
\begin{equation}\label{eq:min}
\min(v_{x_1},v_{x_2}) \geq \mu \ \ \text{on } \R^2_+,
\end{equation}
in the viscosity sense.  To see this, let $x \in \R^2_+$ and let $p \in D^-v(x)$.  Then we have
\[(p_1 - \mu(x))_+ (p_2 - \mu(x))_+ \geq \sigma(x)^2.\]
If $\sigma(x)>0$, then we must have $\min(p_1,p_2)\geq \mu(x)$ as desired.  If $\sigma(x) = 0$, then by \eqref{eq:zeros} we have $\mu(x)=0$, and we have $\min(p_1,p_2)\geq 0=\mu(x)$ by virtue of the monotonicity of $v$.

Let $a >0$ and set $\bar{v}(x) = v(x) + \sqrt{a}(x_1 + x_2)$.  By \eqref{eq:specific-super} and \eqref{eq:min} we see that $\bar{v}$ is a viscosity solution of
\[(\bar{v}_{x_1} - \mu)_+ (\bar{v}_{x_2} - \mu)_+ \geq {\sigma}^2 + a  \ \ \ \text{on } \R^2_+.\]
By Proposition \ref{prop:Hcont} and Remark \ref{rem:H2} we see that (H1) and (H2) are satisfied.  Therefore we can apply Theorem \ref{thm:comp} to find that $u \leq \bar{v}$.  Sending $a \to 0$ completes the proof.
\end{proof}

Recall that $\mu$ and $\sigma^2$ are not independent functions in the DLPP problem, even though we have treated them as such for much of the analysis.  From this point on, we will need to recall their relationship, as it is important for proving uniqueness in (P). Specifically, we need to assume that $\mu$ and $\sigma^2$ satisfy (F3) for the \emph{same} choice of $\zeta$ at each $x \in \Gamma_i$. When this holds, we say that $\mu$ and $\sigma^2$ \emph{simultaneously} satisfy (F3).  Since $\sigma=\mu$ for exponential DLPP and $\sigma=\sqrt{\mu(1+\mu)}$ for geometric DLPP, $\sigma$ is always a monotone increasing function of $\mu$, and hence $\mu$ and $\sigma^2$ simultaneously satisfy (F3) in both cases. We recall that $\Omega$, $\Omega_i$, $\Gamma_i$, and (F1)--(F3) are defined in Section \ref{sec:results}, and that $\mu\equiv 0$ on $\Omega$.
\begin{proposition}\label{prop:hypo}
Let $\mu$ and $\sigma^2$ simultaneously satisfy (F1) and (F3). Then $H$ given by \eqref{eq:Hmine} satisfies (H3)${}_\O$ with $\O = \R^2_+ \setminus \bar{\Omega}$.
\end{proposition}
\begin{proof}
Let $\xi \in \O$.   If $\xi \in \Omega_i$, then we can choose $\eps_\xi$ small enough so that $B_{2\eps_\xi}(\xi) \subset \Omega_i$.  By Proposition \ref{prop:Hcont} we see that any choice for $\v_\xi$ will suffice since $\mu$ and $\sigma^2$ are Lipschitz with constant $C_{lip}$ when restricted to $\Omega_i$.  

If $\xi \in \Gamma_i$ for some $i$, then let $\zeta$ be as given in (F3).  Assume for now that $\zeta=-1$, and set $\v_\xi = (1,-1)/\sqrt{2}$.  Let $\eps_\xi>0$ be less than half the value of $\eps$ from (F3), and then choose $\eps_\xi>0$ smaller, if necessary, so that $B_{2\eps_\xi}(\xi)$ has an empty intersection with $\Gamma$ and all other $\Gamma_j$, and $\eps_{\xi}\leq 1/2$.  Let $\mu_i$ and $\sigma^2_i$ denote the Lipschitz extensions of $\mu\vert_{\Omega_i}$ and $\sigma^2\vert_{\Omega_i}$ to $\bar{\Omega_i}$, respectively, and make the same definitions for $\mu_{i-1}$ and $\sigma^2_{i-1}$.  Then (F3) implies that $\mu_i \geq \mu_{i-1}$ and $\sigma^2_i \geq \sigma^2_{i-1}$ on $B_{2\eps_\xi}(\xi) \cap \Gamma_i$.  Furthermore, since $\mu$ and $\sigma^2$ are upper semicontinuous, we have $\mu=\mu_i$ and $\sigma=\sigma_i$ on $B_{2\eps_\xi}(\xi) \cap \Gamma_i$.  

Let $y \in B_{\eps_\xi}(\xi)$, $\eps < \eps_\xi$, $p \in \R^2$, and $\v \in \S^{d-1}$ with $|\v-\v_\xi| < \eps_\xi$. If $y+\eps\v \in \bar{\Omega_{i}}$, then since $\Gamma_i$ is monotone, $|\v-\v_\xi|\leq \frac{1}{2}$, and $y\in B_{2\eps_\xi}(\xi)$, we must have that $y\in \Omega_{i}$.  Since $\mu_i$ and $\sigma_i^2$ are Lipschitz on $\bar{\Omega_{i}}\cap B_{2\eps_\xi}(\xi)$, we can invoke Proposition \ref{prop:Hcont} to show that  (H3)${_\O}$ holds.  

Now suppose that $y + \eps\v \in \Omega_{i-1}$.  If $y \in \Omega_{i-1}$, then (H3)${_\O}$ holds as before, so assume that $y \in \bar{\Omega_{i}}$. Let $\eps'>0$ such that $y + \eps'\v \in \Gamma_i$. Then we have 
\begin{align*}
  \mu(y + \eps\v) - \mu(y) &= \mu_{i-1}(y + \eps\v) - \mu_i(y + \eps'\v) + \mu_i(y + \eps'\v) - \mu_i(y) \\
  &\leq \mu_{i-1}(y + \eps\v) - \mu_{i-1}(y + \eps'\v)  + \mu_i(y + \eps'\v) - \mu_i(y) \\
  &\leq 2C_{lip} \eps,
\end{align*}
where we used the fact that $\mu_i \geq \mu_{i-1}$ on $\Gamma_i \cap B_{2\eps_\xi}(\xi)$.  We have an identical estimate for $\sigma^2$, and the proof is completed by invoking Proposition \ref{prop:Hcont}.
\end{proof}
\begin{corollary}\label{cor:comp-trunc}
Let $\mu$ and $\sigma^2$ simultaneously satisfy (F1) and (F3). Let $u\in C([0,\infty)^2)$ be a truncatable viscosity solution of \eqref{eq:specific-sub}, let $v \in C([0,\infty)^2)$ be a monotone viscosity solution of \eqref{eq:specific-super}, and suppose that \eqref{eq:zeros} holds.
Then $u \leq v$ on $\Omega \cup \partial \R^2_+$ implies $u\leq v$ on $\R^2_+$.
\end{corollary}
The proof of Corollary \ref{cor:comp-trunc} is similar to Corollary \ref{cor:comp}.

We now prove an important perturbation result.  Roughly speaking, it says that if we smooth out the macroscopic mean $\mu$ and variance $\sigma$ (i.e., remove the discontinuities), then the resulting change in the  value function $W$ is uniformly small. This result is used in the proof of our main result, Theorem \ref{thm:main}. The proof relies on the uniqueness of truncatable viscosity solutions of (P) (Theorem \ref{thm:comp-trunc} and Corollary \ref{cor:comp-trunc}), and the result can then be used to prove a comparison principle for (P) without the truncatability assumption (see Theorem \ref{thm:final-comp}).  
\begin{theorem}\label{thm:perturbation}
Let $\mu$ and $\sigma^2$ satisfy \eqref{eq:zeros} and simultaneously satisfy (F1), (F3).  Let $\mu_k,\sigma^2_k \in C^{0,1}([0,\infty)^2)$ satisfy (F1*) with $\theta = \frac{1}{k}$.  Furthermore suppose that
\begin{equation}\label{eq:gamma-mu}
\mu_*(x) \leq \liminf_{\substack{k \to \infty \\ y\to x}} \mu_k(y), \ \ \mu^*(x) \geq \limsup_{\substack{k \to \infty \\ y \to x}}\mu_k(y),
\end{equation} 
and
\begin{equation}\label{eq:gamma-sigma}
\sigma_*(x) \leq \liminf_{\substack{k \to \infty \\ y\to x}} \sigma_k(y), \ \ \sigma^*(x) \geq \limsup_{\substack{k \to \infty \\ y \to x}}\sigma_k(y),
\end{equation} 
for all $x \in \R^2_+$.   Then for every $z \in [0,\infty)^2$ we have 
\[W_{\mu_k,\sigma_k}(z,\cdot)\longrightarrow W_{\mu,\sigma}(z,\cdot) \ \  \text{locally uniformly on} \ [z,\infty).\]
\end{theorem}
\begin{proof}
For simplicity, let us set $V_k(x) = W_{\mu_k,\sigma_k}(z,x)$ and $V(x)=W_{\mu,\sigma}(z,x)$ for $x \in [z,\infty)$.
Since $\mu_k,\sigma^2_k \in C^{0,1}([0,\infty)^2)$, we can apply Theorem \ref{thm:reg} with $\theta = 0$ to find that $V_k$ is continuous on $[z,\infty)$. We can apply Theorem \ref{thm:reg} again with $\theta = 1/k$ to show that for every $R>\max(z_1,z_2)$, there exists $C=C(C_{lip},\|\mu\|_\infty,\|\sigma\|_\infty,R)$ and a modulus of continuity $\omega$  such that
\begin{equation}\label{eq:regeps}
|V_k(x) - V_k(y)| \leq C(\sqrt{|x-y|} + \omega(|x-y|) + \omega(k^{-1})) 
\end{equation}
for all $x,y \in [z_1,R]\times[z_2,R]$. This approximate H\"older estimate is sufficient to apply a slightly modified version of the Arzel\`a-Ascoli theorem (see, for instance,~\cite[Theorem 2]{calder2013b}).  Therefore, by passing to a subsequence if necessary, there exists $v \in C([z,\infty))$ such that $V_k \to v$ locally uniformly on $[z,\infty)$.    By Proposition \ref{prop:Utrunc}, $V_k$ is a monotone truncatable viscosity solution of
\begin{equation}\label{eq:viscosity-eps}
(V_{k,x_1} - \mu_k)_+ (V_{k,x_2} - \mu_k)_+ = {\sigma_k}^2  \ \ \text{on } (z,\infty).
\end{equation}
Since $V_k \to v$ locally uniformly and \eqref{eq:gamma-mu}-\eqref{eq:gamma-sigma} hold, we can apply Proposition \ref{prop:trunc}, and classical results from the theory of viscosity solutions~\cite{crandall1992}, to find that $v$ is a monotone truncatable viscosity solution of
\begin{equation}\label{eq:viscosity}
(v_{x_1} - \mu)_+ (v_{x_2} - \mu)_+ = {\sigma}^2 \ \ \text{on } (z,\infty).
\end{equation}

We claim that $v=V$ on $\partial (z,\infty)$.  To see this: Let $x \in \partial (z,\infty)$, hence $x_i=z_i$ for some $i$.  Without loss of generality, assume that $x_1=z_1$.  
Then by \eqref{eq:Wbc1} and Fatou's lemma we have
\begin{align*}
v(x)=\lim_{k\to\infty}V_k(x) &= \lim_{k\to \infty} \int_{z_2}^{x_2} \mu_k(z_1,t) \, dt \\
&\leq \int_{z_2}^{x_2} \limsup_{k\to\infty} \mu_k(z_1,t) \, dt\\
&\leq \int_{z_2}^{x_2} \mu(z_1,t) \, dt= V(x),
\end{align*}
where the last line follows from \eqref{eq:gamma-mu} and the fact that $\mu$ is upper semicontinuous.  By a similar argument with Fatou's lemma we have
\begin{equation}\label{eq:fatou}
  v(x) \geq \int_{z_2}^{x_2} \mu_*(z_1,t) \, dt.
\end{equation}
Notice that (F1) implies that $\mu_*=\mu$ on $\Omega_i$ for all $i$ and on $\Omega$.  Hence, all the points $x\in [0,\infty)^2$ for which $\mu_*(x)\neq \mu(x)$ are contained in $\cup_{i \in \Z} \Gamma_i\cup\Gamma$.  Since the curves $\Gamma_i$ are strictly increasing and $\Gamma$ is strictly decreasing, the curve $t \mapsto (z_1,t)$ for $t\in [z_2,x_2]$ has a finite number of intersections with $ \cup_{i \in \Z} \Gamma_i\cup\Gamma$.  It follows that
\[v(x) \stackrel{\eqref{eq:fatou}}{\geq}\int_{z_2}^{x_2} \mu_*(z_1,t) \, dt = \int_{z_2}^{x_2} \mu(z_1,t) \, dt = V(x),\]
and hence $v(x)=V(x)$, which establishes the claim.

By Proposition \ref{prop:hypo}, $H$ given by \eqref{eq:Hmine} satisfies (H3)${}_\O$ for $\O = \R^2_+ \setminus{\bar{\Omega}}$.  By (F1*) and \eqref{eq:vardef} we have $V_k(x) = 0$ for $x \in \bar{\Omega_\theta}\cap[z,\infty)$, and hence $v(x)=0$ for $x \in \bar{\Omega}\cap[z,\infty)$.  Similarly, we have that $V(x)=0$ for $x \in \bar{\Omega} \cap [z,\infty)$.  It follows that $v = V$ on $[z,\infty) \setminus \O$, and by applying  a translated form of Corollary \ref{cor:comp-trunc} to find that $v = V$ on $[z,\infty)^2$.
\end{proof}
\begin{remark}\label{rem:conv}
Sequences generated by inf-~and sup-convolutions of $\mu$ and $\sigma^2$ satisfy the hypotheses of Theorem \ref{thm:perturbation}. Recall that the sup-convolution of $\mu:[0,\infty)^2 \to \R$ is defined by
\begin{equation}\label{eq:sup-convolution}
\mu^k(x) = \sup_{y \in [0,\infty)^2} \Big\{ \mu(y) - k |x-y| \Big\},
\end{equation}
and the inf-convolution by $\mu_k:= -(-\mu)^k$.  
\end{remark}
\begin{corollary}\label{cor:perturbation}
Let $\mu$ and $\sigma^2$ simultaneously  satisfy (F1), (F3) and \eqref{eq:zeros}, let $\mu_k,\sigma^2_k \in C^{0,1}([0,\infty)^2)$ satisfy (F1*) with $\theta = \frac{1}{k}$, and let $\mu_s$ satisfy (F2). If   \eqref{eq:gamma-mu}--\eqref{eq:gamma-sigma} hold for all $x \in \R^2_+$ then
\[U_{\mu_k+\mu_s,\sigma_k} \longrightarrow U_{\mu+\mu_s,\sigma} \ \ \text{locally uniformly on } [0,\infty)^2.\]
\end{corollary}
\begin{proof}
Fix $y \in [0,\infty)^2$. By Proposition \ref{prop:dpp} we have
\begin{equation}\label{eq:dppUk}
U_{\mu_k+\mu_s,\sigma_k}(y) = \max_{x \in \partial \R^2_+ \, : \, x \leqq y} \Big\{ U_{\mu_k+\mu_s,\sigma_k}(x) + W_{\mu_k,\sigma_k}(x,y)\Big\},
\end{equation}
and
\begin{equation}\label{eq:dppU}
  U_{\mu+\mu_s,\sigma}(y) = \max_{x \in \partial \R^2_+ \, : \, x \leqq y} \Big\{ U_{\mu+\mu_s,\sigma}(x) + W_{\mu,\sigma}(x,y)\Big\}.
\end{equation}
Arguing by symmetry, it follows from Theorem \ref{thm:perturbation} that 
\begin{equation}\label{eq:symm}
W_{\mu_k,\sigma_k}(\cdot,y) \longrightarrow W(\cdot,y)  \ \ \text{uniformly on } \ [0,y].
\end{equation}
It follows from \eqref{eq:Ubc1} and a similar argument as in Theorem \ref{thm:perturbation} that $U_{\mu_k+\mu_s,\sigma_k}(x) \to U_{\mu+\mu_s,\sigma}(x)$ for any $x \in \partial \R^2_+$.  By the Arzel\`a-Ascoli Theorem we find that
\begin{equation}\label{eq:conv2}
U_{\mu_k+\mu_s,\sigma_k} \longrightarrow U_{\mu+\mu_s,\sigma} \ \ \text{uniformly on } \ [0,y] \cap \partial \R^2_+.
\end{equation}
Combining \eqref{eq:dppUk}--\eqref{eq:conv2}, we have that $U_{\mu_k+\mu_s,\sigma_k}(y) \to U_{\mu+\mu_s,\sigma}(y)$.  Locally uniform convergence follows again from the Arzel\`a-Ascoli Theorem.
\end{proof}

\begin{theorem}\label{thm:final-comp}
Let $\mu$ and $\sigma^2$ simultaneously  satisfy (F1), (F3) and \eqref{eq:zeros}, and let $\mu_s$ satisfy (F2).  Let $u\in C([0,\infty)^2)$ be a viscosity solution of \eqref{eq:specific-sub} and  let $v \in C([0,\infty)^2)$ be a monotone viscosity solution of \eqref{eq:specific-super}.
  Then if $u \leq \phi \leq v$ on $\partial \R^2_+$, where $\phi$ is given in the statement of Theorem \ref{thm:main}, then $u\leq v$ on $\R^2_+$.
\end{theorem}
\begin{proof}
  Let $\mu^k,\sigma^{2,k}$ and $\mu_k,\sigma^2_k$ be the sup-~and inf-convolutions of $\mu$ and $\sigma^2$ as defined in \eqref{eq:sup-convolution} (see  Remark \ref{rem:conv}), respectively.   To simplify notation, let us write $U^k := U_{\mu^k+\mu_s,\sigma^k}$, $U_k := U_{\mu_k+\mu_s,\sigma_k}$, and $U:=U_{\mu+\mu_s,\sigma}$. By definition we have $U_k\leq U\leq U^k$, and by Corollary \ref{cor:perturbation} and Remark \ref{rem:conv} we have $U_k,U^k \to U$ locally uniformly on $[0,\infty)^2$ as $k \to \infty$.

Since $\mu_k \leq \mu$ and $\sigma_k \leq \sigma$ we have that $v$ is a viscosity solution of
\[(v_{x_1} - \mu_k)_+ (v_{x_2} - \mu_k)_+ \geq {\sigma_k}^2 \ \ \text{on } \R^2_+.\]
By Theorem \ref{thm:hjb}, $U_k$ is a viscosity solution of
\[(U_{k,x_1} - \mu_k)_+ (U_{k,x_2} - \mu_k)_+ = {\sigma_k}^2  \ \  \text{on } \R^2_+.\]
Furthermore, we have $U_k = \phi_k \leq \phi \leq v$ on $\partial \R^2_+$ where $\phi_k(x) = (x_1+x_2)\int_0^1 \mu_k(tx) + \mu_s(tx) \, dt$.
Since $\mu_k$ and $\sigma^2_k$ are globally Lipschitz we can apply Corollary \ref{cor:comp} to obtain $U_k \leq v$.  Sending $k \to \infty$ we have $U \leq v$. By a similar argument we can prove that $u \leq U$, which completes the proof.
\end{proof}

\section{Proof of main result}
\label{sec:conv}

In this section we give the proof of our main result, Theorem \ref{thm:main}.  We first have a preliminary convergence result on the interior $(0,\infty)^2$, which we later adapt to account for the boundary source $\mu_s$.  
For  $N\geq 1$ we define
\begin{equation}\label{eq:wn}
w_N(x,y) := L\Big(\lfloor Nx\rfloor + \vb{1}_x; \lfloor Ny \rfloor \Big),
\end{equation}
where
\begin{equation}\label{eq:1x}
\vb{1}_x = \big(1_{\{x_1 =0\}},1_{\{x_2=0\}}\big),
\end{equation}
and $L$ is defined in \eqref{eq:dlpp}.

\begin{lemma}\label{lem:prelim-conv}
Assume $\mu$ satisfies (F1) and (F3). Suppose that the weights $X(i,j)$ satisfy \eqref{eq:weights-assumption} and are either all exponential, or all geometric random variables, consructed as in Section \ref{sec:results}. In the exponential case, set $\sigma = \mu$, and in the geometric case, set $\sigma = \sqrt{\mu(1+\mu)}$.  
 Then for every $y \in (0,\infty)^2$  we have 
\[\frac{1}{N} w_N(\cdot,y) \longrightarrow W_{\mu,\sigma}(\cdot,y) \ \ \text{uniformly on } [0,y],\] 
with probability one.
\end{lemma}
\begin{proof}
  Let $y \in (0,\infty)^2$. Let $\mu^k$ and $\mu_k$ be the sup-~and inf-convolutions of $\mu$, defined in \eqref{eq:sup-convolution} (see Remark \ref{rem:conv}).  In the exponential case, set $\sigma^k=\mu^k$ and $\sigma_k = \mu_k$, and in the geometric case, set $\sigma^k=\sqrt{\mu^k(1+\mu^k)}$ and  $\sigma_k=\sqrt{\mu_k(1+\mu_k)}$.  To simplify notation, let us also set $W^k:=W_{\mu^k,\sigma^k}$, $W_k:=W_{\mu_k,\sigma_k}$, and $W:=W_{\mu,\sigma}$, and note that $W_k \leq W \leq W^k$. Notice that by the definition of $\sigma$, we have that \eqref{eq:zeros} holds for both the exponential and geometric cases. We can therefore invoke Theorem \ref{thm:perturbation} to find that
\begin{equation}\label{eq:Wconv}
W_k(x,y)\longrightarrow W(x,y) \ \ \text{and} \ \ W^k(x,y) \longrightarrow W(x,y) \ \ \text{for all } x \in [0,y].
\end{equation} 

Let $N\geq 1$.   In the exponential case,  for $(i,j) \in \N^2$ let $X^k(i,j)$ be independent and exponentially distributed with parameter $\lambda = \mu^k(iN^{-1},jN^{-1})$, and let $X_k(i,j)$ be independent and exponentially distributed with parameter $\lambda = \mu_k(iN^{-1},jN^{-1})$. In the geometric case, for $(i,j) \in \N^2$ let $X^k(i,j)$ be independent and geometrically distributed with parameter $q = (1+\mu^k(iN^{-1},jN^{-1}))^{-1}$, and let $X_k(i,j)$ be independent and geometrically distributed with parameter $q = (1+\mu_k(iN^{-1},jN^{-1}))^{-1}$.  In either case, set
\begin{equation}\label{eq:dlpp2}
L_k(M,N;Q,P) = \max_{p \in \Pi_{(M,N),(Q,P)}} \sum_{(i,j) \in p} X_k(i,j), 
\end{equation}
\begin{equation}\label{eq:dlpp3}
 L^k(M,N;Q,P) = \max_{p \in \Pi_{(M,N),(Q,P)}} \sum_{(i,j) \in p} X^k(i,j),
\end{equation}
and set
\begin{equation}\label{eq:dlpp4}
w_{k,N}(x,y):=L_k\Big(\lfloor Nx\rfloor + \vb{1}_x; \lfloor Ny \rfloor \Big), \ \ \text{and } \ \ w^k_N(x,y):=L^k\Big(\lfloor Nx\rfloor + \vb{1}_x; \lfloor Ny \rfloor \Big).
\end{equation}
We can define $X_k(i,j)$ and $X^k(i,j)$ on the same probability space as $X(i,j)$ in such a way that $X_k(i,j) \leq X(i,j) \leq X^k(i,j)$ for all $(i,j) \in \N^2$ with probability one. We therefore have $w_{k,N} \leq w_{N} \leq w^k_N$ with probability one.  Since $\mu_k,\sigma_k,\mu^k,$ and $\sigma^k$ are continuous on $[0,\infty)^2$, we can invoke Theorem \cite[Theorem 1]{rolla2008} to find that
\[\frac{1}{N}w_{k,N}(x,y) \longrightarrow W_k(x,y) \ \ \text{ and } \ \ \frac{1}{N}w^k_N(x,y) \longrightarrow W^k(x,y),\]
with probability one, for fixed $x \in [0,y]$. We should note that \cite[Theorem 1]{rolla2008} as stated applies only to exponential DLPP.  The proof for geometric DLPP (with weights constructed as in Section \ref{sec:results}) is very similar, with only minor modifications.   It follows that for every $k\geq 1$ we have
\[W_{k}(x,y) \leq \liminf_{N\to \infty} \frac{1}{N} w_N(x,y)\leq \limsup_{N\to \infty} \frac{1}{N}w_N(x,y) \leq W^k(x,y),\]
with probability one. Sending $k\to \infty$ and recalling \eqref{eq:Wconv} we have for every $x \in [0,y]$ that
\begin{equation}\label{eq:pointwise}
\frac{1}{N}w_N(x,y) \longrightarrow W(x,y) \ \ \text{with probability one.}
\end{equation}
Uniform convergence follows from the fact that $x\mapsto w_N(x,y)$ and $x\mapsto W(x,y)$ are monotone decreasing and $x \mapsto W(x,y)$ is uniformly continuous on $[0,y]$; the proof is similar to \cite[Theorem 1]{calder2014}. 
\end{proof}

To incorporate the boundary source $\mu_s$ we need the following lemma, which follows from the law of large numbers.
\begin{lemma}\label{lem:slln}
Let $Y_1,\dots,Y_n,\dots$ be a sequence of \iid~exponential random variables with mean $\lambda=1$. Let $\nu:[0,\infty) \to [0,\infty)$ be  bounded with a locally finite set of discontinuities, and let $f:[0,\infty) \to [0,\infty)$ be non-decreasing with at most polynomial growth. Then we have with probability one that
\begin{equation}\label{eq:slln}
\frac{1}{n} \sum_{i=1}^n f(\nu(n^{-1}i) Y_i) \longrightarrow \int_0^1\E(f(\nu(t)Y)) \, dt \ \ \text{ as } n \to \infty,
\end{equation}
where $Y$ is a random variable with the exponential distribution with mean $\lambda=1$.
\end{lemma}
Note that Lemma \ref{lem:slln} mimics the constructions of the weights $X(i,j)$ given in Section \ref{sec:results}. When $X(i,j)$ are exponential random variables, we have $f(t) =t$, and $\nu = \mu + \mu_s$, and when $X(i,j)$ are geometric random variables, $f(t) = \lfloor t \rfloor$ and $\nu$ is defined according to the construction in Section \ref{sec:results}.
\begin{proof}
Let $K$ be a positive integer. Consider the partition of $[0,1]$ given by $0=t_0 < t_1 < \cdots < t_{K-1} < t_K = 1$, where $t_j = j/K$, and let $k_j = \lfloor n t_j\rfloor$.  
Set $m_j = \inf_{(t_{j-1},t_j]} \nu$ and $M_j = \sup_{(t_{j-1},t_j]}\nu$. Then we have that 
\begin{equation}\label{eq:partition}
\frac{1}{n}\sum_{i=1}^n f(\nu(n^{-1}i) Y_i) = \frac{1}{n} \sum_{j=1}^K \sum_{i =k_{j-1}+1}^{k_j} f(\nu(n^{-1}i) Y_i)  \leq\sum_{j=1}^K \frac{1}{n} \sum_{i=k_{j-1}+1}^{k_j} f(M_j Y_i),
\end{equation}
where the last inequality follows from the monotonicity of $f$.
Fix $j$ and let $Z_i = f(M_j Y_i)$. Then $Z_1,\dots,Z_n,\dots$ are \iid, and the polynomial growth restriction on $f$ guarantees that the moments of $Z_i$ are finite. We therefore  have by the law of large numbers  that
\[\frac{1}{n} \sum_{i=1}^{k_j} Z_i = \left(\frac{\lfloor n t_j\rfloor}{n}\right) \frac{1}{\lfloor n t_j\rfloor} \sum_{i=1}^{\lfloor nt_j \rfloor} Z_i\longrightarrow t_j\E(f(M_j Y)),\]
with probability one as $n\to \infty$.
Similarly, we have
\[\frac{1}{n} \sum_{i=1}^{k_{j-1}} Z_i \longrightarrow t_{j-1}\E(f(M_j Y)),\]
with probability one as $n\to \infty$.  It follows that
\[\frac{1}{n} \sum_{i=k_{j-1}+1}^{k_j} f(M_jY_i) = \frac{1}{n} \sum_{i=1}^{k_j} Z_i - \frac{1}{n} \sum_{i=1}^{k_{j-1}} Z_i \longrightarrow (t_j - t_{j-1}) \E(f(M_j Y)),\]
with probability one as $n\to \infty$. Since the above holds for every $j=1,\dots,K$, we have from \eqref{eq:partition} that
\[\limsup_{n\to \infty} \frac{1}{n}\sum_{i=1}^n f(\nu(n^{-1}i) Y_i) \leq \sum_{j=1}^K(t_j - t_{j-1}) \E(f(M_j Y)), \]
with probability one. By the assumptions on $f$ and $\nu$, $t \mapsto \E(f(\nu(t)Y))$ is continuous except possibly at points of discontinuity of $\nu$, which are locally finite. Thus $t \mapsto \E(f(\nu(t)Y))$ is Riemann integrable, and taking $K \to \infty$ we have
\[\limsup_{n\to \infty} \frac{1}{n}\sum_{i=1}^n f(\nu(n^{-1}i) Y_i) \leq \int_0^1 \E(f(\nu(t)Y)) \, dt,\]
with probability one. The proof of the analogous $\liminf$ inequality is similar.
\end{proof}

We now have the proof of Theorem \ref{thm:main}.
\begin{proof}
Let $x \in \partial \R^2_+$, and suppose that $x_2=0$. If $x_1 =0$, then $N^{-1}L(0;0) = N^{-1} X(0,0) \to 0 = \phi(0)$ with probability one as $N \to \infty$. If $x_1 > 0$  then we have
\[\frac{1}{N} L(0;\lfloor Nx\rfloor ) = \frac{1}{N} \sum_{i=0}^{\lfloor Nx_1\rfloor} X(i,0).\]
It follows from Lemma \ref{lem:slln} and the construction of the weights $X(i,j)$ in Section \ref{sec:results} that 
\[\frac{1}{N} L(0;\lfloor Nx \rfloor) \longrightarrow x_1 \int_0^1 \mu(x_1t,0) + \mu_s(x_1t,0) \, dt = \phi(x),\]
with probability one as $N \to \infty$. The case where $x_1=0$ and $x_2 > 0$ is similar. As in Lemma \ref{lem:prelim-conv}, we can use the fact that $L$ and $\phi$ are monotone non-decreasing, and $\phi$ is uniformly continuous, to show that we actually have
\begin{equation}\label{eq:conv1}
  \frac{1}{N} L(0;\lfloor N\cdot \rfloor) \longrightarrow U = \phi
\end{equation} 
locally uniformly on $\partial \R^2_+$ with probability one.

Let $y \in \R^2_+$.   From the definition of $L$ we have the following dynamic programming principle  
\begin{equation}\label{eq:dpp1}
L(0;\lfloor Ny\rfloor) = \max_{x \in \partial \R^2_+ \, : \, x \leqq y} \Big\{ L(0;\lfloor Nx\rfloor) + w_N(x,y) \Big\}.
\end{equation}
 Combining Lemma \ref{lem:prelim-conv}, Proposition \ref{prop:dpp}, and  \eqref{eq:conv1}, we can pass to the limit in \eqref{eq:dpp1} to obtain
\[\frac{1}{N}L(0;\lfloor Ny\rfloor)  \longrightarrow \max_{x \in \partial\R^2_+ \, : \, x \leqq y}\Big\{ U(x) + W(x,y) \Big\} = U(y),\]
with probability one. As in Lemma \ref{lem:prelim-conv}, locally uniform convergence follows from the monotonicity of $U$ and $x \mapsto N^{-1} L(0;\lfloor Nx \rfloor )$, along with the uniform continuity given by Theorem \ref{thm:reg}.
\end{proof}

\section{Numerical scheme}
\label{sec:num}

We present here a fast numerical scheme for computing the viscosity solution $U$ of (P).  The scheme is a minor modification of the scheme used in~\cite{calder2014,calder2013b}.  Since information propagates along coordinate axes in the definition of the variational problem \eqref{eq:vardef} for $U$, it is natural to consider using backward difference quotients to approximate (P).  Letting $U^h_{i,j}$ denote the numerical solution on the grid $h\N^2_0$ of spacing $h$, we have
\begin{equation}\label{eq:back-diff}
\Big( U^h_{i,j} - U^h_{i-1,j} - h \mu_{i,j} \Big)_+ \Big( U^h_{i,j} - U^h_{i,j-1} - h \mu_{i,j} \Big )_+ = h^2\sigma^2_{i,j},
\end{equation}
where $\mu_{i,j} = \mu(hi,hj) + \mu_s(hi,hj)$ and $\sigma_{i,j} = \sigma(hi,hj)$.   Given $U^h_{i-1,j}$ and $U^h_{i,j-1}$, we can solve \eqref{eq:back-diff}  for $U^h_{i,j}\geq \max(U^h_{i-1,j}+h\mu_{i,j},U^h_{i,j-1}+h\mu_{i,j})$ via the quadratic formula to obtain
\begin{equation}\label{eq:scheme1}
U^h_{i,j} = \frac{1}{2}\Big( U^h_{i-1,j} + U^h_{i,j-1} \Big) + h\mu_{i,j} + \frac{1}{2} \sqrt{\Big(U^h_{i-1,j} - U^h_{i,j-1} \Big)^2 + 4h^2\sigma^2_{i,j}},
\end{equation}
for $i,j\geq 1$.  The choice of the positive root in \eqref{eq:scheme1} reflects the monotonicity of the scheme, and ensures that it captures the viscosity solution of (P).  When $i=0$ or $j=0$, we recall the boundary condition \eqref{eq:Ubc1} to obtain
\begin{equation}\label{eq:scheme2}
U^h_{0,j} = U^h_{0,j-1} + h\mu_{0,j} \ \ \text{and} \ \ U^h_{i,0} = U^h_{i-1,0} + h\mu_{i,0}.
\end{equation}
Notice that when $i=0$, if we set $U^h_{-1,j}=0$ and $\sigma_{i,j}=0$ in \eqref{eq:scheme1}, then \eqref{eq:scheme1} and \eqref{eq:scheme2} are equivalent.  In fact, even when $\sigma_{i,j} \neq 0$, \eqref{eq:scheme1} and \eqref{eq:scheme2} are asymptotically equivalent as $h\to 0$ provided $U^h_{0,j} \gg h$.   The same observations hold when $j=0$ if we set $U^h_{i,-1}=0$.  
Thus, to account for  the boundary condition in (P), we can simply set
\begin{equation}\label{eq:discrete-bc}
U^h_{i,j}=0  \ \ \ \text{for } (i,j) \not\in \N^2_0,
\end{equation}
and compute $U^h_{i,j}$ via \eqref{eq:scheme1} for all $(i,j) \in \N^2_0\cap [0,R]^2$, for any $R>0$.  In summary, we propose the following numerical scheme for approximating viscosity solutions of (P):
\[\text{(S)} \left\{\begin{aligned}
    U^h_{\mbox{\scriptsize \textit{i,j}}}&= \frac{1}{2}\big( U^h_{\mbox{\scriptsize \textit{i--}1\textit{,j}}} + U^h_{\mbox{\scriptsize \textit{i,j--}1}} \big) + h\mu_{\mbox{\scriptsize \textit{i,j}}} + \frac{1}{2} \sqrt{\big(U^h_{\mbox{\scriptsize \textit{i--}1\textit{,j}}} - U^h_{\mbox{\scriptsize \textit{i,j--}1}} \big)^2 + 4h^2\sigma^2_{\mbox{\scriptsize \textit{i,j}}}},& \text{if } (i,j) \in \N^2_0\\
    U^h_{\mbox{\scriptsize \textit{i,j}}}&= 0,& \text{otherwise}.
\end{aligned}\right.\]
Note that we can visit the grid points in any sweeping pattern that visits $(i-1,j)$ and $(i,j-1)$ before $(i,j)$, which reflects  the cone of influence in the percolation problem.  This scheme requires visiting each grid point exactly once and hence has linear complexity.  

Our first result guarantees that the simple boundary condition in (S) agrees with the boundary condition in (P) as $h\to 0$. 
\begin{lemma}\label{lem:discrete-bc}
  Let $U^h_{i,j}$ satisfy the scheme (S) and suppose that $\sigma_{i,j}$ is bounded by $M$ for all $(i,j) \in \N^2_0\cap \partial \R^2_+$.  If $i,j \leq h^{-1}R$ then there exists a constant $C>0$ such that
\begin{equation}\label{eq:discrete-bcfinal}
\left|U^h_{i,0} - h\sum_{k=0}^{i} \mu_{k,0}\right|,\left|U^h_{0,j} - h\sum_{k=0}^{j} \mu_{0,k}\right| \leq  C(1  + RM^2)\sqrt{h}.
\end{equation}
\end{lemma}
\begin{proof}
Let us give the proof for $i=0$. The case of $j=0$ is similar.  Define
\[J:= \sup \Big\{ j\geq 0 \, : \, U^h_{0,j} \leq \sqrt{h}\Big\}.\]
For $j\geq J$ it follows from the scheme (S) and a Taylor expansion that
\begin{equation*}
U^h_{0,j} = \frac{1}{2} U^h_{0,j-1} + h\mu_{0,j} + \frac{1}{2}U^h_{0,j-1} +  O\Big(h^\frac{3}{2}M^2\Big) = U^h_{0,j-1} + h\mu_{0,j} + O\Big(h^\frac{3}{2}M^2\Big).
\end{equation*}
Iterating we have
\[U^h_{0,j} = h\left(\sum_{k=J+1}^{j} \mu_{0,k}\right) + U^h_{0,J} +O\Big(h^\frac{3}{2}jM^2\Big)=h\left(\sum_{k=J+1}^{j} \mu_{0,k}\right) + O\Big(\sqrt{h} + jh^\frac{3}{2}M^2\Big).\]
Since $j\leq h^{-1}R$ we have
\begin{equation}\label{eq:bc-discrete1}
U^h_{0,j} \leq h\left(\sum_{k=0}^{j} \mu_{0,k}\right) + O\Big(\big(1  + RM^2\big)\sqrt{h}\Big).
\end{equation}
Noting the equivalence of \eqref{eq:scheme1} and \eqref{eq:scheme2} when $\sigma_{0,j}=0$, we can set $\sigma_{0,j}=0$ in \eqref{eq:scheme1} and iterate as before to obtain
\[U^h_{0,j} \geq h\left(\sum_{k=0}^{j} \mu_{0,k}\right).\]
Combining this with \eqref{eq:bc-discrete1} completes the proof.
\end{proof}

\begin{theorem}\label{thm:lip-conv}
  Suppose that $\mu$ and $\sigma^2$ are non-negative, globally Lipschitz continuous on $[0,\infty)^2$ and satisfy \eqref{eq:zeros}, and let $\mu_s$ satisfy (F2).  For $h>0$ let $U^h(x)=U_{\lfloor h^{-1}x_1\rfloor,\lfloor h^{-1}x_2\rfloor}$ denote the extension of the numerical solution $U^h_{i,j}$ of (S) to $[0,\infty)^2$.  Then we have
\begin{equation}\label{eq:scheme-conv}
U^h \longrightarrow U \ \ \ \text{locally uniformly on } [0,\infty)^2,
\end{equation}
where $U$ is the unique monotone viscosity solution of (P).
\end{theorem}
\begin{proof}
The proof follows the standard framework outlined by Barles and Souganidis~\cite{barles1991}. This general theory guarantees convergence of any scheme that is monotone, stable, and consistent, provided the PDE enjoys strong uniqueness---a comparison principle for semicontinuous sub- and supersolutions. Corollary \ref{cor:comp} is the required strong uniqueness result, and it is easy to see that the scheme \eqref{eq:back-diff} is both monotone and consistent. Indeed, for any $\psi \in C^1([0,\infty)^2)$ we have
\begin{align*}
&\frac{1}{h^2}\Big( \psi(x) - \psi(x-he_1) - h \mu(x) \Big)_+ \Big( \psi(x) - \psi(x-he_2) - h\mu(x)  \Big )_+ \\
&\hspace{2in}\longrightarrow \Big(\psi_{x_1}(x) - \mu(x)\Big)_+ \Big(\psi_{x_2}(x) - \mu(x)\Big)_+,
\end{align*}
as $h \to 0$, which is the required consistency. To show monotonicity, let $u,v:[0,\infty)^2$ such that $u(x)=v(x)$ and $u \leq v$. Then we have
\begin{align*}
&\Big( v(x) - v(x-he_1) - h \mu(x) \Big)_+ \Big( v(x) - v(x-he_2) - h\mu(x)  \Big )_+ \\
&\hspace{2cm}= \Big( u(x) - v(x-he_1) - h \mu(x) \Big)_+ \Big( u(x) - v(x-he_2) - h\mu(x)  \Big )_+ \\
&\hspace{2cm}\leq \Big( u(x) - u(x-he_1) - h \mu(x) \Big)_+ \Big( u(x) - u(x-he_2) - h\mu(x)  \Big )_+,
\end{align*}
where the last line follows from the monotonicity of $t \mapsto (p_1-t)_+(p_2-t)_+$.

Therefore, to complete the proof, we need to show that the scheme is stable, and that the boundary condition is satisfied.  Stability refers to a bound on $U^h$, independent of $h$.  By Lemma \ref{lem:discrete-bc}, (F2), and the continuity of $\mu$, we have that
\begin{equation}\label{eq:bc-conv}
  U^h \longrightarrow \phi \ \ \ \text{locally uniformly on } \partial \R^2_+ \ \ \text{as } h \to 0,
\end{equation}
where $\phi(x) = (x_1+x_2)\int_0^1 \mu(tx) + \mu_s(tx) \, dt$, which verifies the boundary condition.

Stability follows from a comparison principle for (S), and is similar to~\cite[Lemma 3.3]{calder2013b}.  We give the argument here for completeness. Let 
\[V(x) = \|\mu+\mu_s\|_{\infty}(x_1 + x_2) + 2\|\sigma\|_\infty\sqrt{x_1x_2} + 1.\]
We claim that $U^h(x) \leq V(x)$.   To see this, suppose to the contrary that $U^h(x) > V(x)$ for some $x \in [0,R]^2$, $R>0$. First note that 
\[\phi(x) \leq (x_1 + x_2)\|\mu + \mu_s\|_\infty = V(x) - 1,\]
for $x \in  \partial \R^2_+$.  Therefore, by \eqref{eq:bc-conv}, we have that $U^h \leq V - \frac{1}{2}$ on $[0,R]^2\cap \partial \R^2_+$ for $h$ small enough. Therefore, there exists $z \in [h,R]^2$ such that
\begin{equation}\label{eq:z}
U^h(z) > V(z) \ \ \text{and} \ \ U^h(z-he_i) \leq V(z-he_i) \ \ \text{for } i=1,2.
\end{equation}
Note that by the concavity of $t \mapsto \sqrt{t}$ we have that
\[V(z) - V(z-he_i) \geq h\|\mu+\mu_s\|_\infty + h\|\sigma\|_\infty \frac{\sqrt{z_1z_2}}{z_i}.\]
It follows that
\[\Big( V(z) - V(z-he_1) - h \|\mu+\mu_s\|_\infty \Big) \Big( V(z) - V(z-he_2) - h\|\mu+\mu_s\|_\infty \Big ) \geq h^2\|\sigma\|_\infty^2.\]
By monotonicity of $t \mapsto (p_1-t)_+(p_2-t)_+$ we therefore have
\begin{align}\label{eq:tocontradict}
&\Big( V(z) - V(z-he_1) - h \mu(z) \Big) \Big( V(z) - V(z-he_2) - h\mu(z)  \Big ) \geq h^2\|\sigma\|_\infty^2 \notag \\
&\hspace{0.3in}\geq \Big( U^h(z) - U^h(z-he_1) - h \mu(z) \Big) \Big( U^h(z) - U^h(z-he_2) - h\mu(z)  \Big ).
\end{align}
This contradicts \eqref{eq:z}, hence $U^h \leq V$. The proof is completed by invoking~\cite[Theorem 2.1]{barles1991}.
\end{proof}
We now extend the numerical convergence result to $\mu$ $\sigma^2$ satisfying (F1) and (F3).
\begin{corollary}\label{cor:final-conv}
  Suppose that $\mu$ and $\sigma^2$ simultaneously satisfy (F1), (F3) and \eqref{eq:zeros}, and let $\mu_s$ satisfy (F2).  Define $U^h$ as in Theorem \ref{thm:lip-conv}.  Then we have
\begin{equation}\label{eq:scheme-conv2}
  U^h \longrightarrow U \ \ \ \text{locally uniformly on } [0,\infty)^2,
\end{equation}
where $U$ is the unique monotone viscosity solution of (P).
\end{corollary}
\begin{proof}
Define $\mu^k,\sigma^k,\mu_k,\sigma_k,U_k$ and $U^k$ as in the proof of Theorem \ref{thm:final-comp}. 
By definition we have $U_k\leq U\leq U^k$, and by Corollary \ref{cor:perturbation} and Remark \ref{rem:conv} we have $U_k,U^k \to U$ locally uniformly on $[0,\infty)^2$ as $k \to \infty$.

  Let $U^h_k$ and $U^{k,h}$ denote the numerical solutions defined by (S) for $\mu_k+\mu_s,\sigma_k$ and $ \mu^k + \mu_s,\sigma^k$, respectively, extended to $[0,\infty)^2$ as in Theorem \ref{thm:lip-conv}.  Since $\mu^k,\sigma^{2,k},\mu_k$, and $\sigma^2_k$ are Lipschitz continuous and $\mu_s$ satisfies (F2), we can apply Theorem \ref{thm:lip-conv} to show that 
\begin{equation}\label{eq:convergence}
U^h_k \longrightarrow U_k \ \ \text{and} \ \ U^{k,h} \longrightarrow U^k,
\end{equation}
locally uniformly on $[0,\infty)^2$ as $h \to 0$. Since $\mu_k\leq \mu\leq \mu^k$ and $\sigma_k \leq \sigma \leq \sigma^k$, we can make an argument, as in Theorem \ref{thm:lip-conv}, based on a comparison principle for (S), to show that $U^h_k \leq U^h \leq U^{k,h}$ for all $h,k$.  The proof is completed by combining this with \eqref{eq:convergence} and the locally uniform convergence $U_k,U^k \to U$.
\end{proof}

\subsection{Numerical simulations}
\label{sec:sim}

We present here some numerical simulations comparing the numerical solutions of (P), computed by (S), to realizations of directed last passage percolation (DLPP).  We restrict our attention to the box $[0,1]^2$ for simplicity.  For the case of exponential DLPP, we consider three macroscopic means, $\lambda_1,\lambda_2,$ and $ \lambda_3$ given by
\begin{equation}\label{eq:lambda1}
\lambda_1(x) = \begin{cases}
1,& \text{if } x_1 \geq 0.5 \text{ or } x_2 \geq 0.5, \\
0,& \text{otherwise},
\end{cases}
\end{equation} 
\begin{equation}\label{eq:lambda2}
\lambda_2(x) = \text{exp}\left(-10\left|x - (0.25,0.75)\right|^2\right) + \text{exp}\left(-10\left|x - (0.75,0.25)\right|^2\right),
\end{equation} 
and
\begin{equation}\label{eq:lambda3}
\lambda_3(x) = \begin{cases}
0.5,& \text{if } |x-(1,0)|^2 \leq 0.49 \text{ or } |x-(0,1)|^2 \leq 0.49,\\
1,& \text{otherwise.}
\end{cases}
\end{equation} 
Since the results are very similar for geometric DLPP, we consider only one macroscopic parameter $q$ given by
\begin{equation}\label{eq:q}
q(x) = \begin{cases}
0.5,& \text{if } x_1 \geq 0.5 \text{ or } x_2 \geq 0.5, \\
1,& \text{otherwise}.
\end{cases}
\end{equation}

Figure \ref{fig:demo1-1T} compares the level sets of the numerical solutions of (P) with simulations of exponential/geometric DLPP on a $1000\times 1000$ grid.    The smooth curves correspond to the level sets of the  numerical solution of (P) while the rough curves correspond to the level sets of the last passage time from the  DLPP simulation.  Figure \ref{fig:demo1-5T} shows the same comparison, except for DLPP simulations on a $5000\times 5000$ grid.  In both cases, the numerical solutions of (P) were computed on a $1000\times1000$ grid.  To give an idea of the computational complexity, it takes approximately a quarter of a second to numerically solve the PDE on this grid in MATLAB on an average laptop.
\begin{figure}[t!]
        \centering
        \subfigure[Exponential DLPP with mean $\lambda_1$]{
                \includegraphics[width=0.45\textwidth,clip=true,trim=30 30 30 25]{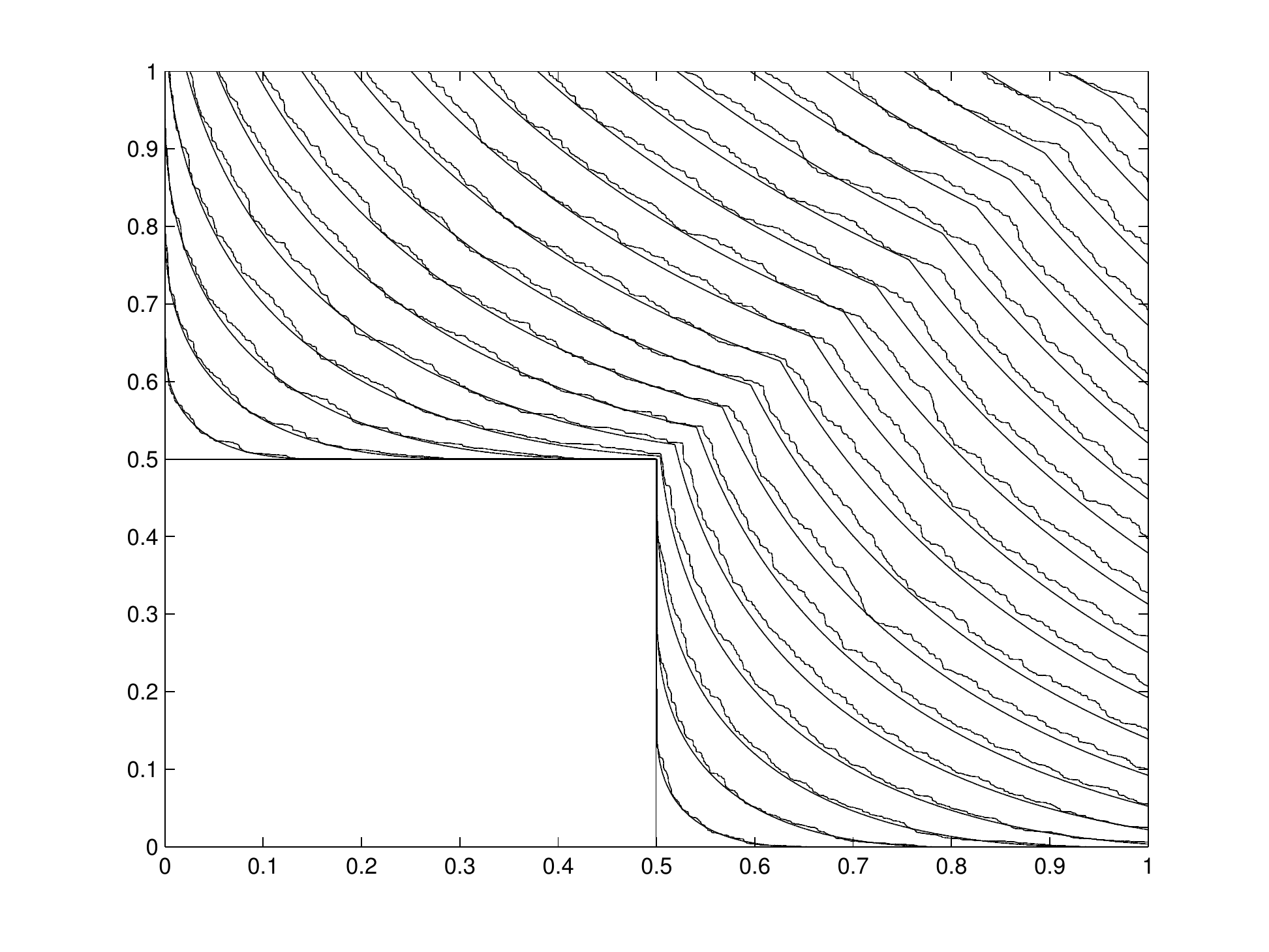}
                \label{fig:demo1_exp1T}}
        \subfigure[Geometric DLPP with parameter $q$]{
                \includegraphics[width=0.45\textwidth,clip=true,trim=30 30 30 25]{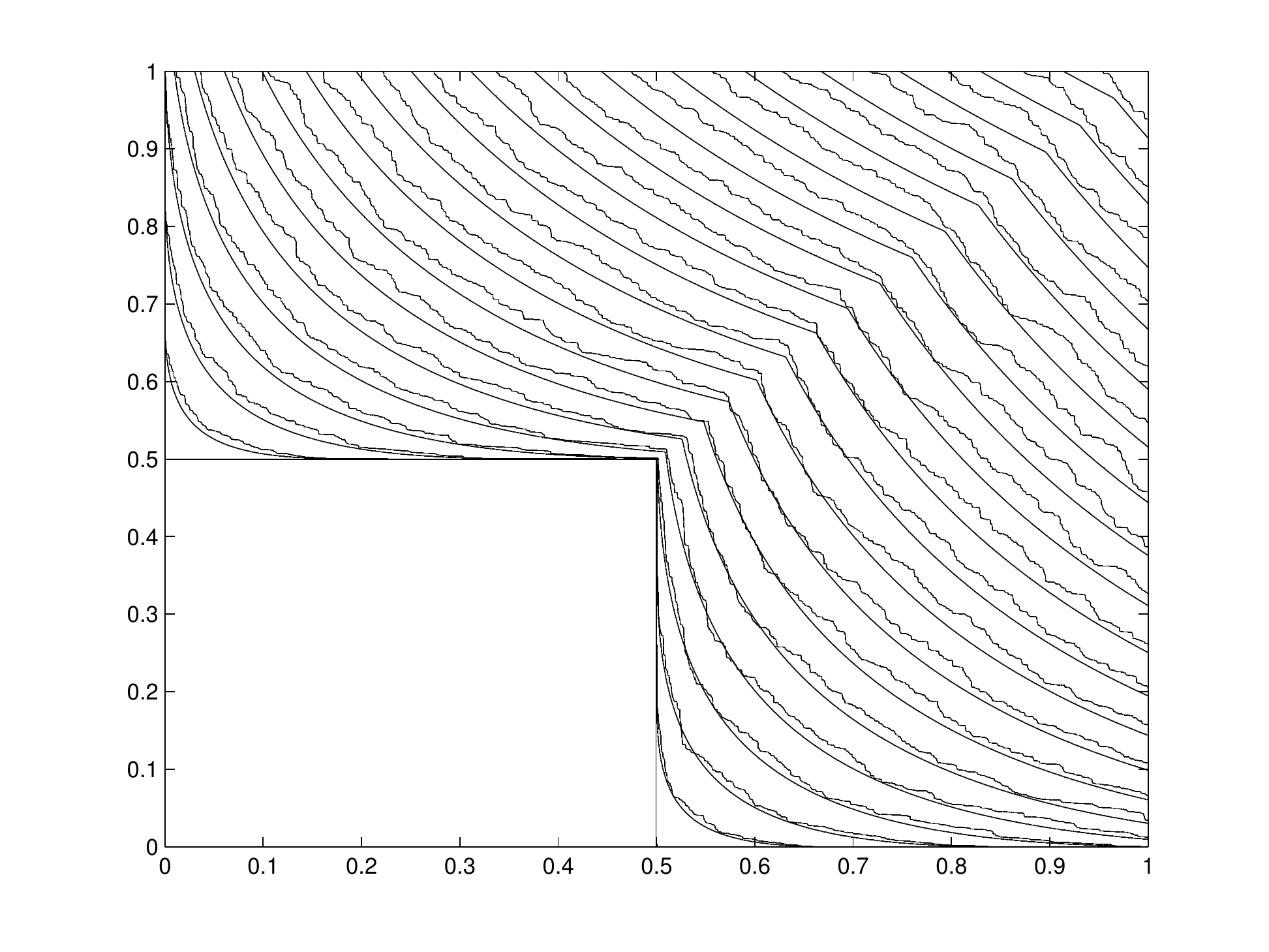}
                \label{fig:demo1_geo1T}}

       \subfigure[Exponential DLPP with mean $\lambda_2$]{
                \includegraphics[width=0.45\textwidth,clip=true,trim=30 30 30 25]{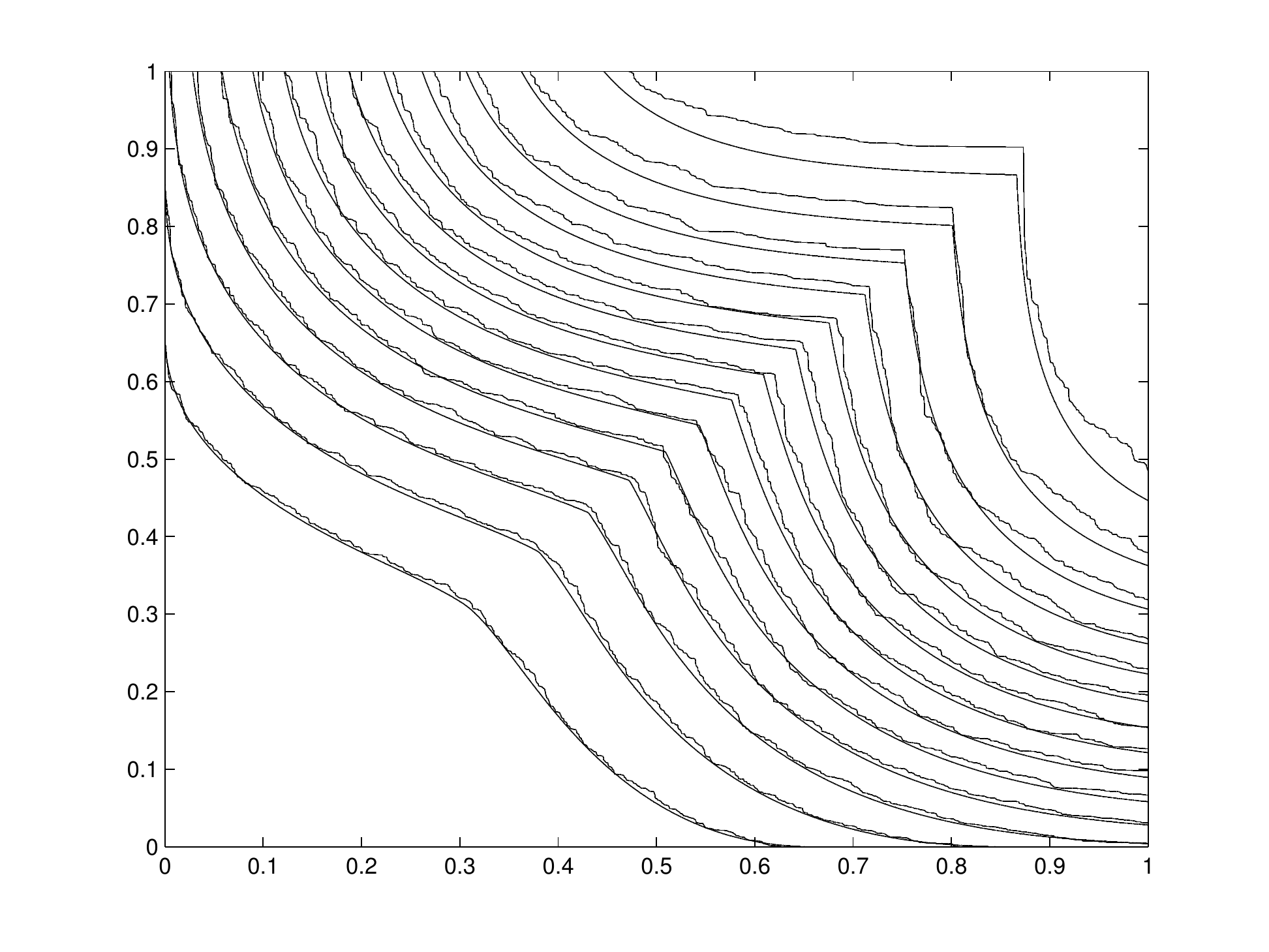}
                \label{fig:demo3_exp1T}}
       \subfigure[Exponential DLPP with mean $\lambda_3$]{
                \includegraphics[width=0.45\textwidth,clip=true,trim=30 30 30 25]{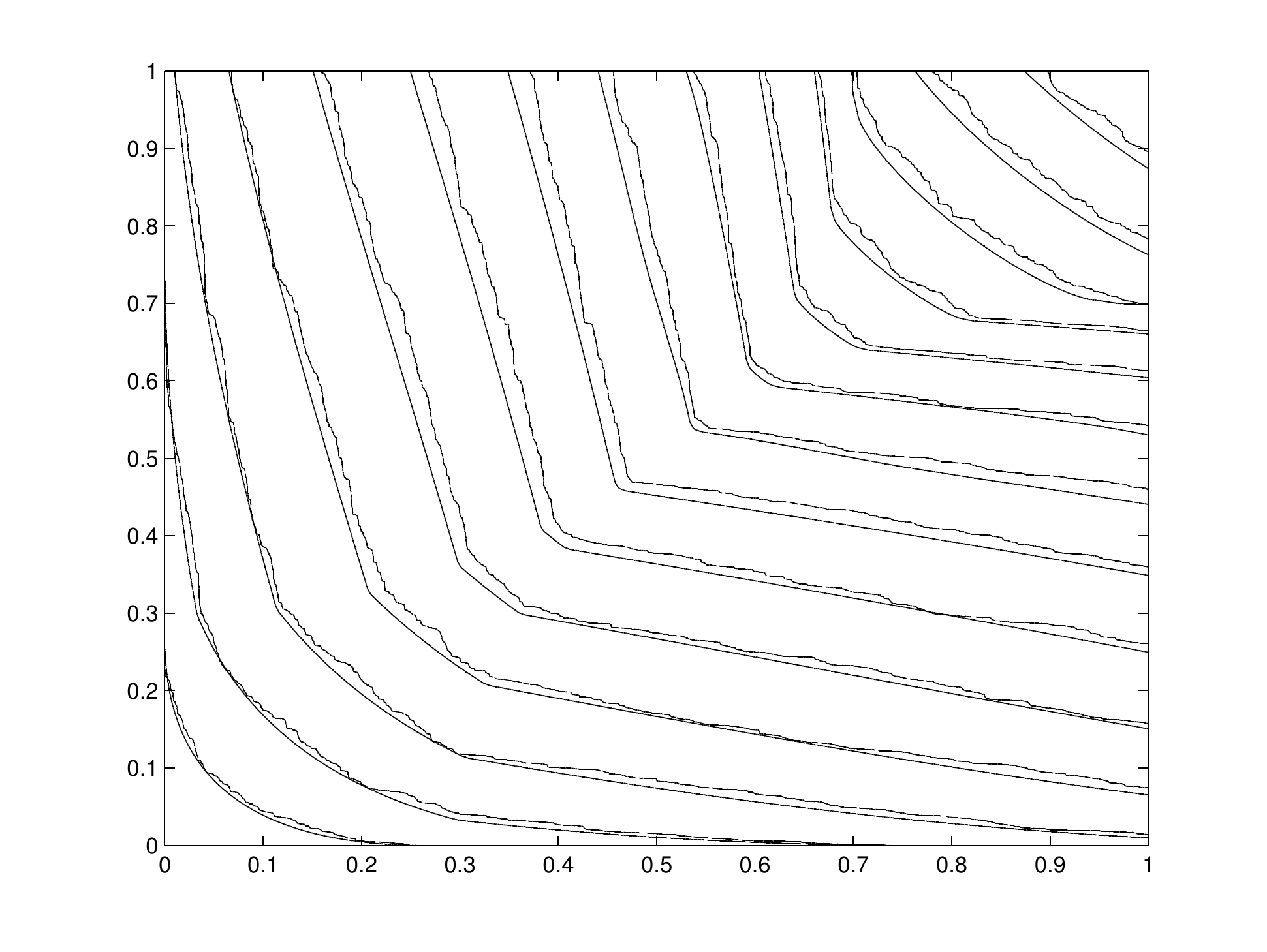}
                \label{fig:demo4_exp1T}}
        \caption{Comparisons of the level sets of numerical solutions of (P), computed via (S), and the level sets of exponential/geometric DLPP simulations on a $1000\times 1000$ grid. The smooth lines correspond to the numerical solutions of (P), while the rough lines correspond to the DLPP simulations. }
        \label{fig:demo1-1T}
\end{figure}
\begin{figure}[t!]
        \centering
        \subfigure[Exponential DLPP with mean $\lambda_1$]{
                \includegraphics[width=0.45\textwidth,clip=true,trim=30 30 30 25]{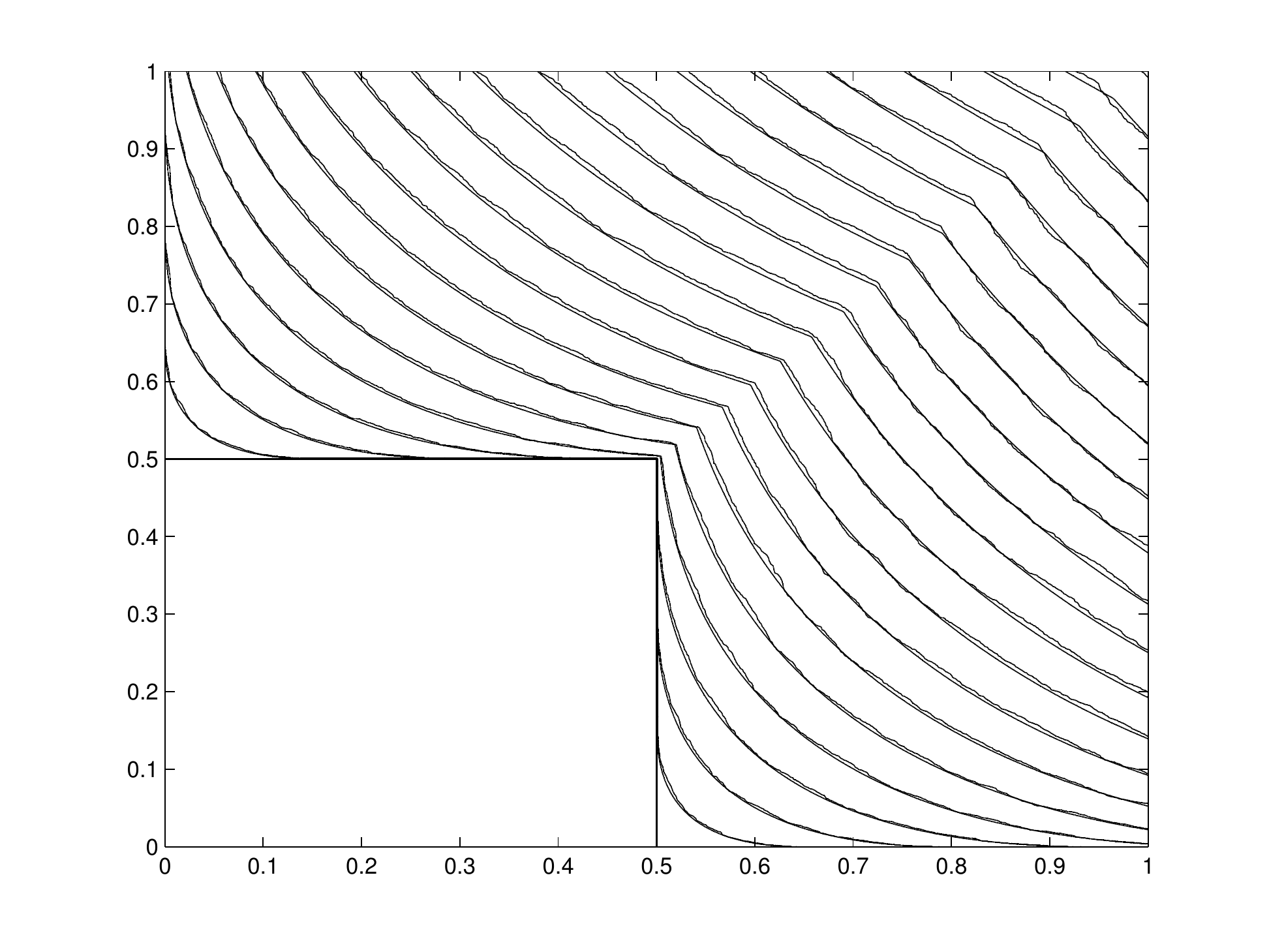}
                \label{fig:demo1_exp5T}}
        \subfigure[Geometric DLPP with parameter $q$]{
                \includegraphics[width=0.45\textwidth,clip=true,trim=30 30 30 25]{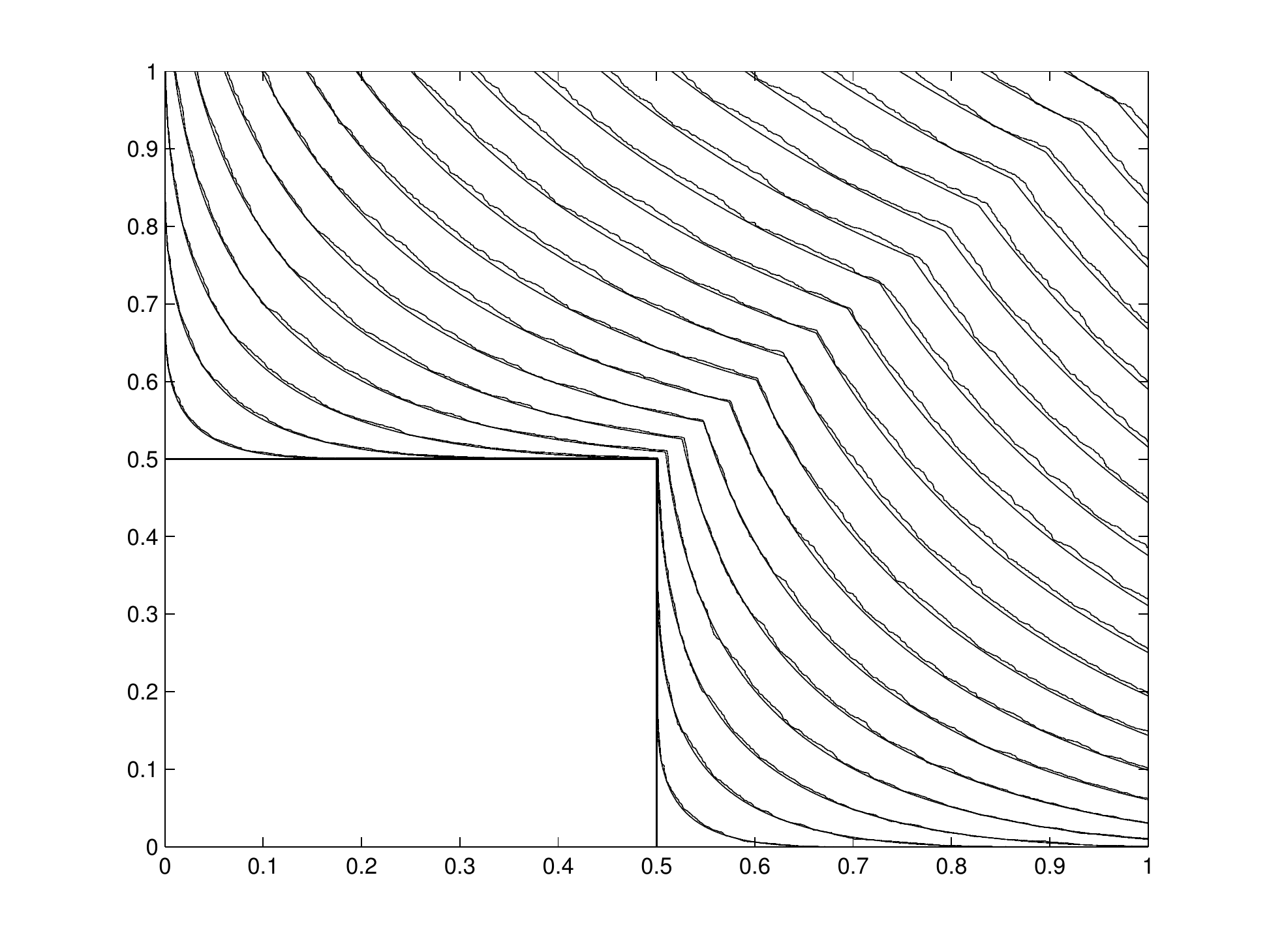}
                \label{fig:demo1_geo5T}}

       \subfigure[Exponential DLPP with mean $\lambda_2$]{
                \includegraphics[width=0.45\textwidth,clip=true,trim=30 30 30 25]{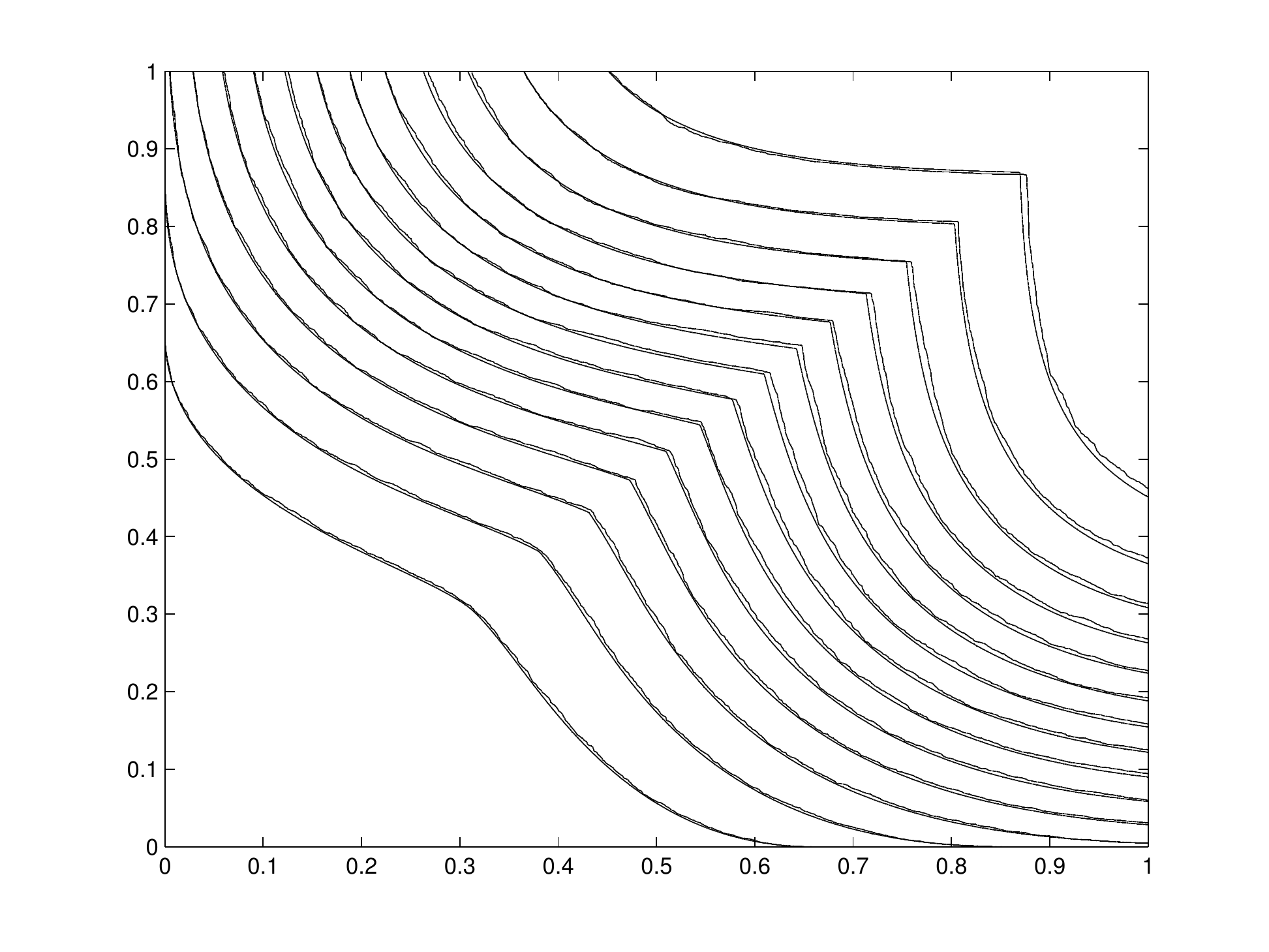}
                \label{fig:demo3_exp5T}}
       \subfigure[Exponential DLPP with mean $\lambda_3$]{
                \includegraphics[width=0.45\textwidth,clip=true,trim=30 30 30 25]{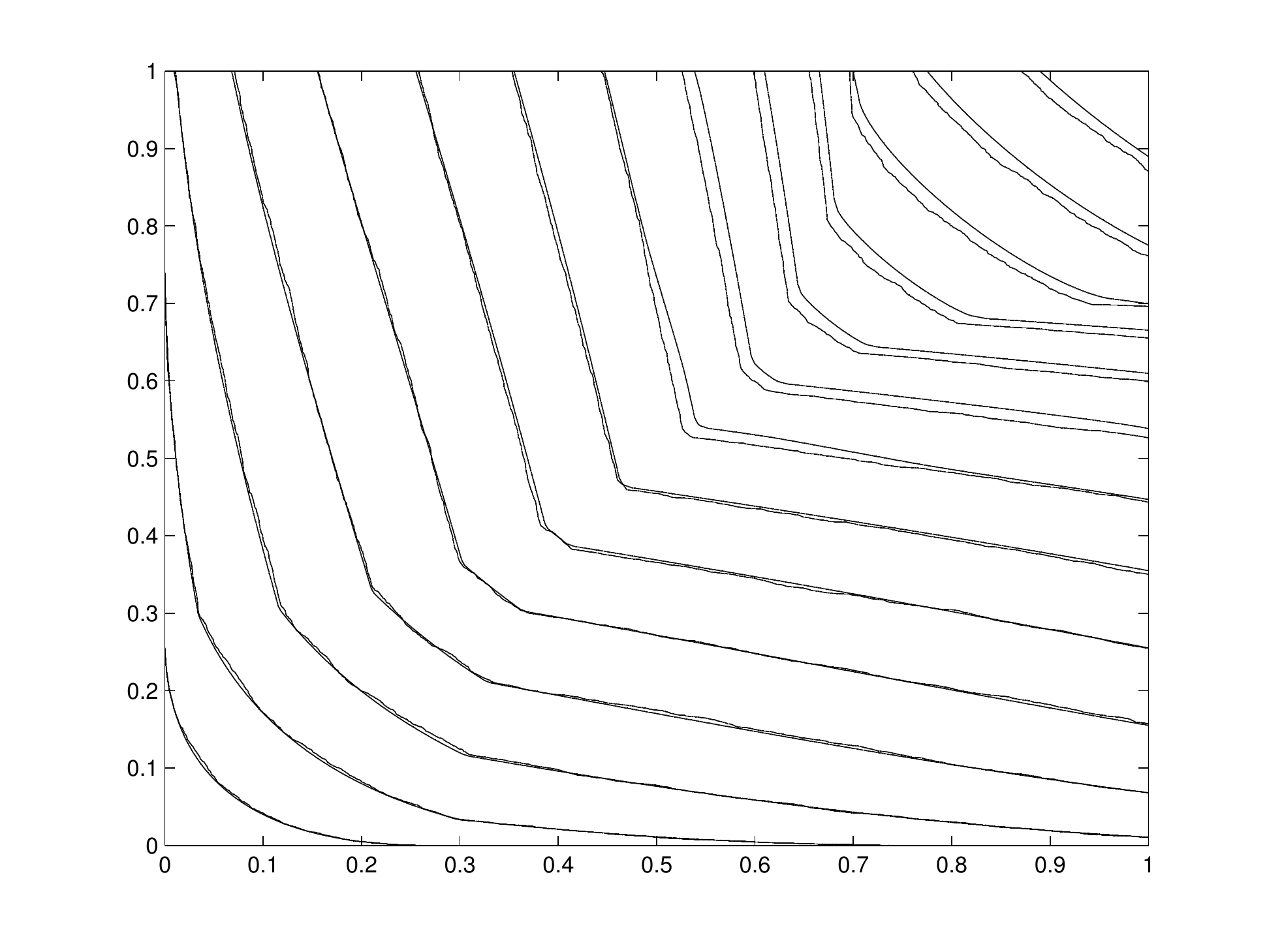}
                \label{fig:demo4_exp5T}}
               \caption{Comparisons of the level sets of numerical solutions of (P), computed via (S), and the level sets of exponential/geometric DLPP simulations on a $5000\times 5000$ grid. }
        \label{fig:demo1-5T}
\end{figure}

\subsection{Finding maximal curves}
\label{sec:max-curve}

We now propose an algorithm based on dynamic programming for finding maximizing curves, and we prove in Theorem \ref{thm:max-curve} and Corollary \ref{cor:max-curve2} that the curve produced by our algorithm is approximately optimal for the variational problem \eqref{eq:vardefU} defining $U$. Other approaches to finding maximizing curves, such as the method of characteristics~\cite{evansPDE}, or solving the Euler-Lagrange equations~\cite{rolla2008}, are not guaranteed to produce optimal curves, due to crossing characteristics, and the possibility of local minima. Our method is related to the method of synthesis in optimal control theory for computing optimal controls from solutions of Hamilton-Jacobi-Bellman equations~\cite{bardi1997}. 

Our algorithm has a parameter $\eps>0$ and a starting point $x \in \R^2_+$, and computes a curve $\gamma_\eps$ with $\gamma_\eps(0)=0$ and $\gamma_\eps(1)=x$ that nearly  maximizes $J$.  The algorithm works by starting at $x$ and tracing our way back to the origin by solving a series of dynamic programming problems. We set $x_0=x$, and generate $x_1,\dots,x_k,\dots$ as follows:  Given we are at step $k\geq 0$,  we use a dynamic programming principle (similar to Proposition \ref{prop:dpp}) to write 
\begin{equation}\label{eq:dynamic}
U(x_k) = \max_{s \in [0,1]} \Big\{ U(y(s)) + W(y(s),x_k)\Big\},
\end{equation}
where $y(s) =x_k - (1-s,s)\eps$. An application of H\"older's inequality yields
\begin{equation}\label{eq:app-holder}
J(\gamma) \leq \mu^*(x_k) \eps + 2\sigma^*(x_k) \eps\sqrt{s(1-s)} + o(\eps),
\end{equation}
for any $\gamma \in \A$ with $\gamma(0)=y(s)$ and $\gamma(1)=x_k$.  When $\mu$ and $\sigma$ are continuous, this upper bound can be attained (in the limit as $\eps\to0$) by the diagonal curve $\gamma(t) = (1-t)y(s) + x_kt$. Thus we are justified in making the following approximation 
\begin{equation}\label{eq:second}
W(y(s),x_k)=\sup_{\gamma \in \A \, : \,  \gamma(0)=y(s), \gamma(1)=x_k} J(\gamma) \approx \mu(x_k)\eps  + 2\sigma(x_k) \eps \sqrt{s(1-s)}.
\end{equation}
Substituting \eqref{eq:second} into \eqref{eq:dynamic} we find that
\begin{equation}\label{eq:approx-dynamic}
U(x_k) \approx \mu(x_k)\eps + \max_{s \in [0,1]} \left\{ U(y(s)) + 2\sigma(x_k)\eps\sqrt{s(1-s)} \right\}.
\end{equation}
We then define 
\begin{equation}\label{eq:iter}
x_{k+1}:=y(s^*_k)_+ = (x_k - (1-s^*_k,s^*_k)\eps)_+,
\end{equation}
where $s^*_k\in [0,1]$ is the maximizing argument in \eqref{eq:approx-dynamic} and $x_+ = (\max(x_1,0),\max(x_2,0))$.  The algorithm is terminated as soon as $x_k \in \partial \R^2_+$ and we append the final terminal point $x_{k+1} = 0$. In \eqref{eq:approx-dynamic}, we set $U(y(s)) = 0$ whenever $y(s) \not\in [0,\infty)^2$.
The algorithm is summarized in Algorithm \ref{alg:max-curve}. 
\begin{algorithm*}[h]
\normalsize Given a step size $\eps>0$ and $x_0 \in \R^2_+$, we generate $x_1,\dots,x_k,\dots$ as follows:\\
$k=0$;\\
\While{$x_k \in \R^2_+$} {
$s^*_k= \argmax_{s \in [0,1]} \left\{ U(x_k - (1-s,s)\eps) + 2\sigma(x_k)\eps\sqrt{s(1-s)} \right\}$\;
$x_{k+1} = (x_k - (1-s^*_k,s^*_k) \eps)_+$\;}
$x_{k+1} = 0$;
\caption{Find $\eps$-optimal curve}\label{alg:max-curve}
\end{algorithm*}

 Notice that the boundary source $\mu_s$ does not appear explicitly in Algorithm \ref{alg:max-curve}, though it does appear implicitly through the solution $U$ of (P). Each step of the algorithm moves a distance of at least $\eps/2$ in the direction $(-1,0)$ or $(0,-1)$.  If $x_0 \in [0,R]^2$, then the algorithm will terminate in at most $4R/\eps$ steps.  Furthermore, when $\mu$ and $\sigma^2$ are Lipschitz, we can show that the polygonal curve $\gamma_\eps$ generated by Algorithm \ref{alg:max-curve} has energy within $O(\eps)$ of the maximizing curve.  This is summarized in the following result.
\begin{theorem}\label{thm:max-curve}
  Let $R>0$, suppose that $\mu$ and $\sigma^2$ are non-negative,  globally Lipschitz continuous on $[0,R]^2$ with constant $C_{lip}>0$, and suppose that $\mu_s$ satisfies (F2).  Let $\eps>0$, $x_0 \in (0,R]^2$, and let $x_1,\dots,x_K$ be the points generated by Algorithm \ref{alg:max-curve}. Let $\gamma_\eps:[0,1]\to [0,\infty)^2$ be the monotone polygonal curve passing through $x_K,x_{K-1},\dots,x_1,x_0$. Then there exists a constant $C=C(\|\mu\|_\infty,\|\sigma\|_\infty)>0$ such that
\begin{equation}\label{eq:max-energy}
  U_{\mu+\mu_s,\sigma}(x_0)\leq J_{\mu+\mu_s,\sigma}(\gamma_\eps) + C(1 + C_{lip}R) \eps.
\end{equation}
\end{theorem}
\begin{proof}
  For convenience, we set $U=U_{\mu+\mu_s,\sigma}$, $J=J_{\mu+\mu_s,\sigma}$, and we extend $\mu$, $\sigma$ and $U$ to functions on $\R^2$ by setting $\mu(x)=\sigma(x)=U(x)=0$ for $x \not\in[0,\infty)^2$.  Writing $\Delta t = 1/K$ and $t_j = j\Delta t$ for $j=0,\dots,K$, we can parameterize $\gamma_\eps$ so that 
\begin{equation}\label{eq:gamprime}
\gamma_\eps'(t) = \frac{1}{\Delta t} (x_{K-j} - x_{K-j+1}) = \frac{\eps}{\Delta t} (1-s^*_{K-j},s^*_{K-j}),
\end{equation}
 for $t \in (t_{j-1},t_j)$ and $j\geq 3$. It follows that 
\begin{align}\label{eq:energy1}
\int_{t_2}^1 \ell(\gamma_\eps(t)&,\gamma_\eps'(t))\, dt \\
&=  \sum_{j=3}^K \int_{t_{j-1}}^{t_j} \ell(\gamma_\eps(t),\gamma_\eps'(t))\, dt  \notag \\
&= \eps \sum_{j=3}^K \frac{1}{\Delta t}\int_{t_{j-1}}^{t_j}\mu(\gamma_\eps(t)) + 2\sigma(\gamma_\eps(t)) \sqrt{(1-s^*_{K-j})s^*_{K-j}} \, dt \notag \\
&\geq \eps\sum_{j=3}^K \left(\frac{1}{\Delta t}\int_{t_{j-1}}^{t_j} \mu(x_{K-j}) + 2\sigma(x_{K-j})\sqrt{(1-s^*_{K-j})s^*_{K-j}} \, dt - 3C_{lip} \eps  \right) \notag \\
&= \left(\sum_{j=3}^K \mu(x_{K-j}) + 2\sigma(x_{K-j})\sqrt{(1-s^*_{K-j})s^*_{K-j}} \right)\eps - 3KC_{lip} \eps^2.
\end{align}

An application of H\"older's inequality gives
\begin{equation}\label{eq:Jupper}
J(\gamma) \leq \left(\mu(x_{K-j}) + 2\sigma(x_{K-j})\sqrt{s(1-s)}\right)\eps + 3C_{lip} \eps^2,
\end{equation}
for $j\geq2$ and any $\gamma \in \A$ with $\gamma(0)=y(s)$ and $\gamma(1)=x_{K-j}$. Combining this with the dynamic programming principle \eqref{eq:dynamic} we have
\begin{equation}\label{eq:dynamic2}
U(x_{K-j}) \leq\mu(x_{K-j})\eps +  \max_{s \in [0,1]} \left\{ U(y(s)) +2\sigma(x_{K-j})\eps \sqrt{s(1-s)}) \right\}+ 3C_{lip} \eps^2,
\end{equation}
for all $j \geq 2$.  
By the definition of $s^*_{K-j}$ we have
\begin{equation}\label{eq:dynamic3}
U(x_{K-j})\! \leq\! U(x_{K-j+1})\! +\! \eps\left(\mu(x_{K-j})\! +\! 2\sigma(x_{K-j})\! \sqrt{(1-s^*_{K-j})s^*_{K-j}}\right)\!+\! 3C_{lip} \eps^2,
\end{equation}
for $j\geq 3$.
By iterating this inequality for $j=K,\dots,3$ we have
\begin{align}\label{eq:key-dyn}
U(x_0) &\leq U(x_{K-2}) + \left(\sum_{j=3}^K \mu(x_{K-j}) + 2\sigma(x_{K-j})\sqrt{(1-s^*_{K-j})s^*_{K-j}} \right)\eps + 3KC_{lip}\eps^2\notag \\ 
&\hspace{-1.5mm}\stackrel{\eqref{eq:energy1}}\leq U(x_{K-2}) + \int_{t_2}^1 \ell(\gamma_\eps(t),\gamma'_\eps(t)) \, dt + 6KC_{lip}\eps^2.
\end{align}

We have two cases now.  Suppose first that $y(s^*_{K-2}) \not\in [0,\infty)^2$.  Then $U(y(s^*_{K-2}))=0$ and by \eqref{eq:dynamic2} we have that $U(x_{K-2}) \leq C \eps$.  Combining this with \eqref{eq:key-dyn} we have
\begin{equation}\label{eq:key2}
U(x_0) \leq  J(\gamma_\eps) + C\eps + 6KC_{lip} \eps^2.
\end{equation} 
The proof is completed by noting that $K \leq 4R/\eps$.

Suppose now that $y(s^*_{K-2}) \in [0,\infty)^2$. Then \eqref{eq:dynamic3} holds for $j=2$ and combining this with \eqref{eq:key-dyn} we have
\begin{equation}\label{eq:key3}
U(x_0) \leq U(x_{K-1}) + \int_{t_2}^1 \ell(\gamma_\eps(t),\gamma'_\eps(t)) \, dt  +   6(K+1)C_{lip} \eps^2.
\end{equation}
Since $x_K=0$ we must have $x_{K-1} \in \partial \R^2_+$.  It follows that
\[\int_0^{t_1} \ell(\gamma_\eps(t),\gamma_\eps'(t))\, dt = U(x_{K-1}).\] 
Inserting this into \eqref{eq:key3} we see that
\[U(x_0) \leq J(\gamma_\eps)+  6(K+1)C_{lip} \eps^2.\vspace{-20pt} \]
\end{proof}

If $\mu$ and $\sigma^2$ are not globally Lipschitz continuous, then Algorithm \ref{alg:max-curve} is not guaranteed to yield optimal curves.  However, it can be easily modified to give an algorithm that does.  
\begin{corollary}\label{cor:max-curve2}
Suppose that $\mu$ and $\sigma^2$ simultaneously satisfy (F1), (F3) and \eqref{eq:zeros}, and let $\mu_s$ satisfy (F2).  Let $\mu_k$ and $\sigma_k$ be sequences of functions such that $\mu_k$ and $\sigma^2_k$ are Lipschitz with constant $k$,  $\mu_k \leq \mu$, $\sigma_k \leq \sigma$ and $U_{\mu_k+\mu_s,\sigma_k} \to U_{\mu+\mu_s,\sigma}$ locally uniformly.  Let $x_0 \in (0,R]^2$ and let $\gamma_{k}:[0,1]\to [0,\infty)^2$ be the monotone polygonal curve generated by applying Algorithm \ref{alg:max-curve} to $x_0, \mu_k,\sigma_k$ and $U_k$ with $\eps = k^{-1}(U_{\mu+\mu_s,\sigma}(x_0) -U_{\mu_k+\mu_s,\sigma_k}(x_0))$. Then we have 
\begin{equation}\label{eq:max-energy2}
  U(x_0)\leq J(\gamma_{k}) + o(1) \ \ \text{as } k \to \infty.
\end{equation}
\end{corollary}
\begin{proof}
  Let us set $J_k = J_{\mu_k+\mu_s,\sigma_k}$, $J=J_{\mu+\mu_s,\sigma}$, $U_k=U_{\mu_k+\mu_s,\sigma_k}$, and $U=U_{\mu+\mu_s,\sigma}$.  
By Theorem \ref{thm:max-curve} there exists a constant $C=C(\|\mu\|_\infty,\|\sigma\|_\infty)>0$ such that
\[U_k(x_0) \leq J_k(\gamma_{k}) + C(1 + k R)k^{-1}(U(x_0) - U_k(x_0)) \leq J(\gamma_{k}) + C(1 + R)(U(x_0) - U_k(x_0)).\]
It follows that
\[U(x_0) \leq J(\gamma_{k}) + C(2 + R)(U(x_0) - U_k(x_0)).\vspace{-20pt}\]
\end{proof}
\begin{figure}[t!]
        \centering
        \subfigure[Exponential DLPP with mean $\lambda_1$]{
                \includegraphics[width=0.45\textwidth,clip=true,trim=30 30 30 25]{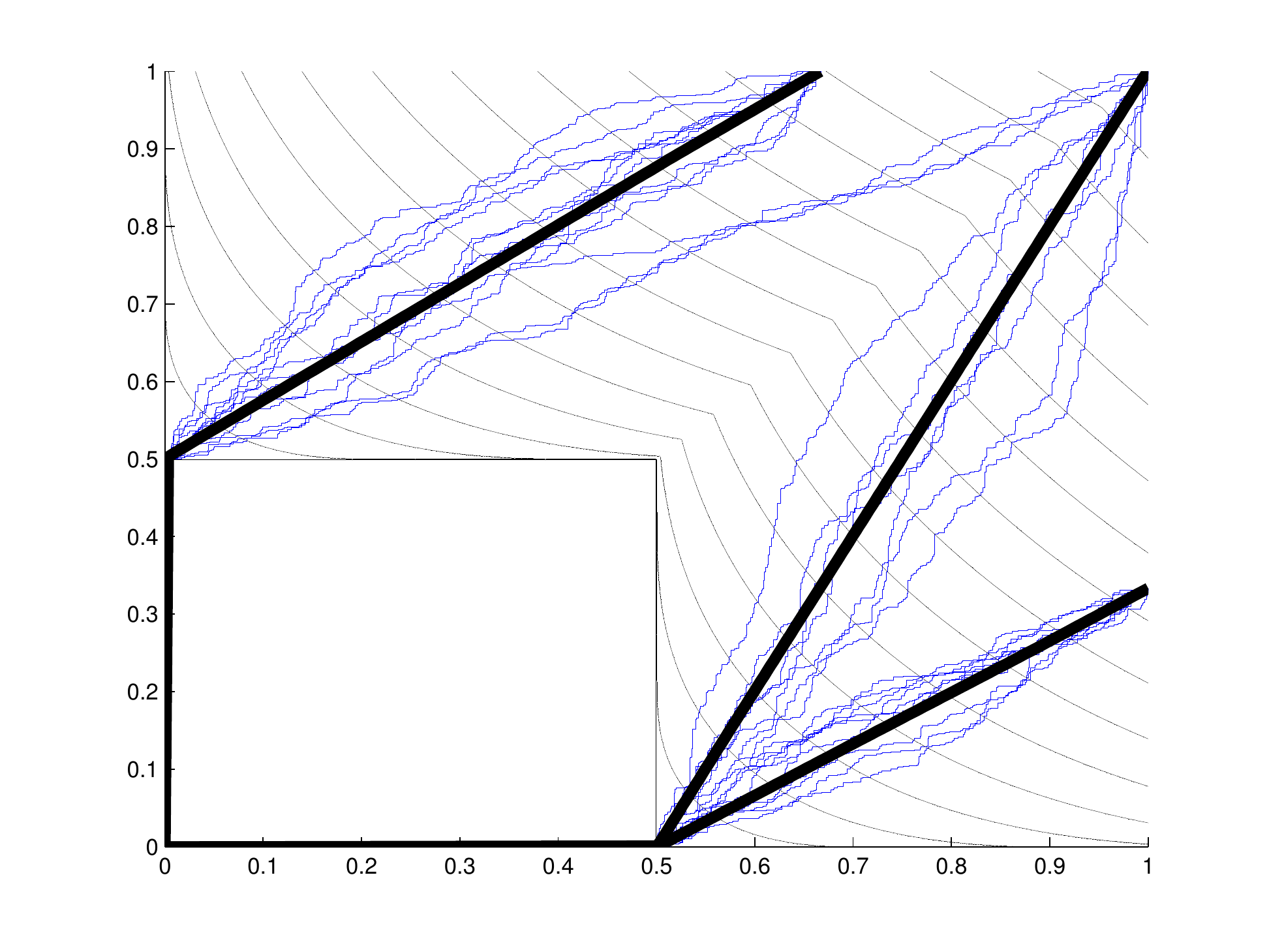}
                \label{fig:demo1_exp_optpath}}
        \subfigure[Geometric DLPP with parameter $q$]{
                \includegraphics[width=0.45\textwidth,clip=true,trim=30 30 30 25]{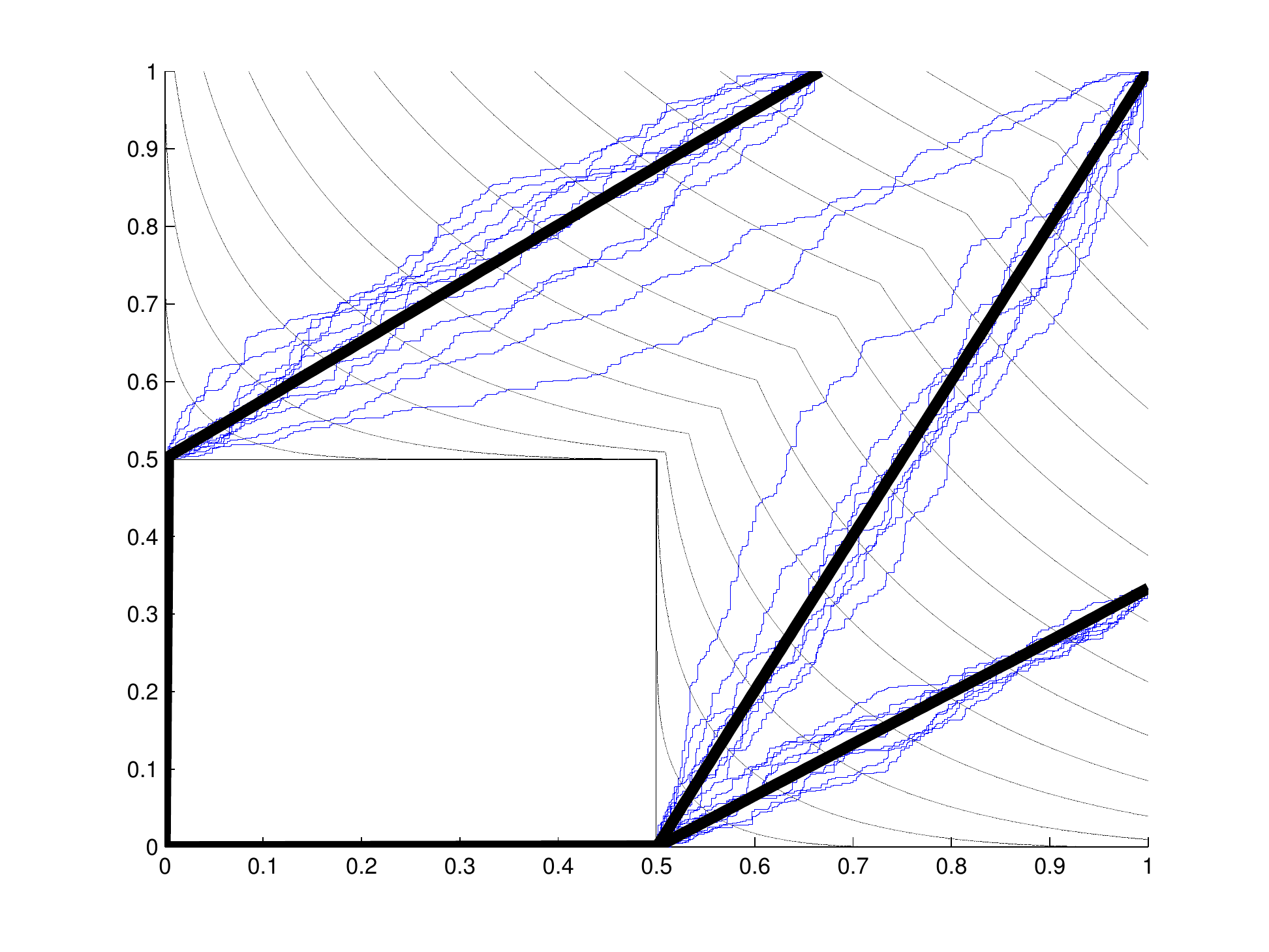}
                \label{fig:demo1_geo_optpath}}

        \subfigure[Exponential DLPP with mean $\lambda_2$]{
                \includegraphics[width=0.45\textwidth,clip=true,trim=30 30 30 25]{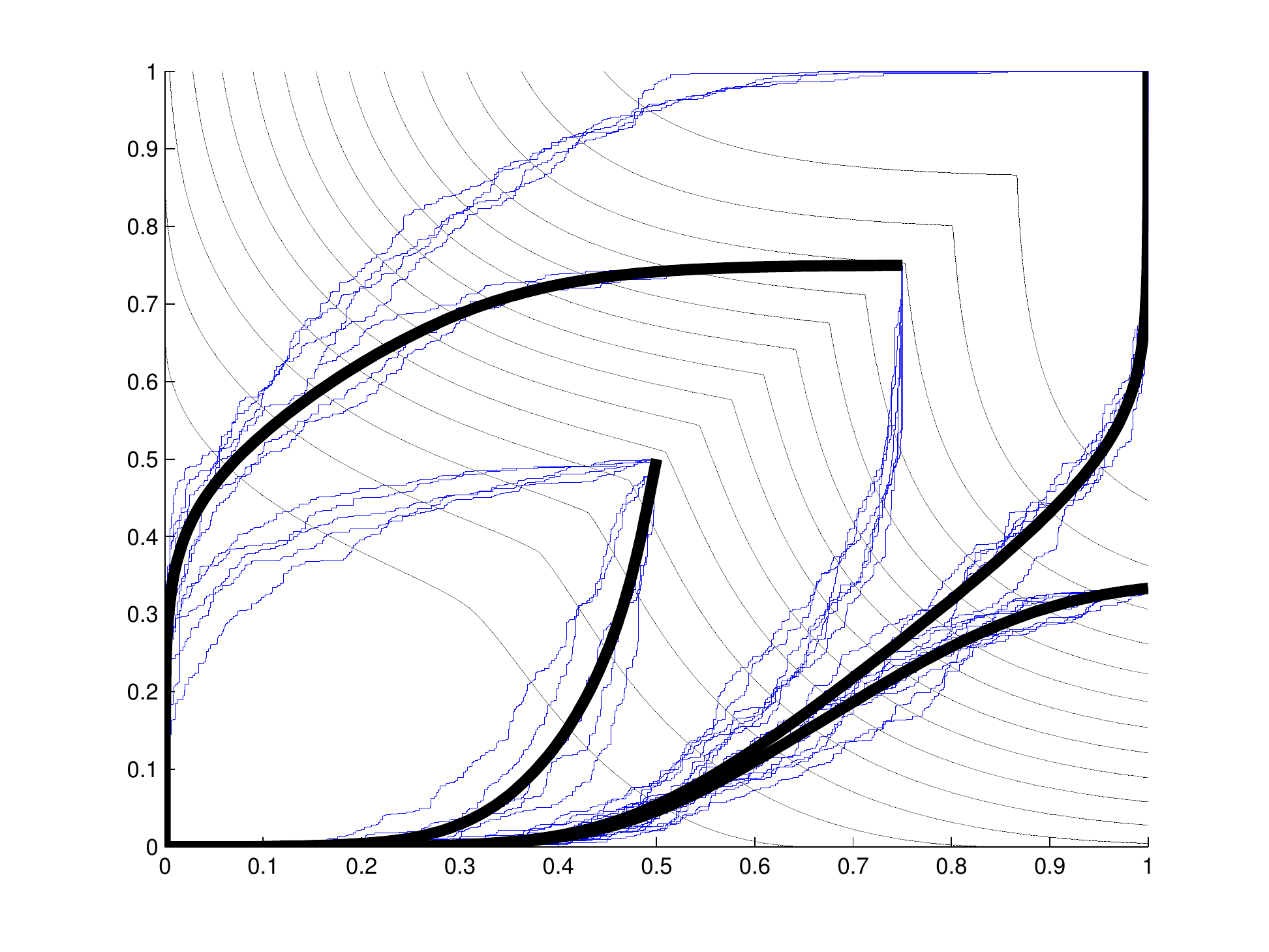}
                \label{fig:demo3_exp_optpath}}
        \subfigure[Exponential DLPP with mean $\lambda_3$]{
                \includegraphics[width=0.45\textwidth,clip=true,trim=30 30 30 25]{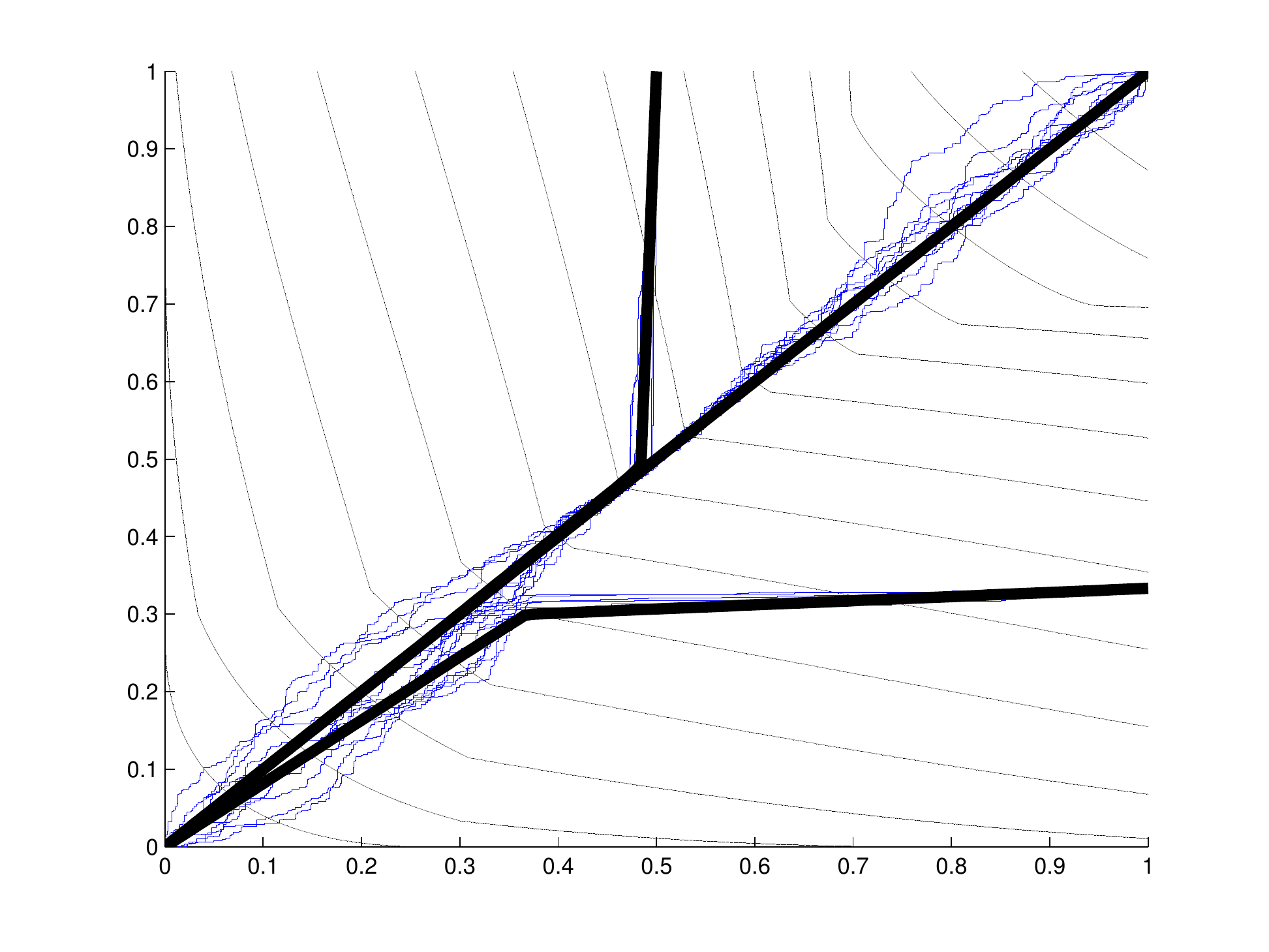}
                \label{fig:demo4_exp_optpath}}
        \caption{Comparisons of the curve $\gamma_\eps$ ($\eps=0.01$) generated by Algorithm \ref{alg:max-curve} to the optimal paths from $10$ realizations of DLPP for the macroscopic weights considered in Section \ref{sec:sim}.  In each experiment, we show the curve $\gamma_\eps$ and optimal paths for several different terminal points $x_0 \in (0,1)^2$.  Notice that in a), b) and c), there are multiple optimizing curves, and Algorithm \ref{alg:max-curve} finds only one curve, depending on the choice one makes when there are multiple maximizing arguments for $s_k^*$.  The DLPP simulations were performed on a $1000\times1000$ grid, $s_k^*$ was computed via an exhaustive search with a grid size of $0.01$. }
        \label{fig:demo2}
\end{figure}

We now show some simulation results using Algorithm \ref{alg:max-curve} to compute approximately optimal curves for the exponential/geometric DLPP simulations presented in Section \ref{sec:sim}.  Figure \ref{fig:demo2} shows the curves generated by Algorithm \ref{alg:max-curve} along with optimal paths for $10$ realizations of DLPP on a $1000\times1000$ grid. We also show the level sets of the numerical solutions of (P) to give points of reference. In all cases, we used a step size of $\eps=0.01$ and computed $s_k^*$ in Algorithm \ref{alg:max-curve} by an exhaustive search with a grid size of $0.01$. With these choices of parameters, Algorithm \ref{alg:max-curve} runs in approximately a quarter of a second, assuming the numerical solution $U$ is already available.   Note also that we implemented Algorithm \ref{alg:max-curve} exactly as written, even when $\mu$ and $\sigma$ are discontinuous, and do not substitute continuous versions as in Corollary \ref{cor:max-curve2}.

\begin{figure}[t!]
        \centering
        \subfigure[Source at $\{x_2=0\}$]{
                \includegraphics[width=0.45\textwidth,clip=true,trim=30 30 30 25]{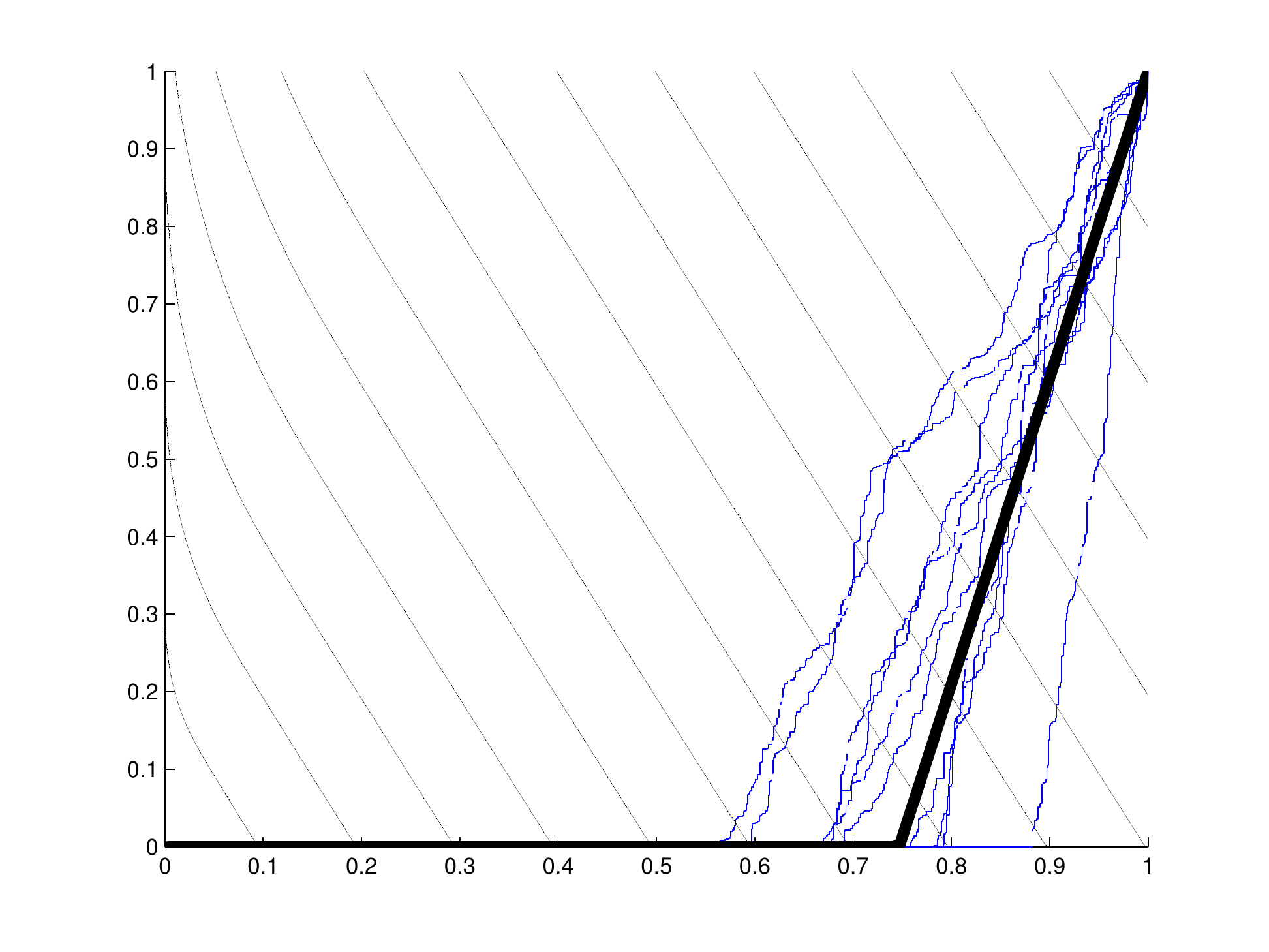}
                \label{fig:sources_demo1}}
        \subfigure[Source at $\{x_2=0.25\}$]{
                \includegraphics[width=0.45\textwidth,clip=true,trim=30 30 30 25]{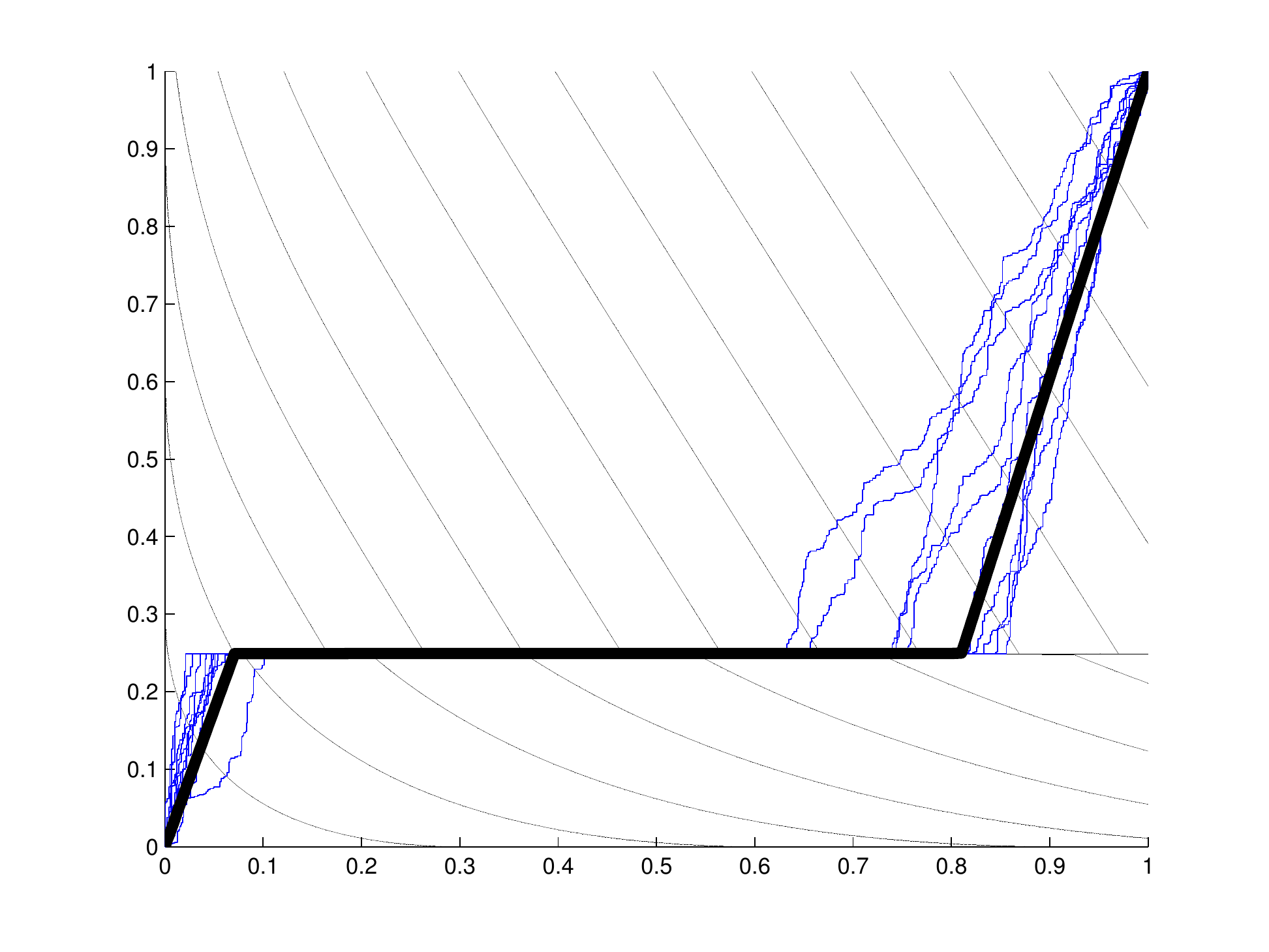}
                \label{fig:sources_demo2}}

        \subfigure[Source at $\{x_2=0.5\}$]{
                \includegraphics[width=0.45\textwidth,clip=true,trim=30 30 30 25]{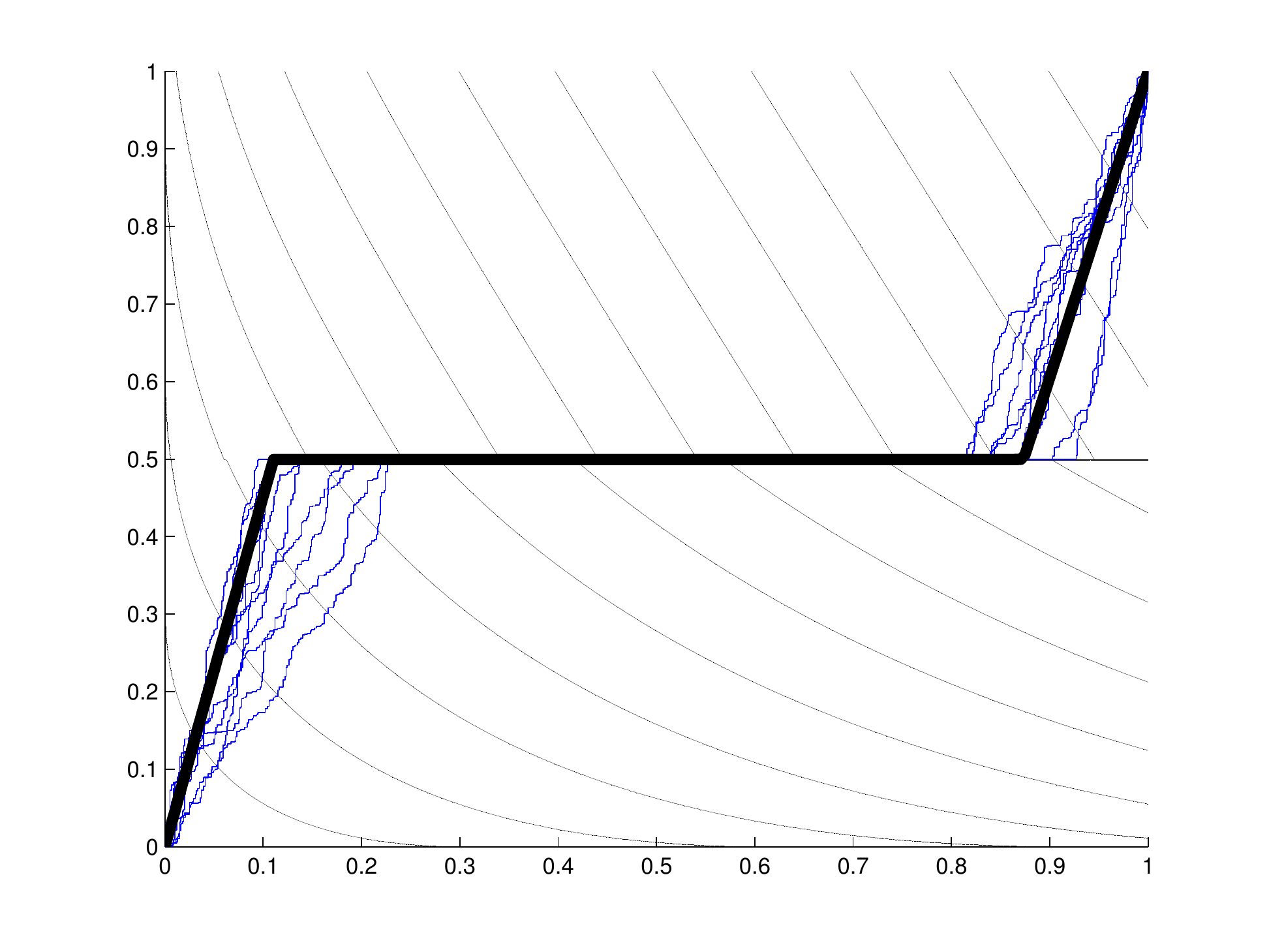}
                \label{fig:sources_demo3}}
        \subfigure[Source at $\{x_2=0.75\}$]{
                \includegraphics[width=0.45\textwidth,clip=true,trim=30 30 30 25]{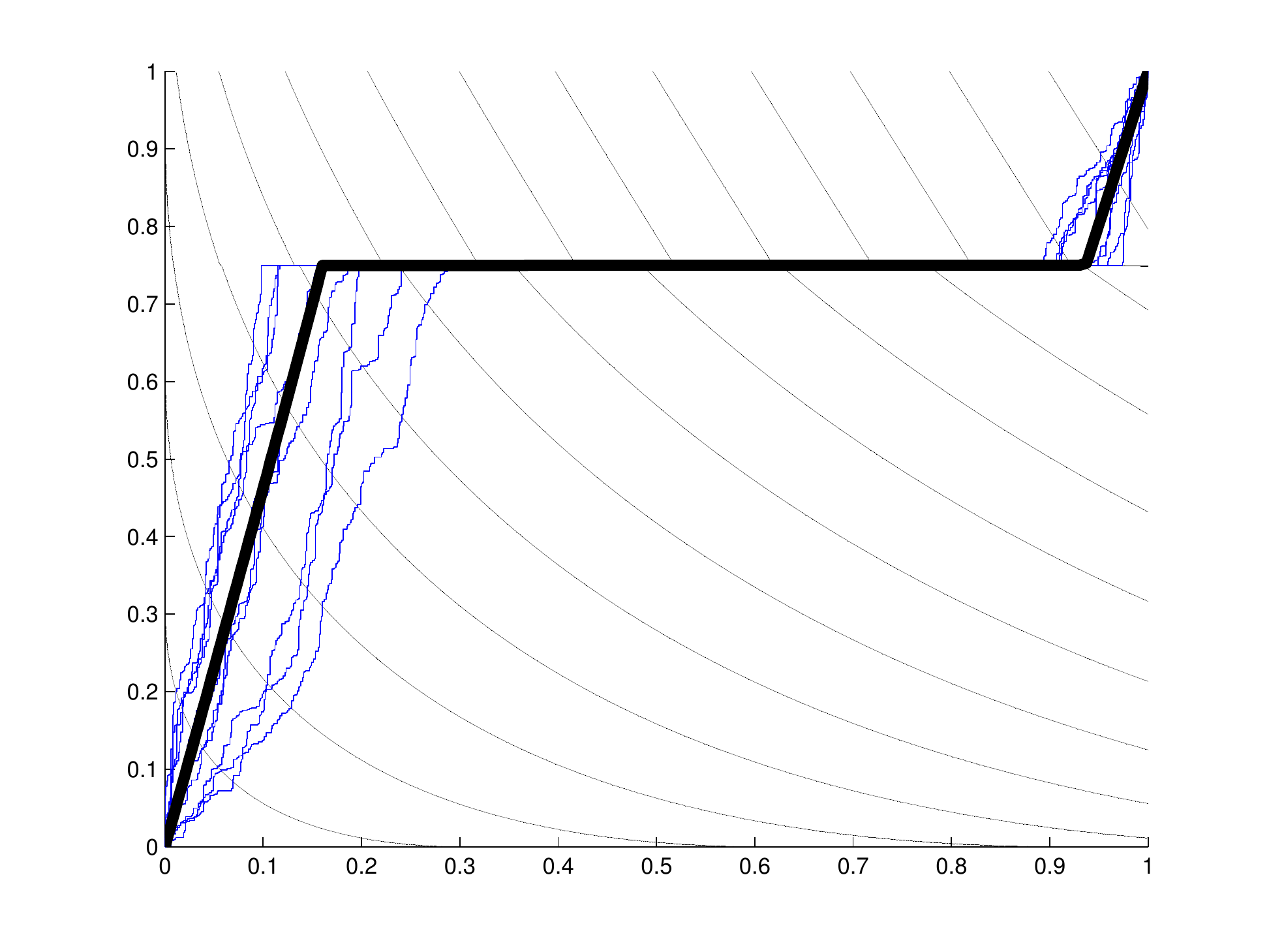}
                \label{fig:sources_demo4}}
        \caption{Comparisons of the optimal curve $\gamma_\eps$ ($\eps=0.01$) generated by Algorithm \ref{alg:max-curve} to the optimal paths from 10 realizations of exponential DLPP.  The macroscopic weight functions are constant $\mu=1$ on $[0,1]^2$ plus a source term $\mu_s=2$ concentrated on a horizontal line. The simulations were performed on a $1000\times 1000$ grid.}
        \label{fig:sources_demo}
\end{figure}

As in~\cite{rolla2008}, it is expected that the optimal paths for DLPP will asymptotically concentrate around optimal curves for the variational problem, and this is clearly reflected in the simulations in Figure \ref{fig:demo2}.  Notice that for exponential DLPP with means $\lambda_1,\lambda_2$ and geometric DLPP with parameter $q$, there are multiple maximizing curves for any terminal point $x$ along the diagonal $\{x_1=x_2\}$.  We see that some of the DLPP realizations concentrate around one optimal path, while the remaining realizations concentrate around the other.  Algorithm \ref{alg:max-curve} will of course only find one of the maximizing curves, depending on the choice one makes when there are multiple maximizing arguments in the definition of $s_k^*$.

We now show some simulations with a source term $\mu_s$.  Here we consider exponential DLPP with mean $\lambda=1$ on $[0,1]\times (0,1]$ and $\lambda = 3$ on $[0,1]\times \{0\}$. Figure \ref{fig:sources_demo1} shows the optimal curve generated by Algorithm \ref{alg:max-curve}, along with the level sets of the numerical solution of (P) and the optimal paths from 10 realizations of exponential DLPP on a $1000\times1000$ grid. 

Although our assumptions only allow sources on the boundary $\partial \R^2_+$, many of the results in the paper can be shown to hold for sources along horizontal or vertical lines in $\R^2_+$.  The idea is to find the appropriate dynamic programming principle that plays the role of  Proposition \ref{prop:dpp}, so that the effect of the weights in the bulk is separated from the source.  In the case of a source along the line $\{x_2=\alpha\}$ for $\alpha \in (0,1)$, and assuming no boundary sources, the dynamic programming principle would be
\[U(y) = \max_{0 \leq x_1 \leq x_1'\leq y_1} \left\{W(0,(x_1,\alpha)) + \int_{x_1}^{x_1'} \mu(t,\alpha) + \mu_s(t,\alpha) \, dt + W( (x_1',\alpha),y)  \right\},\]
where $U=U_{\mu+\mu_s,\sigma}$, $W=W_{\mu,\sigma}$, $\mu$ and $\sigma^2$ are, say, Lipschitz on $[0,\infty)^2$, and $\mu_s$ represents the source, which  is nonzero only on the line $\{x_2=\alpha\}$.   We can then  use this dynamic programming principle and its discrete version (similar to  \eqref{eq:dpp1}) in the proof of Theorem \ref{thm:main}.  The one caveat is that $U$ is in general discontinuous along the line containing the source, though $U$ remains locally uniformly continuous on each of the components of $\R^2_+$ obtained by removing the source line.  Thus, $U$ can only be identified via the variational problem \eqref{eq:vardef}, since we have not proven uniqueness of discontinuous viscosity solutions of (P). However, our  numerical results suggest that either uniqueness holds for (P) in some special cases where $U$ is discontinuous, or at the very least our numerical scheme for (P) selects the ``correct'' viscosity solution for the percolation problem.

Figure \ref{fig:sources_demo2}, \ref{fig:sources_demo3}, and \ref{fig:sources_demo4} show the optimal curve generated by Algorithm \ref{alg:max-curve}, along with DLPP simulations for sources on the horizontal lines $\{x_2=0.25\}$, $\{x_2=0.5\}$, and $\{x_2=0.75\}$, respectively.  

\subsection{TASEP with slow bond rate}
\label{sec:bond}

Finally, we consider the totally asymmetric simple exclusion process (TASEP) with a slow bond rate at the origin.  This model was originally introduced by Janowsky and Lebowitz~\cite{janowsky1994exact}, and some partial results were obtained more recently by Sepp\"al\"ainen \cite{seppalainen2001hydrodynamic}.  The process of interest is the usual TASEP with exponential rates of $1$ at all locations in $\Z$ except for the origin, which has a slower rate of $r \in (0,1]$.  One can think of this as modeling traffic flow on a road with a single toll both that every car must pass through.

Through the correspondence with DLPP, the slow bond rate corresponds to a source on the diagonal $\{x_1=x_2\}$.  In the context of our paper, we would have
\begin{equation}\label{eq:diag-source}
\mu(x) = \begin{cases}
1/r,& \text{if } x_1=x_2,\\
1,& \text{otherwise}.
\end{cases}
\end{equation}
Notice that $\mu$ does not satisfy the assumptions of Theorem \ref{thm:main}, and we do not expect the continuum limit (P) to hold in this case.

A quantity of interest is
\[\kappa(r):=  \lim_{N\to \infty} \frac{1}{N}L(N,N) \ \ \text{for } r \leq 1,\]
which corresponds to the reciprocal of the maximum TASEP current~\cite{seppalainen2001hydrodynamic}. It is known that $\kappa(1)=4$ and Sepp\"al\"ainen~\cite{seppalainen2001hydrodynamic} proved the following bounds:
\begin{equation}\label{eq:kappa-bound}
\max\left\{4,\frac{r^2 + 2(1+r)}{2r(1+r)} \right\} \leq \kappa(r) \leq 3 + \frac{1}{r}.
\end{equation}
It is an open problem to determine $\kappa(r)$ for $r < 1$.  In particular, one is interested in whether $\kappa(r) > 4$ for all $r <1$, or if there are some values of $r$ close to $r=1$ for which the inverse current $\kappa(r)$ remains unchanged.  

Even though we do not expect our continuum limit Hamilton-Jacobi equation to hold for the slow bond rate problem, it is nevertheless interesting to see what our results would say about this open problem were they to hold.  It is easy to see that $U_{\mu,\sigma}(1,1) = 4/r$ for $\mu=\sigma$ given by \eqref{eq:diag-source}.  Indeed, one can see that the optimal curve in the variational problem \eqref{eq:vardefU} must lie on the diagonal $\{x_1=x_2\}$, which gives the energy $4/r$.  This would suggest that
\[\kappa(r) =   \lim_{N\to \infty} \frac{1}{N}L(N,N) = U_{\mu,\sigma}(1,1)=\frac{4}{r}.\]
Notice that this violates the bounds in \eqref{eq:kappa-bound}, which indicates that the Hamilton-Jacobi equation continuum limit (Theorem \ref{thm:main}) does \emph{not} hold for sources along diagonal lines.  

It has recently come to our attention that the slow bond rate problem has been setteled by Basu, Sidoravicius, and Sly~\cite{basu2014last}. They show that the inverse current is \emph{always} affected when $r<1$, but do not give an explicit formula for $\kappa(r)$.

\section{Discussion and future work}
\label{sec:discussion}

In this work, we identified a Hamilton-Jacobi equation for the continuum limit of a macroscopic two-sided directed last passage percolation (DLPP) problem.  We rigorously proved the continuum limit when the macroscopic rates are discontinuous.  Furthermore, we presented a numerical scheme for solving the Hamilton-Jacobi equation, and an algorithm for finding optimal curves based on a dynamic programming principle.  Below we make some remarks, discuss simple extensions of this work, and ideas for future work. 
\begin{itemize}
  \item \textbf{Regularity of $\mu,\sigma$:} There are many simple modifications of (F1) under which one can prove Theorem \ref{thm:main}.    For example, the existence of the set $\Omega$ bounded by the strictly decreasing curve $\Gamma$ and $\partial \R^2_+$ on which $\mu=\sigma=0$ is not necessary, and one can check that the proofs hold without this assumption.  This would correspond to a TASEP model with step initial condition.  The curves $\Gamma_i$ on which $\mu$ and $\sigma$ may admit discontinuities can all be chosen to be strictly decreasing instead of increasing, with appropriate modifications in the proofs.  In fact, we can even allow the  curves to switch from strictly increasing to strictly decreasing, provided the critical point is isolated, and we make an additional cone condition assumption at this point. However, the curves $\Gamma_i$ cannot have any positive measure flat regions, as this can induce discontinuities in $U$, as shown in Remark \ref{rem:discont}. 
  \item \textbf{Discontinuous viscosity solutions:}  The regularity assumption (F1) was chosen to ensure that $U$ is locally uniformly continuous.  This is essential for invoking the  Arzel\`a-Ascoli Theorem in the proof of Theorem \ref{thm:perturbation}, and in the proof of the comparison principle for (P) (Theorem \ref{thm:comp-trunc}).  We believe that Theorem \ref{thm:main} holds under far more general assumptions on $\mu$, allowing $U$ to be discontinuous.  Presently, we do not know how to prove this.  The largest obstacle seems to be proving uniqueness of viscosity solutions of (P) when the solutions $U$ and the macroscopic weights $\mu$ are discontinuous. Our numerical results seem to support this conjecture, as the numerical scheme is able to very accurately capture discontinuities in $U$.  
  \item \textbf{Hydrodynamic limit of TASEP:} As we showed in Section \ref{sec:formal-eq},  the Hamilton-Jacobi equation (P) is formally equivalent to the conservation law governing the hydrodynamic limit of TASEP~\cite{georgiou2010,seppalainen1996hydrodynamic}.  It would be very interesting to make this connection rigorous.
  \item \textbf{Higher dimensions:} The main obstacle in generalizing the Hamilton-Jacobi equation (P), and the results in this paper, to dimensions $d\geq 3$, is the fact that the exact form of the time constant \eqref{eq:time-const} for  \iid~random variables $X(i,j)$ is unknown.  If an exact form for the time constant $U$ were to be discovered for $d\geq 3$, then we anticipate no problems in generalizing the results in this paper to higher dimensions. We should note that although the exact form of $U$ is unknown for $d\geq 3$, it is known that $U$ is continuous, 1-homogeneous, symmetric in all variables, and superadditive, under fairly broad assumptions on the distribution of $X(i,j)$~\cite{martin2004limiting}.  This is enough to show that $U$ is the viscosity solution of some Hamilton-Jacobi equation, but the explicit form of the equation is unknown.
\end{itemize}

\begin{acknowledgements}
The author would like to thank Jinho Baik for suggesting the problem and for stimulating discussions. The author would also like to thank the anonymous referee whose comments and suggestions have greatly improved this manuscript.
\end{acknowledgements}

\appendix

\section{Proof of Theorem \ref{thm:comp-trunc}}
For completeness we give the proof of Theorem \ref{thm:comp-trunc} here.  The proof is similar to \cite[Theorem 2.8]{calder2014}.
\begin{proof}
Suppose that
\[\lambda := \sup_{\R^d_+} (u-v) > 0.\]
Let
\begin{equation}\label{eq:R-def}
R = \sup \left\{ r > 0 \, : \, u \leq v + \frac{\lambda}{2}  \ \text{ on }  \ D_r\right\},
\end{equation}
where
\begin{equation}\label{eq:Dr-def}
D_r = \{x \in \R^d_+ \, : \, x_1 +  \cdots + x_d < r\}.
\end{equation}
Since $\O \in \R^d_+$, we have by hypothesis that  $u \leq v$ on $\partial \R^d_+$.  Therefore, since $u$ and $v$ are continuous we have $R \in (0,\infty)$. By \eqref{eq:R-def} there exists $\xi_0 \in \R^d_+\cap \partial D_R$ such that 
\begin{align}\label{eq:z_0-prop}
&u(\xi_0)=v(\xi_0) + \frac{\lambda}{2} \ \  \text{and} \notag \\
&\text{every neighborhood of } \xi_0 \text{ contains some } y \in \R^d_+ \text{ with } u(y) > v(y) + \frac{\lambda}{2}.
\end{align}
For $t>0$ set $\xi = \xi_0 + (t,\dots,t)$ and
\begin{equation}\label{eq:G-def}
\G = \{x \in [0,\infty)^d \, : \, x \leqq \xi \}.
\end{equation}
Let $u^\xi$ denote the $\xi$-truncation of $u$.  By \eqref{eq:z_0-prop} and \eqref{eq:R-def} we see that
\begin{equation}\label{eq:delta-def}
\delta := \sup_{\R^d_+} (u^\xi - v) > \frac{\lambda}{2}> 0
\end{equation}
By \eqref{eq:z_0-prop} we have $u(\xi_0) >v(\xi_0)$, and hence $\xi_0 \in \O$.  Let $\eps_{\xi_0}$ and $\v_{\xi_0} \in \S^1$ be as given in (H3)${}_\O$.   Choose $t>0$ small enough, and 
$\eps_{\xi_0}>0$ smaller if necessary, so that $\G\setminus D_R \in B_{\eps_{\xi_0}}(\xi_0) \subset \R^d_+$.  For $\alpha > 0$ define
\begin{equation}\label{eq:aux-function}
\Phi_\alpha(x,y) = u^\xi(x) - v(y) - \frac{\alpha}{2}\left| x - y - \frac{1}{\sqrt{\alpha}} \v_{\xi_0}\right|^2.
\end{equation}

We claim that for $\alpha$ large enough, there exists $x_\alpha,y_\alpha \in B_{\eps_{\xi_0}}(\xi_0)$ such that 
\begin{equation}\label{eq:Malpha}
M_\alpha := \sup_{\R^d_+ \times \R^d_+} \Phi_\alpha = \Phi_\alpha(x_\alpha,y_\alpha).
\end{equation}
To see this, first substitute $y=x-\frac{1}{\sqrt{\alpha}}\v_{\xi_0}$ into \eqref{eq:Malpha} to find
\[M_\alpha \geq u^\xi(x) - v\left(x-\frac{1}{\sqrt{\alpha}}\v_{\xi_0}\right),\]
for any $x \in \R^d_+$ such that $x-\frac{1}{\sqrt{\alpha}} \in \R^d_+$.
Since $u^\xi$ and $v$ are continuous, it follows from \eqref{eq:delta-def} that
\begin{equation}\label{eq:Mapos}
\liminf_{\alpha \to \infty} M_\alpha \geq \sup_{\R^d_+} (u^\xi-v) = \delta >  \frac{\lambda}{2} > 0.
\end{equation}
Since $u^\xi$ is bounded, and $v$ is monotone, we have by \eqref{eq:aux-function}  that
\begin{equation}\label{eq:xy-close}
|x-y| \leq \frac{C}{\sqrt{\alpha}} \ \ \text{whenever } \Phi_{\alpha}(x,y) \geq 0.
\end{equation}
Let $x,y \in \R^d_+$ such that $\Phi_\alpha(x,y) \geq 0$.  
Set $w = \pi_x(y) = \pi_y(x)$ and $\hat{w} = \pi_\xi(w)$, and define
\begin{equation}\label{eq:xyhat}
\hat{x} = x + \hat{w} - w \ \ \text{and} \ \ \hat{y} = y + \hat{w} -w.
\end{equation}
A short calculation shows that $\pi_\xi(x)=\pi_\xi(\hat{x})$.  
Since $u^\xi = u \circ \pi_\xi$ we have
\begin{equation}\label{eq:uz}
u^\xi(\hat{x}) = u^\xi(\pi_\xi(\hat{x})) = u^\xi(\pi_\xi(x)) = u^\xi(x).
\end{equation}
Since $v$ is Pareto-monotone and $\hat{y} \leqq y$ we have by \eqref{eq:uz} that
\begin{equation}\label{eq:1part}
u^\xi(\hat{x}) - v(\hat{y}) \geq u^\xi(x) - v(y)
\end{equation}
Since $\hat{x}-\hat{y} = x-y$, we see from \eqref{eq:1part} and \eqref{eq:aux-function} that
\begin{equation}\label{eq:xyhat-better}
\Phi_\alpha(\hat{x},\hat{y}) \geq \Phi_\alpha(x,y).
\end{equation}
Furthermore, by \eqref{eq:xy-close} we have
\[|\hat{x}-\hat{w}| = |x-w| \leq |x-y| \leq \frac{C}{\sqrt{\alpha}}.\]
Similarly we have $|\hat{y} - \hat{w}| \leq \frac{C}{\sqrt{\alpha}}$.  Since $\hat{w} \leqq \xi$ we have
\[\hat{x},\hat{y} \in \G_\alpha :=\left\{x' \in [0,\infty)^d \, : \, x' \leqq \xi + \frac{C}{\sqrt{\alpha}}(1,\dots,1)\right\}.\]
It follows from this and \eqref{eq:xyhat-better} that for every $\alpha>0$ there exists $x_\alpha,y_\alpha \in \G_\alpha$ such that $M_\alpha=\Phi_\alpha(x_\alpha,y_\alpha)$.  By \eqref{eq:xy-close} we may pass to a subsequence if necessary to find $x_0 \in \G$ such that $x_\alpha,y_\alpha \to x_0$ as $\alpha \to \infty$.  Then we have
\[\limsup_{\alpha \to \infty} M_\alpha \leq \lim_{\alpha \to \infty}u^\xi(x_\alpha) - v(y_\alpha)\leq \delta.\]
Combining this with \eqref{eq:Mapos} we have $M_\alpha \to \delta=u^\xi(x_0) - v(x_0)$ and 
\begin{equation}\label{eq:ptozero}
\frac{\alpha}{2} \left| x_\alpha - y_\alpha - \frac{1}{\sqrt{\alpha}} \v_{\xi_0}\right|^2 \longrightarrow 0.
\end{equation}
Since $\delta > \lambda/2$, it follows from the definition of $R$ \eqref{eq:R-def} that $x_0 \in \G\setminus D_R\subset B_{\eps_{\xi_0}}(\xi_0)$.  Therefore, for $\alpha>0$ large enough we have $x_\alpha,y_\alpha \in B_{\eps_{\xi_0}}(\xi_0)$, which establishes the claim.

Letting $p = \alpha\left(x_\alpha-y_\alpha - \frac{1}{\sqrt{\alpha}}\right)$ we have by \eqref{eq:Malpha} that $p \in D^+u^\xi(x_\alpha) \cap D^- v(y_\alpha)$.  By hypothesis we have
\begin{equation}\label{eq:super-solution2}
H^*(y_\alpha,p) \geq a.
\end{equation}
Since $u$ is truncatable, $u^\xi$ is a viscosity solution of \eqref{eq:u-sub} and therefore
\begin{equation}\label{eq:sub-solution2}
H_*(x_\alpha,p) \leq 0.
\end{equation}
Subtracting \eqref{eq:sub-solution2} from \eqref{eq:super-solution2} we have
\begin{equation}\label{eq:to-contradict}
a \leq H^*(y_\alpha,p) - H_*(x_\alpha,p).
\end{equation}
Let $w_\alpha = x_\alpha - y_\alpha - \frac{1}{\sqrt{\alpha}}\v_{\xi_0}$ and note that
\[x_\alpha = y_\alpha + \eps \v,\]
where
\[\eps = \frac{1}{\sqrt{\alpha}} |\v_{\xi_0} + \sqrt{\alpha}w_\alpha| = |x_\alpha-y_\alpha| \ \ \text{ and } \ \ \v = \frac{\v_{\xi_0} + \sqrt{\alpha}w_\alpha}{|\v_{\xi_0} + \sqrt{\alpha}w_\alpha|}.\]
By \eqref{eq:ptozero} we have $\sqrt{\alpha}w_\alpha \to 0$.  Therefore, for $\alpha$ large enough we have $|\v_{\xi_0} - \v| < \eps_{\xi_0}$ and $\eps < \eps_{\xi_0}$.  Since $y_\alpha \in B_{\eps_{\xi_0}}(\xi_0)$ we can invoke (H3)${}_\O$ to find that
\begin{equation}\label{eq:regularity}
H^*(y_\alpha,p) - H_*(x_\alpha,p) = H^*(y_\alpha,p) - H_*(y_\alpha + \eps \v,p) \leq m(|p||x_\alpha-y_\alpha| + |x_\alpha-y_\alpha|).
\end{equation}
Note that
\begin{align*}
|p| |x_\alpha - y_\alpha| &{}={} \alpha\left|x_\alpha - y_\alpha - \frac{1}{\sqrt{\alpha}} \v_{\xi_0}\right||x_\alpha-y_\alpha| \\
&{}={}\alpha\left|x_\alpha - y_\alpha - \frac{1}{\sqrt{\alpha}} \v_{\xi_0}\right|\left|x_\alpha - y_\alpha - \frac{1}{\sqrt{\alpha}} \v_{\xi_0} + \frac{1}{\sqrt{\alpha}}\v_{\xi_0}\right|\\
&{}\leq{} \alpha w_\alpha^2 + \sqrt{\alpha} w_\alpha.
\end{align*}
Combining this with \eqref{eq:regularity} and \eqref{eq:to-contradict} we have
\[0 < a \leq m(\alpha w_\alpha^2 + \sqrt{\alpha} w_\alpha + |x_\alpha-y_\alpha|).\]
Sending $\alpha \to \infty$ yields a contradiction.
\end{proof}


\end{document}